\documentclass[10pt,a4paper,twoside]{article}
\usepackage{epsfig,amssymb,amsmath,amsthm,mathrsfs,diagbox}
\usepackage[dvipsnames]{color}
\definecolor{purpleheart}{rgb}{.41,.21,.61}  
\definecolor{bblue}{rgb}{.18,.35,.58}  
\usepackage[colorlinks]{hyperref}
\hypersetup{linkcolor = bblue, citecolor=purpleheart}
\def\YEAR{2019}\newcount\VOL\VOL=\YEAR\advance\VOL by-1995
\def\firstpage{1}\def\lastpage{1000}
\def\received{}\def\revised{}
\def\communicated{}

\makeatletter
\def\magnification{\afterassignment\m@g\count@}
\def\m@g{\mag=\count@\hsize6.5truein\vsize8.9truein\dimen\footins8truein}
\makeatother

\font\eightrm=cmr8
\font\caps=cmcsc10                    
\font\Caps=cmcsc10 scaled \magstep1   
\font\scaps=cmcsc8

\makeatletter
\@specialpagefalse\headheight=8.5pt
\def\DocMath{{\def\th{\thinspace}\scaps }}
\renewcommand{\@oddfoot}{}
\renewcommand{\@evenfoot}{}

\renewcommand{\@evenhead}{%
    \ifnum\thepage>\lastpage\rlap{\thepage}\hfill%
    \else\rlap{\thepage}\slshape\leftmark\hfill{\caps\SAuthor}\hfill\fi}%
\renewcommand{\@oddhead}{%
    \ifnum\thepage=\firstpage{\DocMath\hfill\llap{\thepage}}%
    \else{\slshape\rightmark}\hfill{\caps\STitle}\hfill\llap{\thepage}\fi}%
\makeatother

\def\TSkip{\bigskip}
\newbox\TheTitle{\obeylines\gdef\GetTitle #1
\ShortTitle  #2
\SubTitle    #3
\Author      #4
\ShortAuthor #5
\EndTitle
{\setbox\TheTitle=\vbox{\baselineskip=20pt\let\par=\cr\obeylines%
\halign{\centerline{\Caps##}\cr\noalign{\medskip}\cr#1\cr}}%
	\copy\TheTitle\TSkip\TSkip%
\def\next{#2}\ifx\next\empty\gdef\STitle{#1}\else\gdef\STitle{#2}\fi%
\def\next{#3}\ifx\next\empty%
    \else\setbox\TheTitle=\vbox{\baselineskip=20pt\let\par=\cr\obeylines%
    \halign{\centerline{\caps##} #3\cr}}\copy\TheTitle\TSkip\TSkip\fi%
\centerline{\caps #4}\TSkip\TSkip%
\def\next{#5}\ifx\next\empty\gdef\SAuthor{#4}\else\gdef\SAuthor{#5}\fi%
\ifx\received\empty\relax
    \else\centerline{\eightrm \received}\fi%
\ifx\revised\empty\TSkip%
    \else\centerline{\eightrm  \revised}\TSkip\fi%
\ifx\communicated\empty\relax
    \else\centerline{\eightrm \communicated}\fi\TSkip\TSkip%
\catcode'015=5}}\def\Title{\obeylines\GetTitle}
\def\Abstract{\begingroup\narrower
    \parskip=\medskipamount\parindent=0pt{\caps Abstract. }}
\def\EndAbstract{\par\endgroup\TSkip}

\long\def\MSC#1\EndMSC{\def\arg{#1}\ifx\arg\empty\relax\else
     {\par\narrower\noindent%
     2010 Mathematics Subject Classification: #1\par}\fi}

\long\def\KEY#1\EndKEY{\def\arg{#1}\ifx\arg\empty\relax\else
	{\par\narrower\noindent Keywords and Phrases: #1\par}\fi\TSkip}

\newbox\TheAdd\def\Addresses{\vfill\copy\TheAdd\vfill
    \ifodd\number\lastpage\vfill\eject\phantom{.}\vfill\eject\fi}
{\obeylines\gdef\GetAddress #1
\Address #2
\Address #3
\Address #4
\EndAddress
{\def\xs{5.8truecm}\parindent=0pt
\setbox0=\vtop{{\obeylines\hsize=\xs#1\par}}\def\next{#2}
\ifx\next\empty 
     \setbox\TheAdd=\hbox to\hsize{\hfill\copy0\hfill}
\else\setbox1=\vtop{{\obeylines\hsize=\xs#2\par}}\def\next{#3}
\ifx\next\empty 
     \setbox\TheAdd=\hbox to\hsize{\hfill\copy0\hfill\copy1\hfill}
\else\setbox2=\vtop{{\obeylines\hsize=\xs#3\par}}\def\next{#4}
\ifx\next\empty\ 
     \setbox\TheAdd=\vtop{\hbox to\hsize{\hfill\copy0\hfill\copy1\hfill}
                \vskip20pt\hbox to\hsize{\hfill\copy2\hfill}}
\else\setbox3=\vtop{{\obeylines\hsize=\xs#4\par}}
     \setbox\TheAdd=\vtop{\hbox to\hsize{\hfill\copy0\hfill\copy1\hfill}
	        \vskip20pt\hbox to\hsize{\hfill\copy2\hfill\copy3\hfill}}
\fi\fi\fi\catcode'015=5}}\gdef\Address{\obeylines\GetAddress}

\def\LOCAL{\jobname.files}

\newtheorem{theorem}{Theorem}[section]
\newtheorem{corollary}[theorem]{Corollary}

\newtheorem{lemma}[theorem]{Lemma}
\newtheorem{proposition}[theorem]{Proposition}

\numberwithin{equation}{section}

\newcommand{\A}{\mathbf{A}}
\newcommand{\Af}{\mathbf{A}_{\fin}}

\newcommand{\bs} {\backslash}
\newcommand{\C}{\mathbf{C}}
\newcommand{\cH}{\mathcal{H}}
\newcommand{\comment}[1]{}

\newcommand{\ds}{\displaystyle}
\newcommand{\e}{\varepsilon}
\newcommand{\fin}{\operatorname{fin}}
\newcommand{\ff}{f_{\fin}}
\newcommand{\Ghat}{\widehat{\olG}}
\newcommand{\GL}{\operatorname{GL}}
\newcommand{\GSp}{\operatorname{GSp}}
\newcommand{\Hom}{\operatorname{Hom}}
\newcommand{\Hhat}{\widehat{H}}

\renewcommand{\k}{\mathtt{k}}

\newcommand{\lie}{\mathfrak{g}}
\newcommand{\limN}{\lim_{N\to\infty\atop{(N,\Sb)=1}}}
\newcommand{\lu}{\underline{\lambda}}
\newcommand{\Lie}{\operatorname{Lie}}
\newcommand{\mat}[4]{\begin{pmatrix} {#1} & {#2} \\ {#3} & {#4} \end{pmatrix}}
\newcommand{\meas}{\operatorname{meas}}
\newcommand{\ol}{\overline}
\newcommand{\olG}{\mathbb{G}}
\newcommand{\olT}{\mathbb{T}}

\newcommand{\ord}{\operatorname{ord}}
\newcommand{\PGSp}{\operatorname{PGSp}}
\newcommand{\PGL}{\operatorname{PGL}}
\newcommand{\Q}{\mathbf{Q}}
\newcommand{\Qf}{\mathcal{Q}}
\newcommand{\R}{\mathbf{R}}
\newcommand{\sg}[1]{\left<{#1}\right>}
\newcommand{\smat}[4]{\bigl(\begin{smallmatrix}{#1}&{#2}\\{#3}&{#4}\end{smallmatrix}\bigr )}
\newcommand{\Sb}{\mathbf{S}}
\newcommand{\SL}{\operatorname{SL}}
\newcommand{\SO}{\operatorname{SO}}
\newcommand{\Sp}{\operatorname{Sp}}
\newcommand{\Spin}{\operatorname{Spin}}
\newcommand{\SU}{\operatorname{SU}}
\newcommand{\supp}{\operatorname{supp}}
\newcommand{\tf}{\tilde{f}}
\newcommand{\tA}{{}^t\!A}
\newcommand{\Tc}{\That_c}
\newcommand{\That}{\widehat{\olT}}
\newcommand{\tr}{\operatorname{tr}}

\newcommand{\w}{\omega}
\newcommand{\X}{\mathfrak{X}_{\Sb}}
\newcommand{\Z}{\mathbf{Z}}
\newcommand{\Zhat}{\widehat{\Z}}
\newcommand{\ignore}[1]{}

\renewcommand{\subset}{\subseteq}
\newcommand{\diag}{\operatorname{diag}}

\newcommand{\h}{\mathfrak{h}}
\newcommand{\Ind}{\operatorname{Ind}}
\newcommand{\klie}{\mathfrak{k}}
\newcommand{\g}{\gamma}
\newcommand{\plie}{\mathfrak{p}}

\newcommand{\sgn}{\operatorname{sgn}}

\renewcommand{\t}{\,{}^t\!}
\newcommand{\U}{\operatorname{U}}

\comment{
\makeatletter
\def\@tocline#1#2#3#4#5#6#7{\relax
  \ifnum #1>\c@tocdepth 
  \else
    \par \addpenalty\@secpenalty\addvspace{#2}%
    \begingroup \hyphenpenalty\@M
    \@ifempty{#4}{%
      \@tempdima\csname r@tocindent\number#1\endcsname\relax
    }{%
      \@tempdima#4\relax
    }%
    \parindent\z@ \leftskip#3\relax \advance\leftskip\@tempdima\relax
    \rightskip\@pnumwidth plus4em \parfillskip-\@pnumwidth
    #5\leavevmode\hskip-\@tempdima
      \ifcase #1
       \or\or \hskip 1em \or \hskip 2em \else \hskip 3em \fi%
      #6\nobreak\relax
    \dotfill\hbox to\@pnumwidth{\@tocpagenum{#7}}\par
    \nobreak
    \endgroup
  \fi}
\makeatother
}

\begin{document}
\Title On the Distribution of Satake Parameters for Siegel Modular Forms
\ShortTitle Distribution of Siegel Modular Satake Parameters
\SubTitle
\Author Andrew Knightly and Charles Li
\ShortAuthor
\EndTitle
\Abstract
We prove a harmonically weighted equidistribution result for the $p$-th Satake parameters of the
  family of automorphic cuspidal representations of $\PGSp(2n)$ of fixed weight $\k$
 and prime-to-$p$ level $N\to \infty$.  The main tool is a new asymptotic Petersson
 formula for $\PGSp(2n)$ in the level aspect.
 \EndAbstract
\MSC
11F46 (11F25, 11F30, 11F70, 11F72)
\EndMSC
\KEY
Siegel modular forms, Petersson formula, Satake parameters
\EndKEY
\Address
Andrew Knightly
Department of Mathematics
\ \& Statistics
University of Maine
Neville Hall
Orono, ME  04469-5752
USA
knightly@math.umaine.edu
\Address
Charles Li
Department of Mathematics
The Chinese University of Hong Kong
Shatin
Hong Kong
charlesli@cuhk.edu.hk
\Address
\Address
\EndAddress

\setcounter{tocdepth}{2}

\tableofcontents

\setlength{\parskip}{.1cm}
\section{Introduction}

For a split reductive algebraic group $G$ over a number field $F$,
let $\mathcal{A}(G)$ denote the set of cuspidal automorphic representations of $G(\A_F)$.
Each element of  $\mathcal{A}(G)$ factorizes as a restricted tensor product
  $\pi=\otimes_v \pi_v$
of irreducible representations of the local groups $G(F_v)$.
    If $v$ is a nonarchimedean place of $F$, then the unramified irreducible representations
  of $G(F_v)$ are parametrized (via the Satake isomorphism) by the semisimple conjugacy classes in
  the complex dual group $\widehat{G}=\widehat{G}(\C)$.
   When  $\pi_v$ is unramified, we let
\[ t_{\pi_v}\in \widehat{T}/W\]
  denote the associated Satake parameter.
   Here, $\widehat{T}=\widehat{T}(\C)$ is a split maximal torus of $\widehat{G}$ and
   $W=N_G(T)/T$ is the Weyl group of $G$.

It is of great interest to understand the distribution of the points $t_{\pi_v}$, possibly
  with weights, as $\pi$ and/or $v$ vary.
  If $\pi$ is fixed and $v$ varies, then according to the general Sato-Tate
  conjecture, the points are expected to be
  equidistributed relative to some naturally defined measure
  $\mu_{ST}$ with support in the maximal compact subgroup of $\widehat{T}$.
The support condition reflects the presumed predominance of
  representations satisfying the Ramanujan conjecture.

  Many people have considered the analogous ``vertical" question of fixing $v$ and
  varying $\pi$ in a family, starting with Birch's 1968 investigation of the
  problem for elliptic curves over $\Z/p\Z$, \cite{Bi}.  A table summarizing
  the numerous subsequent vertical equidistribution results for $\GL(2)$ families
  is given in \cite{KR}.
  Extending work of Serre and others on the case of classical holomorphic cusp forms,
  Shin proved generally that if $G$ admits discrete series over $\R$, then
when the $t_{\pi_v}$ are unweighted, the relevant measure
  (for many natural families) is the
  Plancherel measure at $v$, \cite{Sh}.
  Shin and Templier also obtained a quantitative version of this result
  with error bounds, \cite{ST}.  Applications include: (i) a
  diagonal hybrid where the size of the
  family and the place $v$ both tend to infinity (the relevant measure being
  Sato-Tate rather than Plancherel), and (ii) a determination of the distribution of
  the low-lying zeroes of certain families of automorphic $L$-functions for $G$.
  Matz and Templier have recently treated
  the case of $\GL(n)$, \cite{MT}.
  We refer to Sarnak, Shin and Templier \cite{SST} for a precise
  formulation of various Sato-Tate problems and related topics.

  When the Satake parameters at a fixed place $v$ are given the harmonic weights
  that arise naturally in the Petersson/Kuznetsov
  trace formula, it has been shown for many $\GL(2)$ and $\GL(3)$ families
  that they exhibit equidistribution relative to the Sato-Tate measure itself
  rather than the Plancherel measure (\cite{Br}, \cite{Li}, \cite{ftf}, \cite{BM},
  \cite{BBR}, \cite{Z}).
In this paper we consider the distribution of harmonically weighted Satake
  parameters at a fixed place $p$
  for the group $G=\GSp(2n)$.  For simplicity we work over $\Q$ and assume trivial
  central character.  We consider cuspidal representations $\pi$ of level $N$ for
  which $\pi_\infty$ is a fixed holomorphic discrete series representation of
  weight $\k>2n$.  We weight each
  Satake parameter $t_{\pi_p}$ by the globally defined value
\begin{equation}\label{wpidef}
w_\pi =\sum_{\varphi\in E_{\k}(\pi)}\frac{|c_{\sigma}(\varphi)|^2}{\|\varphi\|^2},
\end{equation}
where $E_{\k}(\pi)$ is a finite orthogonal set of cuspidal Hecke eigenforms
  giving rise to $\pi$, and
  $c_{\sigma}(\varphi)$ denotes a Fourier coefficient, defined in \eqref{csigma}.
 Weighted in this way, we prove that the parameters become equidistributed
  relative to a certain probability measure $\mu$ as $N\to\infty$ (see Theorem \ref{equi}).
  In contrast to the $\GL(2)$ case, the measure depends on $p$.
Subject to a natural hypothesis on the
  growth of the geometric side of the trace formula (which holds at least when $n=2$),
 in Theorem \ref{musum} we relate $\mu$ to the Sato-Tate measure.
  In particular, the hypothesis implies that $\mu$ is supported on the tempered spectrum,
  and tends to the Sato-Tate measure as $p\to\infty$.

A very similar equidistribution problem has been studied already in the case
  $n=2$ by Kowalski, Saha, and Tsimerman
  (\cite{KST}), who fix the level $N=1$ and let the archimedean parameter
  $\k \to \infty$.  Using a formula of Sugano, they were able to
  form a connection between Satake parameters and Fourier coefficients, the latter
  of which they control with the intricate Petersson formula for $\GSp(4)$ due
  to Kitaoka \cite{Ki}.
  Their weights involve certain linear combinations of Fourier coefficients and
  depend on the choice of a class group character. In a special case they
  coincide with the ones given in \eqref{wpidef} with $\sigma=I_n$.
  Their methods have been adapted to treat the
  case of higher level $N\to\infty$ by M.\,Dickson \cite{D}.
  Kim, Wakatsuki and Yamauchi have also investigated the equidistribution
  problem for $\GSp(4)$, via Arthur's invariant trace formula \cite{KWY}.
  In all of these works, a quantitative equidistribution statement is proven,
  with application to the distribution of low-lying zeros of $L$-functions.

The key technical tools used in \cite{KST}, namely,
   Kitaoka's formula and Sugano's formula, are not yet available when $n>2$.
  Nevertheless, we can apply two simple ideas to treat the higher rank case.
  The first is to use a Hecke operator as a test function
  in the relative trace formula to derive a Petersson formula for
  $\GSp(2n)$ whose spectral side involves two Fourier coefficients
  (as usual) with the additional inclusion of a Satake parameter.
  In this way we can access the Satake parameters directly
  without the use of Sugano's formula.
  In order to project onto the holomorphic cusp forms of weight $\k$, we
   use a certain matrix coefficient of the weight $\k$ holomorphic discrete series
   $\pi_\k$ of $\GSp(2n,\R)$
  as the archimedean component of our test function.  This function is computed
  explicitly in the Appendix (Theorem \ref{mcthm}).

  The second idea is to take the limit of the kernel function before
  integrating.
 This allows us to avoid computing or estimating all but
  a few of the orbital integrals
  that show up on the geometric side.
  The result is an asymptotic Petersson formula
 with only a finite sum on the geometric side
   (Theorem \ref{A}).\footnote[2]{After this paper was written,
  we became aware of \cite{KST2}, in which a similar idea is applied to
  classical Siegel Poincar\'e series to obtain an analog of Theorem \ref{S=1}
  for $\k\to\infty$.}
In the simplest case where the Hecke operator is trivial, it is given as
  follows (see \S \ref{S1}, where notation is explained in detail).

\begin{theorem} \label{S=1}
 Let $\mathcal{B}_\k(N)$ be an orthogonal basis of the space of degree $n$
   Siegel cusp forms of weight $\k>2n$ with $2|n\k$, and level group
$\Gamma_0(N)=\{\smat ABCD\in \Sp_{2n}(\Z)|\, C\equiv O\mod N\}$.
For symmetric positive-definite half-integral matrices $\sigma_1,\sigma_2$,
\begin{align}\label{S1eq}\lim_{N\to\infty}&\frac1{\psi(N)}\sum_{F\in \mathcal{B}_\k(N)}
  \frac{a_{\sigma_1}(F) \ol{a_{\sigma_2}(F)}} {\|F\|^2}
=\delta_\k(\sigma_1,\sigma_2)c_{n\k\sigma_1},\end{align}
where $a_{\sigma_j}(F)$ are Fourier coefficients, $\psi(N)=[\Gamma_0(1):\Gamma_0(N)]$,
\[c_{n\k\sigma_1}= \frac{(\det\sigma_1)^{\k-(n+1)/2}}
{\pi^{n(n-1)/4}(4\pi)^{n(n+1)/2-n\k}\prod_{j=1}^n\Gamma(\k-\frac{n+j}2)},\]
   and
\begin{equation}\label{deltak}
\delta_\k(\sigma_1,\sigma_2)=
  \sum_{A\in \GL_n(\Z)/\{\pm I_n\}\atop \t A\sigma_1A=\sigma_2}\det(A)^{\k}.
\end{equation}
\end{theorem}
\noindent{\em Remarks:}

 (1) When $\k$ is even,
$\delta_\k(\sigma_1,\sigma_2)=\delta_0(\sigma_1,\sigma_2)=
\#\{A\in\GL_n(\Z)/\{\pm I_n\}|\, \t A\sigma_1A=\sigma_2\}$.  When $\k$ is odd, it may
  happen that $\delta_\k(\sigma_1,\sigma_2)=0$ even when $\delta_0(\sigma_1,\sigma_2)>0$
  (for example, if $\sigma_1,\sigma_2$ are diagonal).  In such cases, the spectral side
  also vanishes, since $a_{\sigma_j}(F)=0$ by \cite[p.\ 45]{Kl}.
  We note that $\delta_\k(\sigma_1,\sigma_2)$ does not always vanish when $\k$ is odd.

(2) When $n=2$ and $\k$ is even, the above result is shown in \cite{CKM}, \cite{D}.
(In \cite[Theorem 1.1, Remark 1.4]{CKM}, it is incorrectly asserted that \eqref{S1eq} holds when $n=2$ and
  $\k$ is odd, but with $\delta_0(\sigma_1,\sigma_2)$ in place of $\delta_\k(\sigma_1,
  \sigma_2)$.  This is incompatible with the observations in our first remark.)

(3) $c_{n\k\sigma_1}$
  is the constant of proportionality between $a_{\sigma_1}(F)$
  and a suitably normalized inner product of $F$ with the
  $\sigma_1$-Poincar\'e series of weight $\k$ (\cite{M}, \cite[p.\ 90]{Kl}, \cite[Lemma 6.2]{D}).

(4) Although we have highlighted the above special case, the main focus of this paper
  is on local Satake parameters, which are absent in the above theorem.\\

Petersson/Kuznetsov trace formulas play a fundamental role in the study
  of automorphic $L$-functions. There are well-established methods for $\GL(2)$,
  and to a lesser extent $\GL(3)$, but applications to other groups are rare.
  For a recent example, Blomer has used Kitaoka's formula to compute first
  and second moments of
  spinor $L$-functions of $\GSp(4)$ of full level and large weight,
  with power-saving error term, \cite{Bl}.  Waibel subsequently treated the case
  of large prime level, \cite{W}.  This requires a close analysis of the generalized
  Kloosterman sums appearing in the off-diagonal terms.
  It would be of great interest to extend Kitaoka's formula from degree 2 to degree $n$.
  The machinery we develop here can form the starting point for such
  a generalization.
  Although at present there is no quantification of the error term for finite $N$ if
  $n>2$ (the $n=2$ case is treated in Appendix \ref{AB}),
  the asymptotic formula is sufficient for obtaining the
  equidistribution result.

A natural question is whether the same method can be used to obtain
    the weighted equidistribution of Satake parameters in the case of $\k\to\infty$.
    Although we certainly expect such a result to hold, it does not seem possible to show it
    by the above method without a more detailed consideration of the off-diagonal terms.
    This is discussed in a remark after the proof of Proposition \ref{Ilim}.

We close this introduction with some additional comments comparing our results and methods
  to those in \cite{KST}.  If $F$ is a Siegel cusp form of degree $n=2$ which is
  unramified at $p$, it determines an unramified local representation $\pi_p$ of $\GSp(4)$.
  Sugano's formula expresses the value of a Laurent polynomial $U_p^{l,m}$ evaluated at
  the Satake parameters of $\pi_p$ as a value of the spherical function in the Bessel model
  for $\pi_p$. The latter can in turn be
  expressed as a certain linear combination of Fourier coefficients of $F$,
  with coefficients involving a character of a class group.
  This is the source of the weights in the equidistribution result of \cite{KST}.
    This method has the advantage that the $U^{l,m}_p$ were shown by
    Furusawa and Shalika to be orthonormal with respect to a certain explicit local measure $\eta_p$
    on the space $\mathfrak{X}_p$ (defined near \eqref{tpi} below) containing the Satake parameters.
    The latter fact leads directly to the weighted equidistribution result, since a refinement
    of Kitaoka's formula has the main term $\delta(l,m)=\sg{U^{l,m}_p,U^{0,0}_p}
    =\int_{\mathfrak{X}_p} U^{l,m}_pd\eta_p$.
    By contrast, the measure $\mu$ in our Theorem \ref{equi} is not given explicitly, because
    we have not found an orthonormal class of functions for which the main term of the trace
    formula has such a simple form.

\vskip .2cm
\noindent{\bf Acknowledgements:}  We would like to thank the referee, whose
detailed comments led to significant improvements in the exposition.
  We thank Yuk-Kam Lau, Nigel Pitt, Abhishek Saha, and
Nicolas Templier for helpful discussions.
We are also grateful for financial support from the University of Maine Office of the
Vice President for Research and from the Simons Foundation. The first author would
  also like to thank the Department of Mathematics at the Chinese University of
  Hong Kong for its hospitality.

{This work was partially supported by a grant from the Simons Foundation
  (\#317659 to the first author).}

\section{The Satake transform}

  Here we recall some basic background about the Satake transform.
  References include \cite{Gr} and \cite{Sha}.
  For notation in this section, let $G$ be a split group defined over $\Q_p$, and
  let $T$ be a (split) maximal torus of $G$ defined over $\Q_p$ and
   contained in a Borel subgroup $B=TN$ with $N$ unipotent.
  Let $X^*(T)=\Hom(T,\GL_1)$ denote the lattice of algebraic characters of $T$, and
  $X_*(T)=\Hom(\GL_1,T)$ the cocharacter lattice.

  For each prime $p$, let $K_p=G(\Z_p)$, which is a maximal
  compact subgroup of $G_p=G(\Q_p)$.
   Let $\mathcal{H}(G_p,K_p)$ be the Hecke algebra of
  locally constant compactly supported complex-valued bi-$K_p$-invariant functions
  on $G_p$.
 The Satake transform of an element $f \in \mathcal{H}(G_p,K_p)$
 is the function on $T(\Z_p) \bs T(\Q_p)$ given by
 \begin{equation}\label{Sf}
 \mathcal Sf(t) = \delta(t)^{1/2} \int_{N(\Q_p)} f(tn) dn.
\end{equation}
Here $\delta(t)=\bigl|\det \operatorname{Ad}(t)|_{\Lie(N_p)}\bigr|$ is the modular function for
   $B_p$.\footnote[3]{Later on we will take $G=\GSp(2n)$ and $B$ the
  Borel subgroup determined by the set of positive roots chosen
  in \S \ref{root}.  Then $\delta(t)^{1/2}=p^{-\sg{\lambda,\rho}}$
if $t=\lambda(p)$ for $\lambda\in X_*(T)$ and $\rho$ is given by \eqref{rho}.}
By way of motivation for \eqref{Sf}, suppose
\[\pi_\chi=\Ind_{B_p}^{G_p}(\chi \delta^{1/2})\]
 is an
  unramified representation of $G_p$, with nonzero $K_p$-invariant vector $\phi$.
  Then for $f\in \mathcal{H}(G_p,K_p)$,
 $\phi$ is an eigenfunction of the operator $\pi_\chi(f)$, with eigenvalue
\begin{equation}\label{wchif}
\w_\chi(f)=\int_{T(\Q_p)}\mathcal{S}f(t)\chi(t)dt.
\end{equation}

The Satake transform is a $\C$-algebra isomorphism
\begin{equation}\label{S}
\mathcal{S}: \mathcal{H}(G_p, K_p) \longrightarrow \mathcal{H}(T(\Q_p),T(\Z_p))^W,
\end{equation}
where the latter denotes the elements
  which are fixed by the Weyl group $W$.
Let $\widehat{T}=\Hom(X_*(T),\C)$ be the dual group of $T$. It satisfies
\[X^*(T)\cong X_*(\widehat{T}), \quad X_*(T)\cong X^*(\widehat{T}).\]
Using the algebra isomorphisms
\begin{equation}\label{3cong}
\mathcal{H}(T(\Q_p),T(\Z_p))^W  \cong
  \C[X_*(T)]^W\cong \C[X^*(\widehat T)]^W,
\end{equation}
 we may view $\mathcal{S}f$ as a function on
  any of these three spaces.  We explain this in some more detail.
 The first isomorphism in \eqref{3cong} arises by identifying an element
 of $T(\Q_p)/T(\Z_p)$ with a tuple of integer powers of $p$, which is
  of the form $\lambda(p)$ for a unique $\lambda\in X_*(T)$.
Thus if we write
\begin{equation} \label{Sf1}
\mathcal Sf = \sum_{t \in T(\Q_p) / T(\Z_p)}
  a_t C_{t}\in \mathcal{H}(T(\Q_p),T(\Z_p))
\end{equation}
where $a_t\in \C$ is nonzero for at most finitely many $t$,
  and $C_t$ is the characteristic function of the coset $t$,
  we can make the identification
\[\mathcal Sf
    = \sum_{\lambda \in X_*(T)} a_\lambda \lambda\in \C[X_*(T)],
\]
    where $a_\lambda = a_{\lambda(p)}$ from \eqref{Sf1}.
Fix an isomorphism $X_*(T)\cong X^*(\widehat T)$ and denote it by $\lambda\mapsto\hat{\lambda}$.
  Then we may in turn identify $\mathcal{S}f$ with the function
       \begin{equation} \label{Sf3}
\mathcal Sf = \sum_{\hat \lambda \in X^*(\widehat T)} a_{\hat{\lambda}} \hat{\lambda}
  \in \C[X^*(\widehat{T})],
       \end{equation}
    where $a_{\hat \lambda}=a_\lambda$.

An unramified character $\chi$ of $T(\Q_p)$ as in \eqref{wchif} can be identified
  with the Satake parameter $t_\chi\in \widehat{T}(\C)$ determined by
\begin{equation}\label{lambdahat}
\chi(\lambda(p))=\hat{\lambda}(t_\chi) \qquad\text{for all } \lambda\in X_*(T)
\end{equation}
(cf. \cite[p. 134, Eq. (3)]{Ca}).
  With this notation, using \eqref{Sf1} and \eqref{Sf3}, and taking $\meas(T(\Z_p))=1$,
  \eqref{wchif} becomes
\begin{equation}\label{wf}
\w_\chi(f)=\sum_\lambda a_\lambda \chi(\lambda(p))=\sum_{\hat{\lambda}}
  a_{\hat{\lambda}}\hat{\lambda}(t_\chi) = \mathcal{S}f(t_\chi).
\end{equation}
When $\pi=\pi_\chi$ is given, we write $t_\pi=t_\chi$  for the Satake parameter
  of $\pi$.

\begin{proposition}\label{Sfbar}
Viewing $\mathcal{S}f$ as a function on $\widehat{T}$ as in \eqref{Sf3}, we have
\[\overline{\mathcal{S}f}=\mathcal{S}f^*,\]
where $f^*(g)=\ol{f(g^{-1})}$.
\end{proposition}

\begin{proof}
For $t\in \widehat{T}$, we need to show that ${\mathcal{S}f(t)}=\overline{\mathcal{S}f^*(t)}$.
  As in \eqref{lambdahat},
there exists a unique unramified character $\chi$ of $T(\Q_p)$
  such that $\chi(\lambda(p))=\hat{\lambda}(t)$ for all $\lambda\in X_*(T)$.
   Recalling that $\pi_\chi(f)^*=\pi_\chi(f^*)$,
  by \eqref{wf} we have, for the spherical unit vector $\phi\in \pi_\chi$,
\[\mathcal{S}f(t)=\w_\chi(f)=\sg{\pi_\chi(f)\phi,\phi}=\sg{\phi,\pi_\chi(f^*)\phi}
  =\overline{\w_\chi(f^*)}=\ol{\mathcal{S}f^*(t)}.\qedhere\]
\end{proof}

\section{The symplectic group}

\subsection{Definition}
Henceforth, we will denote by $G$ the algebraic group $\GSp_{2n}$ defined
  as follows.
For any commutative ring $R$, let $M_{2n}(R)$ be the set of $2n \times 2n$ matrices
   with entries in $R$.  Letting $O$ denote the zero-matrix of suitable dimension,
  and $I_n$ denote the $n\times n$ identity matrix,
  define
\[
    J = \mat{O}{I_n}{-I_n}{O},
\]
\[
    \Sp_{2n}(R) = \{M \in M_{2n}(R) |\, \t MJM = J \},
\]
\[
    \GSp_{2n}(R) = \{M \in M_{2n}(R)|\, \t MJM = r(M)J \text{ and } \,r(M)
  \text{ is a unit in } R \}.
\]
Thus a matrix $M=\smat ABCD$ belongs to $\GSp_{2n}(R)$ if and only if there
  exists a similitude $r(M)\in R^*$ such that
\begin{equation} \label{ac}
   \t AC = \t\, CA, \quad \t BD = \t DB, \quad \t AD-{}^tCB = r(M) I_n.
\end{equation}
Taking inverses in ${}^tMJM=r(M)J$ shows that
  $\t M=\smat {\t A}{\t\, C}{\t B}{\t D}\in \GSp_{2n}(R)$, and the above relations
  applied to this matrix give
\begin{equation} \label{ac2}
   A\t B=B\t A,\quad C\t D=D\t\, C, \quad A\t D-B\t\, C=r(M)I_n.
\end{equation}
Note that $\Sp_{2n}(R) = \{M \in \GSp_{2n}(R) |\, r(M) = 1\}$.
Define
\[\olG=\PGSp_{2n}=\GSp_{2n}/Z,\]
  where $Z$ is the center (the set of scalar matrices).

   Let $S_n(R)$ be the set of $n \times n$ symmetric
   matrices over $R$.
  The Siegel upper half space is the following set of complex symmetric
  matrices
\begin{equation}\label{Hn}
    \cH_{n} = \{X+iY \in S_n(\C) |\, X,Y\in S_n(\R), Y > 0 \},
\end{equation}
where $Y > 0$ means that $Y$ is positive definite.
  It is a complex vector space of dimension $\tfrac{n(n+1)}2$.
Letting
\[\GSp_{2n}(\R)^+ =
   \{ M \in \GSp_{2n}(\R) |\, r(M) > 0 \},\]
  there is a transitive action of $\GSp_{2n}(\R)^+$ on $\cH_n$ given by
\[
    M\cdot\mathfrak{Z} = (A\mathfrak{Z}+B)(C\mathfrak{Z}+D)^{-1} \qquad(\mathfrak{Z}\in\cH_n).
\]
The stabilizer in $\Sp_{2n}(\R)$ of the element $iI_n\in \cH_n$ is the compact subgroup
\begin{equation}\label{K}
    K_\infty = \left\{\mat{A}{B}{-B}{A}\in \Sp_{2n}(\R) |\, A+iB \in \U(n) \right\},
\end{equation}
where $\U(n)$ is the group of $n\times n$ complex unitary matrices $X$ (so
  $X^{-1}={}^t\ol{X}$).
Define the Siegel parabolic subgroup
\begin{align}\notag P(\R)&=\left\{\mat ABOD\in \GSp_{2n}(\R)\right\}\\
\label{B}&=\left\{\mat A{}{}{r\, \t A^{-1}}
  \mat IS{}I|\, A\in \GL_n(\R), r\in \R^*, S\in S_n(\R)\right\}
\end{align}
(where $O$ denotes the $n\times n$ zero matrix), and set $\mathbb{P} =P/Z$.
We recall the decomposition
\[\GSp_{2n}(\R)=P(\R)K_\infty.\]

For a prime $p$, let
\[K_p=\PGSp_{2n}(\Z_p).\]
Haar measure on $\PGSp_{2n}(\Q_p)$ will be normalized so that $\meas(K_p)=1$.
For an integer $N>0$, set
\begin{equation}\label{K0N}
   K_0(N)_p = \left\{\mat{A}{B}{C}{D} \in K_p|\, C \equiv O \mod N\Z_p \right\},
\end{equation}
\[
    K_0(N) = \prod_{p < \infty} K_0(N)_p.
\]

\subsection{Root Data}\label{root}

We review some standard material to fix notation and terminology that will be used
  in the sequel.
Let $F$ be an algebraically closed field.
In this subsection we write $G$ for $G(F)$, and similarly for the other algebraic
  groups considered.
In $G$, the diagonal subgroup
\begin{equation}\label{T}
T=\{t=\diag(a_1,\ldots,a_n,\tfrac r{a_1},\ldots,\tfrac r{a_n})|\,ra_1\cdots a_n\neq 0\}
\end{equation}
is a split maximal torus.
Given $\chi\in X^*(T)$, there exist $k_0,\ldots,k_n\in\Z$ such that
\begin{equation}\label{k}
\chi(t)=r^{k_0}a_1^{k_1}\cdots a_n^{k_n}
\end{equation}
for $t$ as in \eqref{T}.
By associating $\chi$ with the tuple $(k_0,\ldots,k_n)$, we obtain a natural
  identification $X^*(T)\cong \Z^{n+1}$.
Let $e_j\in \Z^{n+1}$ be the $(j+1)$-th standard basis vector ($j=0,\ldots,n$).
The set $\Phi$ of roots (for the action of $T$ on $\Lie(G)$) consists of
\begin{equation}\label{roots}
\begin{aligned}
\pm(e_j-e_i)\quad(1\le i<j\le n),\\
\pm(e_0-e_j-e_i)\quad(1\le i\le j\le n).
\end{aligned}
\end{equation}
We take the set $\Phi^+$ of positive roots to consist of those
  which have the $+$ coefficient.  The corresponding set of simple roots is
\[\Delta = \{e_{j+1}-e_j|j=1,\ldots,n-1\}\cup \{e_0-2e_n\}.\]

Let
\[\rho=\frac12\sum_{\chi\in \Phi^+}\chi=\frac12\left(
  \sum_{1\le i<j\le n}(e_j-e_i)+\sum_{1\le i\le j\le n}(e_0-e_i-e_j)\right).\]
Explicitly,
\begin{equation}\label{rho}
\rho=\frac{n(n+1)}4e_0-ne_1-(n-1)e_2-\cdots-e_n.
\end{equation}

The cocharacter lattice is $X_*(T)=\Hom(F^*,T)$.
  We identify a tuple
  $\lambda=(\ell_0,\ell_1,\ldots,\ell_n)\in \Z^{n+1}$ with the cocharacter
\begin{equation}\label{ell}
\lambda(a)=\diag(a^{\ell_1},\ldots,a^{\ell_n},a^{\ell_0-\ell_1},\ldots,
  a^{\ell_0-\ell_n}).
\end{equation}
In this way, $X_*(T)\cong \Z^{n+1}$.

The composition of a character with a cocharacter yields a rational homomorphism
   $F^*\rightarrow F^*$, which
  is necessarily of the form $x\mapsto x^m$ for $m\in\Z$.  Thus, we have a natural pairing
  $X^*(T)\times X_*(T)\rightarrow \Z$ given by
\begin{equation}\label{pairing}
\chi(\lambda(x))=x^{\sg{\chi,\lambda}}.
\end{equation}
  In terms of the coordinates given above, this works out to
\begin{equation}\label{paircoord}
\sg{(k_0, \ldots, k_n), (\ell_0, \ldots, \ell_n)} =
k_0 \ell_0 + \cdots + k_n \ell_n.
\end{equation}

Our main group of interest is $\olG =\PGSp_{2n}$.
  Letting $\olT=T/Z$ be the maximal torus, we can identify its character
lattice $X^*(\olT)$ with the subset of $X^*(T)$ consisting
  of all characters which are trivial on $Z$.  Thus,
\begin{equation}\label{kset}
X^*(\olT)\cong \{ (k_0,k_1,\ldots,k_n)\in \Z^{n+1}|\, 2k_0+k_1+\cdots+k_n=0\}.
\end{equation}
Note that the roots \eqref{roots} belong to this set, so we may identify
  $\Phi$ with the set of roots in $X^*(\olT)$.
The cocharacter lattice $X_*(\olT)$ can be viewed as the quotient of $X_*(T)$ by the
  subgroup of cocharacters taking values in $Z$.  In terms of the coordinates above
  \eqref{ell}, we have
\begin{equation}\label{lquot}
X_*(\olT)\cong \Z^{n+1}/(2,1,\ldots,1)\Z.
\end{equation}
The pairing \eqref{pairing} makes sense for $(\chi,\lambda)\in X^*(\olT)\times X_*(\olT)$,
  and the formula \eqref{paircoord} is independent of the choice of coset representative
  for $\lambda$.

The Weyl group of $G$, namely
\[W=N_{\olG}(\olT)/Z_{\olG}(\olT)= N_G(T)/Z_G(T),\]
 acts on $T$ (and also $\olT$) by conjugation.  It is isomorphic to $S_n\ltimes (\Z/2\Z)^n$,
  with the following generators:
\[
 t \mapsto \operatorname{diag}(a_{\sigma(1)}, \ldots,a_{\sigma(n)},
\tfrac{r}{a_{\sigma(1)}}, \ldots, \tfrac{r}{a_{\sigma(n)}})
\]
for $\sigma$ in the symmetric group $S_n$,
and, for $1\le i\le n$,
\[
t \mapsto \operatorname{diag}(a_1, \ldots, a_{i-1},  \tfrac{r}{a_i}, a_{i+1},
\ldots,  a_n, \tfrac{r}{a_1}, \ldots, \tfrac{r}{a_{i-1}},  a_i, \tfrac{r}{a_{i+1}},
 \ldots,    \tfrac{r}{a_n}).
\]
  Likewise, $W$ acts faithfully on $X^*(\olT)$ by
\[w\chi(t)=\chi(wtw^{-1}).\]
The corresponding generators are
\begin{equation}\label{W1}
  (k_0, k_1, \ldots, k_n) \mapsto (k_0, k_{\sigma^{-1}(1)}, \ldots, k_{\sigma^{-1}(n)})
\end{equation}
and
 \begin{equation} \label{W2}
  (k_0, k_1, \ldots, k_n) \mapsto  (k_0+k_i, , k_1, \ldots, k_{i-1}, -k_i, k_{i+1}, \ldots, k_n).
 \end{equation}

Using the pairing \eqref{pairing}, an action of $W$ on $X_*(\olT)$ is defined implicitly
  via
\[\sg{w\chi,w\lambda}=\sg{\chi,\lambda}.\]
In terms of the $\ell$-coordinates in \eqref{ell}, the action on $X_*(T)$
  of the Weyl element in \eqref{W1} is given by
\[
   (\ell_0, \ell_1, \ldots, \ell_n) \mapsto (\ell_0, \ell_{\sigma^{-1}(1)}, \ldots,
\ell_{\sigma^{-1}(n)}),
\]
and the one corresponding to \eqref{W2} is given by
\[
  (\ell_0, \ell_1, \ldots, \ell_n) \mapsto (\ell_0,  \ell_1, \ldots, \ell_{i-1},
\ell_0-\ell_i, \ell_{i+1}, \ldots, \ell_n).
\]

Suppose $\chi, \chi'\in X^*(\olT)$ correspond respectively to $(k_i),(k_i')\in\Z^{n+1}$
   as in \eqref{kset}.  Then the pairing
\[(\chi,\chi') = \sum_{k=1}^n k_ik_i'\]
is $W$-invariant.  (This is easily verified using the relation in \eqref{kset}.)
  For a root $\alpha\in \Phi$, there is a unique coroot $\alpha^\vee\in X_*(\olT)$
  satisfying
\[\sg{\chi,\alpha^\vee}=\frac{2(\chi,\alpha)}{(\alpha,\alpha)}\]
for all $\chi\in X^*(\olT)$.  We let $\Phi^\vee\subset X_*(\olT)$ denote the set of coroots.
These are given explicitly as follows.  Let $f_i\in X_*(\olT)$ denote the dual of $e_i$;
in \eqref{lquot}, it is the coset attached to the $i$-th standard basis vector of
  $\Z^{n+1}$.
Then $(e_j-e_i)^\vee=f_j-f_i$ for $i\neq j$, $(e_0-e_j-e_i)^\vee = -f_j-f_i$ for $i\neq j$,
  and $(e_0-2e_j)^\vee = -f_j$.

The above gives a description of the root datum $(X^*(\olT),\Phi,X_*(\olT),\Phi^\vee)$
  of $\olG=\PGSp_{2n}$.  The complex dual group $\Ghat$ (with dual root datum
  $(X_*(\olT),\Phi^\vee,X^*(\olT),\Phi)$) is
  $\Spin(2n+1,\C)$.  In particular, for the maximal torus $\That$ of $\Ghat$, we have
\begin{equation}\label{Tdual}
X^*(\That)\cong X_*(\olT),\qquad X_*(\That)\cong X^*(\olT).
\end{equation}

Now define the positive Weyl chamber
\begin{equation}\label{C+}
\mathcal{C}^+=\{\lambda\in X_*(\olT)|\, \sg{\chi,\lambda}\ge 0\text{ for all }
  \chi\in \Delta\}.
\end{equation}
  We will frequently identify $\mathcal{C}^+$ with its counterpart
  in $X^*(\That)\cong X_*(\olT)$.

By definition, an element $\lambda\in X_*(\olT)$ belongs to
  $\mathcal{C}^+$ if and only if
\[\sg{e_{j+1}-e_j,\lambda}\ge 0\text{ for }j=1,\ldots,n-1\text{ and }
  \sg{e_0-2e_n,\lambda}\ge 0.\]
The above holds if and only if, in the notation of \eqref{lquot},
   every coset representative
$(\ell_0,\ell_1,\ldots,\ell_n)$
for $\lambda$ satisfies
\begin{equation}\label{Ccond0}
\ell_1\le \ell_2\le \cdots\le \ell_n \le \ell_0/2.
\end{equation}
In fact, each such $\lambda$ has a unique representative
satisfying
\begin{equation}\label{Ccond}
0= \ell_1\le \ell_2\le \cdots\le \ell_n \le \ell_0/2.
\end{equation}
 For notational convenience, we will often identify $\lambda$ with
  this coset representative.

\begin{proposition}[Cartan decomposition, \cite{Gr}] The group $\olG(\Q_p)$ is the disjoint
  union of the double cosets $K_p\lambda(p)K_p$ for $\lambda\in\mathcal{C}^+$.
The analogous statement
  (involving $\lambda$ satisfying \eqref{Ccond0} rather than \eqref{Ccond})
  also holds for $G(\Q_p)$.
\end{proposition}

\begin{proposition}\label{fp}
For $g\in G(\Q_p)$ and $m\ge 1$, let $d_m(g)$ denote the generator (chosen as a power of $p$)
  of the fractional ideal of $\Q_p$ generated by the set
\[\{\det B|\, B\text{ is an }m\times m \text{ submatrix of }g\}.\]
Then given $\lambda\in X_*(T)$ satisfying \eqref{Ccond0},
an element $g\in G(\Q_p)$ belongs to the double coset
   $G(\Z_p)\lambda(p)G(\Z_p)$ if and only if each of the following holds:
\begin{enumerate}
\item $r(g)=p^{\ell_0}$
\item for each $m=1,\ldots,n$, $d_m(g)= p^{\ell_1+\cdots+\ell_m}$.
\end{enumerate}
\end{proposition}

\begin{proof}
This result can be extracted from Chapter II of Newman \cite{Ne}.
  For convenience, we sketch some of the details.  Let $g \in G(\Q_p)$.
By the Cartan decomposition, there exist $k_1, k_2 \in G(\Z_p)$ such that $k_1 g
k_2=\lambda(p)$ for a unique cocharacter $\lambda$ satisfying \eqref{Ccond0}.
  We need to show that the two conditions given above are satisfied. The converse
  will then also follow, since $\lambda(p)$ is uniquely determined by its
first $n$ diagonal entries and its similitude.

The first condition is immediate.  For the second, observe that
there exists an integer $a \ge 0$ such that $p^a g\in M_{2n}(\Z_p)$.
The $m$-th diagonal coordinate of $k_1p^agk_2$ is then $p^{a+\ell_m}\in\Z_p$.
Regarding $p^ag$ as an element of $M_{2n}(\Z_p)$ and regarding $k_1, k_2$ as
elements of $\GL_{2n}(\Z_p)$,
by \cite[Chap. II, Sect. 16, Eq. (13)]{Ne} with $R=\Z_p$,
we have
\[ p^{a+\ell_m} =\frac{d_m(p^ag)}{d_{m-1}(p^ag)} = \frac{p^{am} d_m(g)}{p^{a(m-1)}
d_{m-1}(g)} = p^a \frac{d_m(g)}{d_{m-1}(g)} \]
(with $d_0(g)=1$).  Hence
\[ p^{\ell_m} = \frac{d_m(g)}{d_{m-1}(g)}. \]
This is easily seen to be equivalent to condition (2), as needed.
\end{proof}

\section{Adelic Siegel modular forms} \label{siegelmod}
Let $\A$ denote the adele ring of $\Q$, and fix a Haar measure $dg$ on $\olG(\A)$.
Let $L^2=L^2(\olG(\Q)\bs \olG(\A))$ be the space of measurable
  functions $\phi: {G}(\A) \rightarrow \C$ satisfying
\begin{itemize}
\item $\phi(z\g g)=\phi(g)$ for all $z\in Z(\A)$, $\g\in G(\Q)$, $g\in G(\A)$
\item  $\ds \int_{\olG(\Q) \bs \olG(\A)} |\varphi(g)|^2 dg < \infty$.
\end{itemize}
   For any parabolic subgroup $P$ of $G$, $P$ can be written as $MN$, where
   $M$ is the Levi subgroup and $N$ is unipotent.
   An element $\varphi \in L^2$ is cuspidal if for any parabolic
   subgroup $P=MN$ of $G$,
\[
   \int_{N(\Q) \bs N(\A)} \varphi(ng) dn = 0 \quad \text{ for a.e. $g\in G(\A)$}.
\]
  We let $L_0^2\subset L^2$ denote the subspace of cuspidal functions.

The right regular representation of $G(\A)$ on $L^2_0$
  decomposes discretely as $\bigoplus \pi$ with possible multiplicity
  allowed, where $\pi$ are by definition the cuspidal automorphic
  representations of $G(\A)$ with trivial central character.
Any such constituent $\pi$ is a restricted tensor product
\[\pi=\bigotimes_{p\leq \infty} \pi_p\]
where $\pi_p$ is an irreducible admissible representation of $G(\Q_p)$.

Fix an integer $\k>n$ with $n\k$ even (so that \eqref{pikcc} is trivial).
  Then as in \S\ref{pik}, there is a holomorphic discrete series representation
  $\pi_\k$ of $\olG(\R)$ of weight $\k$. In the inner product space $\pi_\k$,
  up to unitary scaling there is a unique holomorphic unit vector
  $\phi_0$ satisfying
\begin{equation}\label{v0}
\pi_\k(\mat AB{-B}A)\phi_0=\det(A+Bi)^{\k} \phi_0\qquad(\mat AB{-B}A\in K_\infty).
\end{equation}

Let $\Pi_\k(N)$ denote the set of cuspidal representations $\pi$
  of $G(\A)$ with trivial central character
    for which $\pi_\infty=\pi_\k$ and $\pi_{\fin}^{K_0(N)}\neq 0$.
For such $\pi$, we let $v_{\pi_\infty}\in V_{\pi_\infty}$
  denote a lowest weight vector as in \eqref{v0}.
  For any representation $\pi_{\fin}$ of $G(\Af)$ and any subgroup $U <G(\Af)$,
  we write $\pi_{\fin}^U$ for the space of $U$-fixed vectors in the space of $\pi_{\fin}$.
Define
\begin{equation}\label{AkN}
A_\k(N) = \bigoplus_{\pi\in\Pi_\k(N)}
   \C v_{\pi_\infty} \otimes \pi_{\fin}^{K_0(N)}.
\end{equation}
This corresponds to a classical space of holomorphic Siegel cusp forms of
  weight $\k$ and level $N$ (see \S \ref{S1} below).

  For $\pi\in \Pi_\k(N)$, let
  $E_\k(\pi,N)$ be an orthogonal basis for the summand indexed by $\pi$ in \eqref{AkN}.
Then the set
\begin{equation}\label{EkN}
E_\k(N) = \bigcup_{\pi\in \Pi_\k(N)} E_\k(\pi,N)
\end{equation}
  is an orthogonal basis for $A_\k(N)$.

\section{The test function and its kernel}\label{test}

  Any function $f\in L^1(\olG(\A))$ defines an operator $R(f)$ on $L^2$ by
\[R(f)\phi(x)=\int_{\olG(\A)}f(g)\phi(xg)dg.\]
In this section we will define the bi-$K_0(N)$-invariant test function $f$
  to be used in the trace formula.

\subsection {Definition of the test function} \label{Cp}

We define, for $g\in G(\R)$,
\[f_\infty(g)={d_\k}\ol{\sg{\pi_\k(g)\phi_0,\phi_0}},\]
where $\phi_0$ is a unit vector in the space of $\pi_\k$ satisfying \eqref{v0},
  and
  $d_\k$ is the formal degree of $\pi_\k$.
The matrix coefficient is independent of the choice of Haar measure on $\olG(\R)$.
  It is computed explicitly in Corollary \ref{matcoeffmain} of the Appendix.
The formal degree $d_\k$, which depends on the choice of Haar measure,
  is given in \eqref{dk}.  For our purposes, the particular choice of measure
  is immaterial.

  By Proposition \ref{integrable},
  $f_\infty\in L^1(\olG(\R))$ precisely when
  \[\k>2n,\]
 so this hypothesis will be in force throughout.

We will take $\ff=\prod_{p<\infty}f_p$ to be a bi-$Z(\Af)K_0(N)$-invariant
  function on $G(\Af)$, of the following form.\footnote[2]{Although we have defined
  $K_p$ and $K_0(N)_p$ in \eqref{K0N} as subsets of $\olG(\Q_p)$, in this section
  we will blur the distinction between these sets and their preimages in $G(\Z_p)$.
  No confusion should occur since everything is invariant under the center.}
  Fix a finite set $\Sb$ of prime numbers not dividing $N$.
For $p\in\Sb$, let $f_p$ be a bi-$Z_pK_p$-invariant function with compact
  support modulo $Z_p$, taking the value $1$ on its support.
Here, $Z_p=Z(\Q_p)$ is the center of $G(\Q_p)$.
  By the Cartan decomposition, $f_p$ is the characteristic function of a set of
  the form $Z_pC_p$, with $C_p$ a finite union of double cosets of the form
   $K_p\lambda(p)K_p$, where, by \eqref{Ccond},
  \[\lambda(p)=\diag(p^{\ell_1},\ldots,p^{\ell_n},p^{\ell_0-\ell_1},\ldots,p^{\ell_0-\ell_n})\]
with $0= \ell_1\le \cdots\le \ell_n\le \ell_0/2$.
  Without any real loss of generality, we make the further assumption
  that the similitude of $C_p$ has constant valuation, i.e.
\[r(C_p)=p^{r_p}\Z_p^*\]
  for some integer $r_p\ge 0$.  This amounts to requiring that $\ell_0=r_p$ for each
  of the $\lambda(p)$ out of which $C_p$ is built.
  In particular, $C_p=K_p$ if $r_p=0$.
  Having fixed such $f_p$ for each $p\in\Sb$, we define the global similitude
 \begin{equation}\label{r}
r=\prod_{p\in \Sb} p^{r_p}\ge 1.
\end{equation}

Now for $p\notin \Sb$, define $C_p=K_0(N)_p$, and set
\[\psi(N)_p = \meas({K_0(N)}_p)^{-1}=[K_p:K_0(N)_p],\]
\[\psi(N)=\prod_p \psi(N)_p=[K_{\fin}:K_0(N)].\]
 We then define $f_p:G(\Q_p)\longrightarrow\C$ by
\[f_p(g)=\begin{cases}\psi(N)_p&\text{if }g\in Z_pC_p\\
0&\text{otherwise.}\end{cases}\]
(In fact this holds as well when $p\in \Sb$, since $\psi(N)_p=1$ in that case.)

Having fixed $\ff$ above, we set
\[f=f_\infty\times \ff.\]
 By our assumption that $\k>2n$, $f\in L^1(\olG(\A))$.
The following is a useful observation about the support of $\ff$.

\begin{proposition} \label{gamma}
Suppose $x, y \in G(\A_{\fin})$ satisfy $r(x)^{-1} r(y) \in \widehat{\Z}^*$.
Then for $\g\in G(\Q)$, $\ff(x^{-1} \gamma y) \neq 0$ only if
   there exists $s\in \Q^*$, uniquely determined up to its sign, such that
\begin{equation} \label{gamma1}
    r(\gamma) = \pm s^2 r.
\end{equation}
Suppose $\g\in G(\Q)$ satisfies \eqref{gamma1}, and set
   $\tilde{\gamma} = s^{-1} \gamma$.  Then $r(\tilde{\gamma}) = \pm r$,
   and $\ff(x^{-1}\gamma y) \neq 0$ if and only if
\begin{equation} \label{gamma2}
 x^{-1} \tilde{\gamma} y \in \prod_p C_p = \prod_{p\in \Sb} C_p \prod_{p|N} K_0(N)_p
\prod_{p\nmid N,p\notin \Sb}K_p\subset M_{2n}(\Zhat).
\end{equation}
\end{proposition}

\begin{proof}
If $\ff(x^{-1} \gamma y) \neq 0$, then $x^{-1} \gamma y=zc$ for some
  $z\in Z(\Af)$ and $c\in \prod_p C_p$.
Since $\Af^*=\Q^*\Zhat^*$, we may write $z=sa$ where $s\in\Q^*$ and $a\in \Zhat^*$.
 We may absorb $a$
  into $c$ so that $z=s$ without loss of generality.
   Taking the similitude and using $r(x)^{-1}r(y)\in\Zhat^*$,
   we see that $r(\g)=s^2ur$ for some $u\in \Zhat^*$.
   Since $r(\g),s,r\in \Q^*$, it follows that $u\in \Zhat^*\cap\Q^*=\{\pm 1\}$,
  as claimed. It is clear that $s$ is unique up to its sign, and that
  $x^{-1}\tilde{\g}y\in \prod_pC_p$.  Conversely,
  since the support of $\ff$ is $Z(\Af)\prod_pC_p$, it is also clear that
  if $x^{-1}\tilde{\g}y \in \prod_pC_p$, then $\ff(x^{-1}{\g}y)\neq 0$.
\end{proof}

\subsection{Spectral properties of $R(f)$}\label{Rf}

Let $f$ be the function defined above.  Recall that it is not completely fixed
because it depends on some choice of a finite union $C_p$ of double cosets for each
prime $p\in \Sb$.
Here we show that the Hecke operator $R(f)$ has finite rank, and we compute its effect on
  adelic Siegel cusp forms.

\begin{proposition} Let $U$ be a unipotent subgroup of $\olG$ and let $g\in \olG(\R)$.
   Then for almost all $x \in U(\R) \bs \olG(\R)$,
   \[ \int_{U(\R)} f_{\infty}(g u x) du = 0. \]
\end{proposition}

\begin{proof}
Let $U'$ be a one-dimensional subgroup of $U$. There exists $E\in M_{2n}(\R)$ such that
$U'(\R) = \{I + tE \,|\, t \in \R \}$.
By Corollary \ref{matcoeffmain}, there exist complex numbers $A$, $B$
and $C\neq 0$ depending on $u \in U(\R)$, $g,x\in\olG(\R)$, and $\k>2n$,
such that
\[f_{\infty}(g (I_n+tE)  ux) = \frac C{(At+B)^\k}.\]
The denominator is nonzero for  $t\in \R$ since the matrix coefficient is finite.
If $A \neq 0$, then by the fundamental theorem of calculus,
\[ \int_{\R} f_{\infty} (g (I_n+tE) u x) dt
  = \left.\tfrac{C}{A}(At+B)^{-\k+1}\right|_{-\infty}^\infty=0. \]
If $A=0$, then
$$\int_{\R} f_{\infty} (g (I_n+tE) u x) dt = \infty.$$
Hence
\[ \int_{U(\R)} f_{\infty} (g ux) du = \int_{U'(\R) \bs U(\R)} \int_{U'(\R)}
 f_{\infty} (g u' u x) du' du \]
is either $0$ or divergent.
It remains to show that this integral is convergent
   for almost all $x\in U(\R)\bs\olG(\R)$.
  But this is immediate from the fact that
because $f_{\infty}\in L^1(\olG(\R))$, the integral
\[ \int_{\olG(\R)} f_{\infty}(g x) dx
= \int_{U(\R) \bs \olG(\R)} \int_{U(\R)} f_{\infty}(g u x) du\, dx \]
is convergent.
\end{proof}

\begin{proposition}\label{contvanish}
For a test function $f$ as defined in \S \ref{Cp}, and the subspace $A_\k(N)\subset L^2_0$
  given in \eqref{AkN},
$R(f)$ annihilates $A_\k(N)^\perp$ and maps $A_\k(N)$ into itself.
\end{proposition}

\begin{proof}
Complete details for the case of $\GL(2)$ are given in
  \cite[Propositions 13.11, 13.12]{KL}, and everything carries over directly
  to the case under consideration here, using the above proposition in place
  of \cite[Corollary 13.10]{KL}.  So we just briefly sketch the ideas.
  It follows easily from the above proposition that
   for any $\phi\in L^2$, $R(f)\phi$ is cuspidal,
i.e. $R(f):L^2\rightarrow L^2_0$.
  Furthermore one checks that $f_\infty$ is self-dual, so the adjoint $R(f)^*$
  also has this property.  It then follows that $R(f)$ annihilates
  $(L^2_0)^\perp$.  Next, one may use a general orthogonality property
  of discrete series matrix coefficients,
  together with the $K_0(N)$-invariance of $f_{\fin}$, to show that the image of $R(f)$
  (and of $R(f)^*$) lies in $A_\k(N)$.
\end{proof}

It remains to compute the effect of $R(f)$ on a nonzero element $v\in A_\k(N)$.
We may assume without loss of generality that
  $v$ is a pure tensor in some cuspidal representation $\pi\in\Pi_\k(N)$.
  Write $\pi=\pi_\k\otimes \pi'\otimes \bigotimes_{p\in \Sb}\pi_p$,
  where $\pi'$ is a representation of $\prod'_{p\notin \Sb} G(\Q_p)$.
  Accordingly, we write $f=f_\infty \times f'\times\prod_{p\in \Sb}f_p$ and
\[v=v_\infty\otimes v'\otimes\bigotimes_{p\in \Sb}v_p,\]
where $v_\infty=\phi_0$ as in \eqref{v0}.
Then (e.g. by \cite[Prop. 13.17]{KL}) we have
\[R(f)v= \pi_\k(f_\infty)v_\infty\otimes \pi'(f')v'\otimes\bigotimes_{p\in \Sb}
  \pi_p(f_p)v_p.\]
By the orthogonality relations for discrete series matrix coefficients
   (\cite[Corollary 10.26]{KL}),
\[\pi_\k(f_\infty)v_\infty = v_\infty,\]
where $\pi_\k(f_\infty)$ is defined using the same Haar measure as that defining
  $d_\k$.
Likewise, because $f'$ is the characteristic function of
  $K_0(N)'=\prod_{p\notin \Sb}K_0(N)_p$
  scaled by $\meas(K_0(N)')^{-1}$, and $v'\in {\pi'}^{K_0(N)'}$,
  we have
\[\pi'(f')v'=v'.\]
For $p\in \Sb$, since $f_p$ is $K_p$ bi-invariant, $\pi_p(f_p)$ preserves the subspace
$\pi_p^{K_p}=\C v_p$.  Hence $v_p$ is an eigenvector.
Writing $(\mathcal{S}f_p)(t_{\pi_p})$ for the eigenvalue as in \eqref{wf}, we have
\begin{equation}\label{Rfv}
R(f)v=\left(\prod_{p\in \Sb}(\mathcal{S}f_p)(t_{\pi_p})\right)v
\end{equation}
for $v$ as above.

\subsection{The kernel function}

For a test function $f$ as defined in \S \ref{Cp}, the associated kernel function
  on $G(\A)\times G(\A)$ is defined as
\[K(x,y)=K_f(x,y)=\sum_{\gamma\in \olG(\Q)}f(x^{-1}\gamma y).\]
It satisfies
\[R(f)\phi(x)=\int_{\olG(\Q)\bs \olG(\A)} K(x,y)\phi(y) dy\qquad (\phi\in L^2).\]

\begin{proposition}\label{Kbound}
$K(x,y)$ is continuous in both variables.
Furthermore, given any subsets $J_1,J_2\subset G(\A)$ each having compact image
  in $\olG(\A)$, for every $\gamma \in \olG(\Q)$ there exists a real
  number $\alpha_\gamma$, independent of $N$ and $\k$, such that
for all $g_1\in J_1$ and $g_2\in J_2$,
\[ d_\k^{-1}\psi(N)^{-1} |f(g_1^{-1} \gamma g_2)| \le \alpha_\gamma\]
and
\[ \sum_{\gamma \in \olG(\Q)} \alpha_\gamma < \infty. \]
\end{proposition}

\begin{proof}
For $g\in \olG(\R)$, define $f_\k(g)=\ol{\sg{\pi_\k(g)\phi_0,\phi_0}}$,
  so that $f_\infty(g)=d_\k f_\k(g)$.
  Define the function $\tf=\tf_\infty\tf_{\fin}$ on $G(\A)$, where $\tf_\infty=d_{\k_0}f_{\k_0}$ with
\[\k_0=2n+1\le \k,\]
 and $\tf_{\fin}$ is the characteristic function of $\prod_p Z_p C_p$, where $C_p=K_p$
  for all $p\notin \Sb$.  In other words, $\tf$ is the test function we have defined
  earlier in the special case where $\k=\k_0$ and $N=1$.
  We claim that
\begin{equation}\label{lecl}
\psi(N)^{-1}d_\k^{-1}|f(g)|\le d_{\k_0}^{-1}|\tf(g)|.
\end{equation}
  Since $\supp(f_{\fin})\subset \supp(\tf_{\fin})$, it is clear from the definitions that
  $\psi(N)^{-1}|f_{\fin}(g)|\le |\tf_{\fin}(g)|$ for all $g\in G(\Af)$.
For the archimedean part, by Corollary \ref{matcoeffmain},
  when $g=\smat ABCD\in\olG(\R)$ we have
\[\left[\frac{r(g)^{n/2}2^n}{|\det(A+D+i(B-C))|}\right]^\k=|f_\k(g)|
    =|\sg{\pi_\k(g)\phi_0,\phi_0}|\]
    \[\le \|\pi_\k(g)\phi_0\|\|\phi_0\|=1,\]
  where we have used the fact that $\phi_0$ is a unit vector and $\pi_\k$ is unitary.
  It follows that the expression in the brackets is at most $1$, and hence
  $|f_\k(g)|\le |f_{\k_0}(g)|$ for all $g\in \olG(\R)$.
  This proves \eqref{lecl}.

Hence, it suffices to prove the assertion for $f=\tf$ (with $N=1$).
    By \cite[Prop. 3.1]{Li}, it suffices to show that there exist bounded
  compactly supported functions
$\psi_1$, $\psi_2$ on $\olG(\A)$ such that $\tf = \tf*\psi_1 = \psi_2*\tf$, where
  convolution is defined by
\[ f_1*f_2 = \int_{\olG(\A)} f_1(\alpha) f_2(\alpha^{-1} g) d\alpha
= \int_{\olG(\A)} f_1(g \alpha^{-1}) f_2(\alpha) d\alpha. \]

For the finite part of $\psi_j$, we take the characteristic function
  of $Z(\Af) K_0(N)$, multiplied by the reciprocal of the measure of this set in
  $\olG(\Af)$.
It remains to define the archimedean components of $\psi_1,\psi_2$.
We claim first that there exists a measurable compactly supported
  function $\xi$ on $\olG(\R)$ satisfying
\[\xi(\mat UV{-V}U g)=\det(U+iV)^{\k_0}\xi(g)\]
for all $\smat UV{-V}U\in K_\infty$, or equivalently, since $G(\R)=K_\infty P(\R)$,
\[\xi(\mat UV{-V}U \mat ABOD)=\det(U+iV)^{\k_0}\xi(\mat ABOD).\]
Let $P'\subset P(\R)$ be a set of representatives for $(K_\infty\cap P(\R))\bs P(\R)$.
Then every element $g\in G(\R)$ has a unique decomposition $g=kb'$ with $k\in K_\infty$
  and $b'\in P'$.  We may choose $P'$ so that it is a real manifold,
  and then take $\xi$ to be any compactly supported function on $P'$, extended to $G(\R)$
  by $\xi(kb')=\det(k)^{\k_0}\xi(b')$.
  One choice of $P'$ is given as follows.  Notice that
\[K_\infty\cap P(\R)=\{\mat U{}{}U|\, U\,{}^tU=I\}\cong \mathrm{O}(n),\]
so using the decomposition \eqref{B} of $P(\R)$, we see that
\[(K_\infty\cap P(\R))\bs P(\R)
  \cong \Bigl(\mathrm{O}(n)\bs \GL_n(\R)\Bigr)\times\R^*\times S_n(\R).\]
Hence we can take $P'\subset P(\R)$ to be the subgroup identifying as
\[P'\cong \Bigl\{\left(\begin{smallmatrix}a_{11}&*&\cdots&*\\
  &a_{22}&\cdots&*\\
  &&\ddots&*\\
  &&&a_{nn}\end{smallmatrix}\right)|\, a_{jj}>0\Bigr\}\times \R^*\times S_n(\R).\]
The proof now proceeds exactly as in \cite[Prop. 3.2]{Li}: one sees easily that
  $\pi_{\k_0}(\xi)\phi_0$ is a vector of weight $\k_0$ as in \eqref{v0}, and hence must be a
  multiple of $\phi_0$.  For an appropriate choice of $\xi$, the multiple is nonzero,
  and hence without loss of generality $\pi_{\k_0}(\xi)\phi_0=\phi_0$.  From here it is easy
  to show directly that $\xi*f_{\k_0}=f_{\k_0}$, so we can take
  $\psi_1=\xi\times \psi_{\fin}$.  Similarly, we can take $\psi_2=\psi_1^*$.
\end{proof}

\subsection{Spectral expression for the kernel}

Because the support of the test function is not compact modulo the center,
  some care is needed
  in order to justify the spectral expansion of the kernel function.

\begin{proposition}
With notation as in \S \ref{siegelmod}, the kernel function of the operator
  $R(f)$ has the spectral expansion
\begin{equation}\label{kerspec}
   K(x,y) = \sum_{\pi\in\Pi_\k(N)}
\Bigl(\sum_{\varphi \in {E}_\k(\pi,N)}
   \frac{\varphi(x) \ol{\varphi(y)}}{\|\varphi\|^2}\Bigr)
   \prod_{p\in \Sb} (\mathcal{S}f_p)(t_{\pi_p}).
\end{equation}
\end{proposition}

\begin{proof}
As shown in \S \ref{Rf}, the operator $R(f)$ vanishes on $A_\k(N)^\perp$
  and is diagonalizable on $A_\k(N)$, the elements of $E_\k(N)$ (see \eqref{EkN})
  being eigenvectors.  It follows easily that
\[\Phi(x,y)= \sum_{\varphi \in E_\k(N)} \frac{R(f) \varphi(x) \ol{\varphi(y)}}{\|\varphi\|^2}\]
is a kernel function for $R(f)$.  (See e.g. \cite[p. 228]{KL} for details.)
It follows that $K(x,y)=\Phi(x,y)$ a.e.
On the other hand, $\Phi$ is continuous in both variables,
  being a sum of finitely many adelic Siegel cusp forms, while
by Proposition \ref{Kbound}, $K(x,y)$ is also continuous.
  Hence they are equal everywhere.
  Using \eqref{Rfv} we see that $\Phi(x,y)$ is equal to the spectral expression
  given in \eqref{kerspec}.
\end{proof}

\section{Asymptotic Fourier trace formula for $\GSp(2n)$}\label{asymp}

\subsection{Additive characters}
  Let
\[\theta:\Q\bs \A\longrightarrow \C^*\]
  be the nontrivial character whose local components are given by
\begin{equation}\label{thetadef}
\theta_p(x)=\begin{cases}e^{2\pi i x}&\text{if }p=\infty\\
e^{-2\pi i r_p(x)}&\text{if }p<\infty,\end{cases}
\end{equation}
where $r_p(x)\in\Q$ is any number with $p$-power denominator satisfying
  $x\in r_p(x)+\Z_p$.
  All characters of the additive group $\Q\bs \A$ are of the form $\theta(qx)$
   for some $q\in \Q$.
  It follows easily that any character of the additive group
  $S_n(\Q)\bs S_n(\A)$ has the form
  $S\mapsto \theta(\tr \sigma S)$ for some $\sigma\in S_n(\Q)$.
  We fix two such matrices $\sigma_1,\sigma_2\in S_n(\Q)$
and define
\[\theta_j(S)=\theta(\tr\sigma_j S)\qquad(j=1,2).\]
Since we are interested in the $\theta_j$-Fourier coefficients of Siegel cusp forms,
  we can in fact assume that the $\sigma_j$ belong to
 the set
\[\mathcal{R}_n^+= \text{set of half-integral
  positive definite symmetric matrices }\sigma \in \GL_n(\Q)\]
  (this means $2\sigma$ has integer entries and even diagonal entries).

\subsection{The setup and the spectral side}

Given a continuous function $\varphi$ on $\olG(\Q)\bs \olG(\A)$, its Whittaker function along
  the unipotent subgroup $N=\{\mat{I_n}SO{I_n}|\, S\in S_n(\A)\}$ is defined by
\[W_\varphi(g,\chi)=\int_{N(\Q)\bs N(\A)}\varphi(ng)\ol{\chi(n)}dn\]
for $g\in \olG(\A)$ and $\chi$ a character.
Write $n_S=\smat{I_n}SO{I_n}$ for $S\in S_n(\A)$.  There exits
 $\sigma\in S_n(\Q)$ such that
$\chi(n_S)=\theta(\tr\sigma S)$ for all $S$.  Define
\begin{equation}\label{csigma}
c_\sigma(\varphi)=W(1,\chi)
  =\int_{S_n(\Q)\bs S_n(\A)}\varphi(n_S)\ol{\theta(\tr \sigma S)}dS.
\end{equation}

In \S\ref{Cp}, we defined a test function, henceforth to be denoted $f_N$,
  from the following data:
  a finite set $\Sb$ of primes, a level $N$ coprime to $\Sb$,
  a compact set $C_p\subset G(\Q_p)$ for each $p\in\Sb$, and a weight $\k>2n$.
  In this section, we compute the following limit:
\begin{equation}\label{I}
I=\limN\,
\iint\limits_{\scriptstyle (N(\Q)\bs N(\A))^2}
  \frac{K_{f_N}(n_1,n_2)}
{\psi(N)}
\ol{\theta_1(n_1)}\theta_2(n_2) dn_1 dn_2,
\end{equation}
where
$\psi(N)=\meas(\ol{K_0(N)})^{-1}$.
We will compute $I$ in two ways, using the spectral and geometric forms of the kernel
  function.
  The preliminary form of the resulting formula is given in Theorem \ref{A} below.

Using the spectral form \eqref{kerspec} of the
  kernel function along with \eqref{csigma}, we formally obtain
\begin{align}\label{Ispec}
  I= \limN&\frac1{\psi(N)}\sum_{\pi\in \Pi_\k(N)}\sum_{\varphi\in E_\k(\pi,N)}\frac{c_{\sigma_1}(\varphi)
  \ol{c_{\sigma_2}(\varphi)}} {\|\varphi\|^2}
\prod_{p\in S}(\mathcal{S}f_p)(t_{\pi_p}).
\end{align}
The existence of the limit will be demonstrated in Proposition \ref{Ilim} below.

\subsection{The geometric side}

First we show how the limit can be eliminated on the geometric side.

\begin{proposition}\label{Ilim}
Let $f=f_1$ be the test function we have defined when $N=1$.
  (Its finite part is the characteristic function of $\prod Z_pC_p$, where $C_p=K_p$
  for all $p\notin\Sb$.)
Then for $I$ as in \eqref{I}, the limit exists, and
\begin{equation}\label{I1}
I= \iint_{(N(\Q) \bs N(\A))^2} \sum_{\gamma \in \mathbb{P}(\Q)} f(n_1^{-1} \gamma n_2)
\ol{\theta_1(n_1)} \theta_2(n_2) dn_1 dn_2.
\end{equation}
\end{proposition}

\begin{proof}
Recall that $[0,1) \times \Zhat$ is a fundamental domain in $\A$ for $\Q \bs \A$.
  We may therefore replace $N(\Q)\bs N(\A)\cong N(\Q\bs \A)$ by the compact set
  $J=N([0,1]\times\Zhat)$.  Applying Proposition \ref{Kbound} with
   $J_1=J_2=J$, the integrand in \eqref{I} is absolutely bounded by the
  constant $\sum \alpha_\g$.
  Hence by the dominated convergence theorem,
\[ I=\limN  \iint_{J\times J}  \frac{K_{f_N}(n_1, n_2)}{\psi(N)}
\ol{\theta_1(n_1)} \theta_2(n_2) dn_1 dn_2
\]
\[ = \iint_{J\times J} \sum_{\gamma \in \olG(\Q)}
  \limN \frac{f_N(n_1^{-1} \gamma n_2)}{\psi(N)}
  \ol{\theta_1(n_1)} \theta_2(n_2) dn_1 dn_2.
\]
By Proposition \ref{gamma}, we can assume $r(\gamma)=\pm r$. Furthermore,
 if
\[f_{N,\fin}(n_{1\fin}^{-1} \gamma n_{2\fin}) \neq 0 \]
  for $n_{j\fin} \in N(\Zhat)$, then
\[ n_{1\fin}^{-1} \gamma n_{2\fin} \in \prod_{p\in \Sb} C_p \prod_{p\notin \Sb} K_0(N)_p. \]
Therefore
\[ \gamma = \mat WXYZ \in \prod_{p\in \Sb} C_p \prod_{p\notin \Sb} K_0(N)_p \subset M_{2n}(\Zhat). \]
In particular $Y\in M_n(N\Z)$.
Hence for any $Y \neq 0$, $f_{N,\fin}(n_{1\fin}^{-1} \gamma n_{2\fin})=0$ when
  $N$ is sufficiently large.
On the other hand, if $Y=0$, then for $n_j\in J$,
\[\frac{f_N(n_1^{-1}\g n_2)}{\psi(N)} = f(n_1^{-1}\g n_2)=f_\infty(n_{1,\infty}^{-1}\g n_{2,\infty})
  f_{\fin}(\g)\]
for $f=f_1$, which is obviously independent of $N$.
  In particular,
\[
I= \iint_{J\times J} \sum_{\gamma \in {\mathbb P}(\Q)} f_{1}(n_1^{-1} \gamma n_2)
\ol{\theta_1(n_1)} \theta_2(n_2) dn_1 dn_2.
\]
As a function of $n_1$ and $n_2$, the summation over ${\mathbb P}(\Q)$ is $N(\Q)$-invariant
  in both variables. So we can replace the region of the double integral
   by $(N(\Q) \bs N(\A))^2$. This completes the proof.
\end{proof}

\noindent{\em Remark:}  As indicated in the Introduction, the same idea cannot directly
be used to understand the case of fixed $N$ and $\k\to\infty$.
The above argument hinges on the fact that
        \[\lim_{N\to\infty} \frac{f(n_1^{-1}\g n_2)}{\psi(N)}=0\]
    for $n_1,n_2\in J$ and $\g\notin \mathbb{P}(\Q)$.
     For the case of varying $\k$ with $N=1$, the analogous argument would require
    \begin{equation}\label{klim}
    \lim_{\k\to\infty} \frac{f(n_1^{-1}\g n_2)}{c_{n\k\sigma}}=0
    \end{equation}
    for the same $\g,n_1,n_2$, with $c_{n\k\sigma}$ as in Theorem \ref{S=1}.
    However, this is false. We have
    $\ds c_{n\k\sigma}^{-1}=\frac{CB^\k\prod_{j=1}^n\Gamma(\k-\frac{n+j}2)}{d_\k}$
    for positive constants $B$ and $C$ depending only on $n$,
and $d_\k$ is a polynomial in $\k$.  By Stirling's approximation, the product of Gamma factors
grows faster than $\k^{(n-\e)\k}$ for any $\e>0$.
    On the other hand, as we will show in Corollary \ref{mcnorm},
     $d_{\k}^{-1}|f(n_1\g n_2)|$
     has at most exponential decay as $\k\to\infty$.
   It follows that the limit appearing in \eqref{klim} is actually divergent
   for many choices of $\g\in \olG(\Q)\setminus\mathbb{P}(\Q)$.\\

We now examine the sum in \eqref{I1}.
By Proposition \ref{gamma}, we may assume that $r(\g)=\pm r$.
  Because $f_\infty(g)=0$ if $r(g) < 0$, we may in fact take $r(\g)=r$.
By a variant of \eqref{B}, we may write
\[\g=\eta \, g_{A,r},\]
where $\eta\in N(\Q)$ and, for $A\in {\GL_n(\Q)}$,
\[g_{A,r}:=\mat A{}{}{r\,{}^tA^{-1}}.\]
Recalling that $\g\in \mathbb{P}(\Q)$ is only defined up to multiplication by the center,
we observe that for $\lambda\in\Q^*$, $\lambda g_{A,r}=g_{\lambda A, r\lambda^2}$.
Hence by our insistance that $r(\g)=r$, we may only scale by $\lambda=\pm 1$.
Therefore
\[ \sum_{\gamma \in \mathbb{P}(\Q)} f(n_1^{-1} \gamma n_2) =
\sum_{A\in \GL_n(\Q)/\{\pm I_n\} } \sum_{\eta\in N(\Q)}f(n_1^{-1}\eta g_{A,r}n_2).\]
    Taking, as we may, $n_1, n_2$ to range through the fundamental domain
  $N([0,1] \times \Zhat)$ for $N(\Q)\bs N(\A)$,
   by \eqref{gamma2} a given summand vanishes unless
\[  \eta g_{A,r} \in n_{1\fin} M_{2n}(\Zhat) n_{2\fin}^{-1} \subseteq M_{2n}(\Zhat).\]
  This implies that $A\in M_n(\Z)$ and $rA^{-1} \in M_n(\Z)$.
  Hence
\[I=\sum_{A\in\{\pm I_n\}\bs M_n(\Z),\atop{rA^{-1}\in M_n(\Z)}}
\iint_{(S_n(\Q)\bs S_n(\A))^2} \sum_{S\in S_n(\Q)}
f \left( \smat{I_n}{-S_1}{O}{I_n} \smat{I_n}{S}{O}{I_n} g_{A,r} \smat{I_n}{S_2}O{I_n} \right)
\]
\[\times
  \theta(-\tr \sigma_1 S_1 + \tr \sigma_2 S_2) dS_1 dS_2.\]
The double integral becomes
\[\int\limits_{ S_n(\Q) \bs S_n(\A)}\int_{S_n(\A)}
f \left( \smat{I_n}{-S_1}{O}{I_n} g_{A,r} \smat{I_n}{S_2}O{I_n} \right)
  \theta(-\tr \sigma_1 S_1 + \tr \sigma_2 S_2) dS_1 dS_2,\]
  which equals
\[ \int\limits_{ S_n(\Q) \bs S_n(\A)}\int_{S_n(\A)}
  f \left( \smat{I_n}{-(S_1-r^{-1} A S_2 \t  A)}{O}{I_n} g_{A,r}  \right)
  \theta(-\tr \sigma_1 S_1 + \tr \sigma_2 S_2) dS_1 dS_2. \]
  Making the substitution  $S'_1=S_1 - r^{-1} A S_2 \t A$, the above becomes
\begin{align*}
   \iint\limits_{\scriptstyle S_n(\A) \times (S_n(\Q) \bs S_n(\A))}
    &f(\left( \mat{I_n}{-S'_1}{O}{I_n} g_{A,r}  \right)\\
    &\times
   \theta(-\tr(\sigma_1 S'_1 + r^{-1} \sigma_1 A S_2 \t A) + \tr \sigma_2 S_2) dS'_1 dS_2
\end{align*}
\begin{align*}
    =\int\limits_{\scriptstyle S_n(\A)} &f(\left( \mat{I_n}{-S'_1}{O}{I_n} g_{A,r}
    \right) \theta(-\tr \sigma_1 S'_1) dS'_1\\
    &\times
\int\limits_{\scriptstyle S_n(\Q)\bs S_n(\A)}
  \theta(\tr(\sigma_2 S_2-r^{-1} \sigma_1 AS_2 \t A)) dS_2.
\end{align*}
The value of the second integral is
\[
   \begin{cases}
  \meas(S_n(\Q) \bs S_n(\A)) = 1 &
       \text{if $\theta(\tr(\sigma_2 S_2-r^{-1} \sigma_1 AS_2 \t A)) = 1$}\\
                                     &\text{for all $S_2 \in S_n(\A)$}, \\\\
   0 & \text{otherwise}.
   \end{cases} \]

\begin{lemma} We have $\theta(\tr (\sigma_2 S- r^{-1} \sigma_1 A S \t A))= 1$ for all $S \in S_n(\A)$
  if and only if
\[\t A \sigma_1 A = r \sigma_2.\] \end{lemma}

\begin{proof} By the fact that $\tr(AB) = \tr(BA)$,
 \[\theta(\tr (\sigma_2 S- r^{-1} \sigma_1 A S \t A))
= \theta(\tr ((\sigma_2 -r^{-1} \t A \sigma_1 A) S)).\]
 The lemma follows from this.
   \end{proof}

\subsection{Main formula}

The result of the above computation is the following asymptotic Petersson formula
  for $\PGSp(2n)$.

\begin{theorem} \label{A} Let $\Sb$ be a finite set of prime numbers, fix sets $C_p$
  as in \S \ref{Cp}, and let $r\ge 1$ be the integer defined in \eqref{r}.
  Let $f=f_1$ be the associated test function defined there
  for the case of level $N=1$.  Then for $\k>2n$ with $2|n\k$, and
  any $\sigma_1,\sigma_2\in \mathcal{R}_n^+$,
\begin{align}\notag \limN&\frac1{\psi(N)}\sum_{\pi\in \Pi_\k(N)}
  \left(\sum_{\varphi\in E_\k(\pi,N)}\frac{c_{\sigma_1}(\varphi)
  \ol{c_{\sigma_2}(\varphi)}} {\|\varphi\|^2}\right)
\prod_{p\in \Sb}(\mathcal{S}f_p)(t_{\pi_p})
\\
\label{Ae}
&= \sum_{A} \int_{S_n(\A)} f(\mat{I_n}SO{I_n} \mat AOO{r\t A^{-1}}) \theta(\tr \sigma_1 S) dS.
 \end{align}
    Here, $\Pi_\k(N)$ and $E_\k(\pi,N)$ were defined in \S \ref{siegelmod},
    $S_n(\A)$ is the set of symmetric $n\times n$ matrices over the adeles $\A$,
and  $A$ runs through the finite set of rank $n$ matrices in $M_n(\Z)/\{\pm I_n\}$
  satisfying the following conditions:
\begin{enumerate}
\item $r\t A^{-1} \in M_n(\Z)$.
\item $\t A\sigma_1 A = r \sigma_2$.
\end{enumerate}
\end{theorem}
\noindent{\em Remarks:} (1) In Appendix \ref{AB}, we give a quantitative version of the
  above for $\GSp(4)$.

(2) Haar measure is normalized as follows.  On the geometric side
  and in the definition of the Fourier coefficients \eqref{csigma}
  we take $\meas(S_n(\Zhat))=1$ and
  Lebesgue measure on the Euclidean space $S_n(\R)$.  The Satake transform
  is defined using $\meas(K_p)=1$.  The archimedean test function $f_\infty$
  depends (via $d_\k$) on an unspecified choice of measure on $\olG(\R)$.
  This choice materializes on the spectral side in $\|\varphi\|^{-2}$.
  The exact relationship between several natural choices of Haar
  measure on $\Sp_{2n}(\R)$ is computed in \cite[\S A]{PSS}.
\begin{proof}
  In view of the discussion in the previous two subsections,
  it just remains to prove that only finitely many matrices
   $A$ satisfy the given conditions.
For fixed $j$, let $d\in \R$ be the entry in the $j$-th row and $j$-th column of $r\sigma_2$.
Suppose $A\in M_n(\Z)$ satisfies condition 2, and let $v\in \Z^n$
  denote the $j$-th column of $A$.  Then ${}^tv \sigma_1 v=d$.
  We will show that there are only finitely many such $v$ (this is well-known),
  from which it follows that the set of $A$ is also finite since $j$ is
  arbitrary.

Because $\sigma_1$ is symmetric, there exists an orthogonal matrix $Q$ such that
  $\sigma_1={}^tQ\Lambda Q$, where $\Lambda=\text{diag}(\lambda_1,\ldots,\lambda_n)$
   is diagonal.  Furthermore the eigenvalues $\lambda_j$ are all positive
  since $\sigma_1$ is positive definite.
  It follows that the linear map $Q:\R^n\longrightarrow \R^n$ restricts to give an
  isometry between the sets
\[
X=\{v\in\R^n|\, {}^tv\sigma_1v=d\}
\]
and
\[Y=\{w\in \R^n|\, {}^tw\Lambda w=d\}.\]
Notice that $Y$ is the ellipsoid
\[\lambda_1x_1^2+\cdots+\lambda_nx_n^2=d,\]
which is compact.  Hence $X$ is also compact.  Since $\Z^n$ is discrete,
  it follows that there are only finitely many integer lattice points in $X$,
  as claimed.
\end{proof}

\section{Refinement of the geometric side}

Fix a matrix $A\in M_n(\Z)$ satisfying the hypotheses of Theorem \ref{A}.
The integral on the geometric side of \eqref{Ae} can be factorized as
\[I_A=\prod_{p\le \infty} I_{A,p}=\prod_{p\le \infty}
\int_{S_n(\Q_p)} f_p(\mat{I_n}SO{I_n} \mat AOO{r\t A^{-1}})
   \theta_{p}(\tr \sigma_1 S) dS.\]
We will see that for all $p\notin \Sb$, $I_{A,p}$ can be computed explicitly,
and is independent of $A$ when $\k$ is even.

\subsection{Archimedean integral}
When $p=\infty$, by Corollary \ref{matcoeffmain} in the Appendix,
\[I_{A,\infty}= \int_{S_n(\R)} f_\infty(\mat{I_n}SO{I_n} \mat AOO{r\t A^{-1}})
 e^{2\pi i\tr \sigma_1 S} dS \]
\[ = \int_{S_n(\R)} \frac{ d_\k 2^{n\k} r^{n\k/2}
  e^{2\pi i\tr \sigma_1 S} }{\det(A+r\t A^{-1}+i S r \t A^{-1})^{\k}}\,dS
\]
\begin{equation}\label{arch0}
=\frac{d_\k 2^{n\k} (\det A)^\k }{ r^{n\k/2}}
    \int_{S_n(\R)} \frac{
  e^{2\pi i\tr \sigma_1 S}}{\det(r^{-1} A \t A + I_n +i S)^{\k}}\, dS.
\end{equation}
We apply the following formula of Ingham (also found in a paper of Siegel).
Let
\begin{equation}\label{Gamman}
 \Gamma_n(a) = \pi^{n(n-1)/4} \prod_{j=1}^n \Gamma(a-\tfrac 12 (j-1)).
\end{equation}
 Then for $\delta > \frac{n-1}2$ and symmetric matrices $X_0,\Lambda > 0$,
\[ \int_{S_n(\R)} \frac{e^{\tr i \Lambda Y}}{\det(X_0 + i Y)^{\delta+(n+1)/2}} dY\]
\[  =
\frac {1} {i^{n(n+1)/2}2^{(n-1)n/2}}
  \frac{ (2\pi i)^{n(n+1)/2}
(\det \Lambda)^\delta}{\Gamma_n(\delta+(n+1)/2)} e^{-\tr \Lambda X_0}
\]
\begin{equation}\label{herz}
=\frac{2^{n}\pi^{n(n+1)/2}
(\det \Lambda)^\delta}{\Gamma_n(\delta+(n+1)/2)} e^{-\tr \Lambda X_0}
\end{equation}
  (cf. \cite[(1)]{I}, \cite[Hilfssatz 37, p. 585]{Si}; we have used the form given by
  Herz \cite[(1.2)]{H}).
In \cite{H}, the measure is $dZ= \prod_{j\le k}dz_{jk}$, where
   $Z=X_0+iY=(\eta_{jk} z_{jk})$ with $\eta_{jk}=1$ if $j=k$, and $1/2$ otherwise.
Thus
\[dZ= i^{n(n+1)/2}2^{(n-1)n/2} dY,\]
which explains the first factor in the above formula.

We evaluate \eqref{arch0} using \eqref{herz} with $\delta=\k-(n+1)/2>(n-1)/2$,
$X_0= I_n + r^{-1} A \t A $, and $\Lambda = 2\pi \sigma_1$.
By condition $(2)$ of Theorem \ref{A},
  $\det(A)= \pm r^{n/2} \left(\frac{\det \sigma_2}{\det \sigma_1}\right)^{1/2}$.
So \eqref{arch0} becomes
\[\frac{d_\k 2^{n\k}}{ r^{n\k/2}}\sgn(\det A)^\k r^{n\k/2}
  \left(\frac{\det \sigma_2}{\det \sigma_1}\right)^{\k/2}
2^{n}\pi^{n(n+1)/2}\frac{\det(2\pi\sigma_1)^{\k-(n+1)/2}}
  {\Gamma_n(\k)\,e^{2\pi\tr(\sigma_1(I_n+r^{-1}A\,\tA))}}.\]
We simplify the above using
 $\tr(\sigma_1r^{-1} A \t A) = \tr( r^{-1} \t A\sigma_1 A) = \tr \sigma_2$
for matrices $A$ as in Theorem \ref{A}.
The result is the following:
\begin{equation}\label{arch}
I_{A,\infty}=\sgn(\det A)^\k \,d_\k
  \left(\frac{\det \sigma_2}{\det \sigma_1}\right)^{\k/2}
\frac{(\det\sigma_1)^{\k-(n+1)/2}(4\pi)^{n\k}}{2^{n(n-1)/2}\Gamma_n(\k)
  \,e^{2\pi\tr(\sigma_1+\sigma_2)}}.
 \end{equation}
Observe that this is independent of $A$ when $\k$ is even.

\subsection{Nonarchimedean integrals: $p\notin \Sb$}\label{qnotp}
In this case, $r\in\Z_p^*$, and $A\in \GL_n(\Z_p)$.
  It follows that
\begin{align*}   f_p(\mat{I_n}SO{I_n} \mat AOO{r\t A^{-1}})\neq 0&\iff
\mat{I_n}SO{I_n} \mat AOO{r\t A^{-1}}\in K_p\\
&\iff S\in M_n(\Z_p).
\end{align*}
For such $S$, $\theta_p(\tr\sigma_1S)=1$, so we find that
\[I_{A,p}=
\int_{S_n(\Q_p)}
   f_p(\mat{I_n}SO{I_n} \mat AOO{r\t A^{-1}})
   \theta_{p}(\tr \sigma_1 S) dS=\meas(S_n(\Z_p))=1.\]
(Recall that we use the $N=1$ test function $f_1$ in \eqref{Ae},
 so that $f_p$ is the characteristic function of $Z_pK_p$ when $p\notin\Sb$.)

\subsection{Proof of Theorem \ref{S=1}}\label{S1}

When $\Sb=\emptyset$, we now have a completely explicit expression for
  the right-hand side of \eqref{Ae}. Since $r=1$, the sum over $A$ is nonzero
  only if $\t A\sigma_1A=\sigma_2$ for some $A\in \GL_n(\Z)/\{\pm I_n\}$.
  In particular, $(\det A)^2 = 1$, so $\det \sigma_1=\det\sigma_2$.
  If we let $I_A$ denote the summand indexed by $A$ in Theorem \ref{A},
  then by the above discussion,
\begin{equation}\label{arch2}
I_{A}=(\det A)^\k\,d_\k\,{(4\pi)^{n\k-n(n-1)/4}}
\frac{(\det\sigma_1)^{\k-(n+1)/2}}{\prod_{j=1}^n\Gamma(\k-\frac{j-1}2)}
e^{-2\pi\tr(\sigma_1+\sigma_2)}.
 \end{equation}

We wish to express the spectral side in classical terms. With notation as in \eqref{Hn},
let $S_\k(N)$ be the
  space of Siegel cusp forms $F$ satisfying
\[F(\g\cdot \mathfrak{Z})=j(\g,Z)^{\k}F(\mathfrak{Z})\]
  for all $\mathfrak{Z}\in\mathcal{H}_n$ and
\[\g\in \Gamma_0(N)=\{\smat ABCD\in \Sp_{2n}(\Z)|\, C\equiv O\mod N\},\]
where
\[j(g,\mathfrak{Z})=r(g)^{-n/2}\det(C\mathfrak{Z}+D)\qquad(g=\smat ABCD\in G(\R)).\]
Any $F\in S_\k(N)$ has a Fourier expansion
\begin{equation}\label{Fourier}
    F(\mathfrak{Z})=\sum_{\sigma\in\mathcal{R}_n^+}a_\sigma(F) e^{2\pi i \tr\sigma \mathfrak{Z}},
    \quad (\mathfrak{Z}\in\mathcal{H}_n).
\end{equation}
We normalize the Petersson/Maass scalar product on $S_\k(N)$ by
\begin{equation}\label{Pet}
\sg{F,H}=\frac1{\psi(N)}\int_{\Gamma_0(N)\bs\mathcal{H}_n}
    F(\mathfrak{Z})\ol{H(\mathfrak{Z})}(\det Y)^{\k-n-1}dX\,dY\quad (\mathfrak{Z}=X+iY),
\end{equation}
where $\psi(N)=[K_{\fin}:K_0(N)]=[\Gamma_0(1):\Gamma_0(N)]$.

We need to choose the quotient measure on $\olG(\Q)\bs\olG(\A)$ compatibly with the
  above.  For any $N\ge 1$, let $D_N\subset \mathcal{H}_n$ be a fundamental
  domain for $\Gamma_0(N)\bs \mathcal{H}_n$, identified with a subset of
  $\Sp_{2n}(\R)$ via
\[\mathfrak{Z}=X+iY\longleftrightarrow b_\mathfrak{Z}=\mat{Y^{1/2}}{XY^{-1/2}}{O}{Y^{-1/2}}.\]
  Then, as in \cite[Prop. 7.43]{KL} for example,
   we may define a quotient measure on $\olG(\Q)\bs\olG(\A)$ by
\[\int_{\olG(\Q)\bs\olG(\A)} h(g)dg
=\int_{D_NK_\infty\times K_0(N)}h(b_\mathfrak{Z}\,k_\infty\times k_{\fin})
  \frac{dX\,dY}{(\det Y)^{n+1}}dk_\infty dk_{\fin},\]
where the compact groups $K_\infty$ and $K_{\fin}$ each have total volume $1$.
This measure $dg$ is independent of the choice of $N$ since
\[\meas(D_N)\meas(K_0(N))=[\ol{\Gamma_0(1)}:\ol{\Gamma_0(N)}]\meas(D_1)\psi(N)^{-1}=\meas(D_1).\]
Taking $N=1$, a well-known computation of Siegel (\cite{Si2}) gives
\[\meas(\olG(\Q)\bs\olG(\A))=\meas(\Sp_{2n}(\Z)\bs\mathcal{H}_n)= 2\prod_{j=1}^n[(j-1)!\pi^{-j}
  \zeta(2j)].\]

Given $F\in S_\k(N)$, its adelic counterpart is the function $\varphi_F\in L^2_0$
  defined by
\begin{equation}\label{phiF}
\varphi_F(g)=F(g_\infty\cdot iI_n)j(g_\infty,iI_n)^{-\k}
\end{equation}
for $g=\g(g_\infty\times k)\in G(\A)=G(\Q)(G(\R)^+\times K_0(N))$.
  The well-definedness of $\varphi_F$ is a consequence of the fact that $G(\Q)\cap
  (G(\R)^+\times K_0(N))=\Gamma_0(N)$.
  With measures normalized as above,
  the map $F\mapsto \varphi_F$ defines a linear isometry from $S_\k(N)$ onto
  the subspace $A_\k(N)\subset L^2_0$ defined in \eqref{AkN}.
  This may be proven just as in
  \cite{Sa}.  (The latter paper works with the principal congruence subgroup
  of level $N$, but the Siegel parabolic case is just the same.)
The relationship between the adelic and classical Fourier coefficients
  is given by the following (see \eqref{csigma}).

\begin{proposition}\label{fourier}
For $F\in S_\k(N)$ and $\sigma\in S_n(\Q)$,
\begin{equation}
c_{\varphi_F}(\sigma)=\begin{cases} e^{-2\pi \tr\sigma}a_F(\sigma)
& \text{if }\sigma\in\mathcal{R}_n^+\\
0&\text{otherwise.}\end{cases}
\end{equation}
\end{proposition}

\begin{proof}
By strong approximation for the adeles,
\[S_n(\A)=S_n(\Q)+S_n(\R)\times S_n(\Zhat).\]
Decomposing $S=X_\Q+(X_\infty\times X_{\fin})\in S_n(\A)$ accordingly,
  by the right $K_0(N)$-invariance and left $G(\Q)$-invariance
  of $\varphi_F$ we have
\[\varphi_F(\mat{I_n}S{}{I_n})=\varphi_F(\mat{I_n}{X_\infty}{}{I_n}).\]
Also, it follows that
\[S_n(\Q)\bs S_n(\A)=S_n(\Z)\bs (S_n(\R)\times S_n(\Zhat)).\]
  If $D$ is any fundamental domain for $S_n(\Z)\bs S_n(\R)$,
 then $D\times S_n(\Zhat)$ is a fundamental domain for
  $S_n(\Z)\bs (S_n(\R)\times S_n(\Zhat))$ (\cite[Theorem 7.40]{KL}).
Therefore \eqref{csigma} becomes
\[c_{\varphi_F}(\sigma)=\int_{S_n(\Z)\bs S_n(\R)}\varphi_F(\mat{I_n}{X_\infty}{}{I_n})
    \theta_\infty(-\tr(\sigma X_\infty))dX_\infty\]
    \[\times
  \int_{S_n(\Zhat)}\theta_{\fin}(-\tr(\sigma X_{\fin}))dX_{\fin}.\]
  Write $X_{\fin}=(X_{ij})$ for $X_{ij}\in \Af$.
  If we write $\sigma=(b_{ij}\sigma_{ij})$, where $b_{ij}$ is equal to $1$ or $1/2$ according
  to whether or not $i= j$,
  then $\tr(\sigma X_{\fin})=\sum_{i,j} \sigma_{ij}X_{ij}$. Note that $\sigma_{ij}\in\Z$
  for all $i,j$ if and only if $\sigma$ is half-integral, i.e., $\sigma\in\mathcal{R}_n$.
  If this condition does not hold, then
  the second integral vanishes.

Assuming $\sigma\in\mathcal{R}_n$, using \eqref{thetadef} and \eqref{phiF} we have
\[c_{\varphi_F}(\sigma)=\int_{S_n(\Z)\bs S_n(\R)}F(X+iI_n)
  e^{-2\pi i\tr(\sigma X)}dX\]
\[=e^{-2\pi\tr \sigma}\int_{S_n(\Z)\bs S_n(\R)}F(X+iI_n)
  e^{-2\pi i\tr(\sigma(X+iI_n ))}dX\]
\[=e^{-2\pi\tr\sigma}a_F(\sigma).\qedhere\]
\end{proof}

Letting $\mathcal{B}_\k(N)$ be an orthogonal basis for $S_\k(N)$,
in the special case $S=\emptyset$ and $r=1$, our main formula \eqref{A} becomes
\begin{align}
\lim_{N\to\infty}\frac1{\psi(N)}&\sum_{F\in \mathcal{B}_\k(N)}
  \frac{a_{\sigma_1}(F) \ol{a_{\sigma_2}(F)}} {\|F\|^2}\\
\notag &=\sum_{A\in\GL_n(\Z)/\pm I_n\atop{\t A\sigma_1A=\sigma_2}}
  (\det A)^\k\,d_\k\,{(4\pi)^{n\k-n(n-1)/4}}
\frac{(\det\sigma_1)^{\k-(n+1)/2}}{\prod_{j=1}^n\Gamma(\k-\frac{j-1}2)}.
 \end{align}

For the formal degree, we will take the
  classical measure on $\Sp(2n)$ corresponding to the measure in \eqref{Pet}.
  So as shown in \cite{PSS},
\[d_\k = \frac1{2^{n}(4\pi)^{n(n+1)/2}}
\prod_{1\le i\le j\le n}(2\k-(i+j)).\]
(See the final remark after Proposition \ref{dkprop} below.)
   We can simplify $\frac{d_\k}{\Gamma_n(\k)}$ using the following.

\begin{lemma} For any integers $n\ge 1$ and $\k>2n$,
\[\frac{\prod_{1\le i\le j\le n}(2\k-(i+j))}
{\prod_{\ell=1}^n\Gamma(\k-\frac{\ell-1}2)}
  =\frac{2^{n(n+1)/2}}{\prod_{j=1}^n \Gamma(\k-\frac{n+j}2)}.\]
\end{lemma}

\begin{proof}
  When $n=1$, by the functional equation for the Gamma function we have
\[\frac{(2\k-2)}{\Gamma(\k)}=
\frac{2(\k-1)}{\Gamma(\k)}=\frac2{\Gamma(\k-1)},\]
  as needed.  Given $n\ge 2$, suppose the formula holds for $n-1$.  Then
\[\frac{\prod_{1\le i\le j\le n}(2\k-(i+j))}
{\prod_{\ell=1}^n\Gamma(\k-\frac{\ell-1}2)}
=
\frac{\prod_{1\le i\le j\le n-1}(2\k-(i+j))}
{\prod_{\ell=1}^{n-1}\Gamma(\k-\frac{\ell-1}2)}
\frac{\prod_{1\le i\le n}
  (2\k-(i+n))}{\Gamma(\k-\frac{n-1}2)}
\]
\[
=\frac{2^{n(n-1)/2}}
  {\prod_{j=1}^{n-1}\Gamma(\k-\frac{n-1+j}2)}
 \frac{\prod_{1\le i\le n}(2\k-(i+n)) }{\Gamma(\k-\frac{n-1}2)}
\]
\[
=2^{n(n+1)/2}\prod_{i=1}^n\frac{(\k-\frac{i+n}2)}{\Gamma(\k-\frac{i+n}2+1)}
  =\frac{2^{n(n+1)/2}}{\prod_{i=1}^n \Gamma(\k-\frac{n+i}2)}.
\qedhere\]
\end{proof}

It now follows that
\begin{align}\notag\lim_{N\to\infty}&\frac1{\psi(N)}\sum_{F\in \mathcal{B}_\k(N)}
    \frac{a_{\sigma_1}(F) \ol{a_{\sigma_2}(F)}} {\|F\|^2}\\
    &=\frac{(\det\sigma_1)^{\k-(n+1)/2}}
{\pi^{n(n-1)/4}(4\pi)^{n(n+1)/2-n\k}\prod_{j=1}^n\Gamma(\k-\frac{n+j}2)}
  \delta_\k(\sigma_1,\sigma_2),
 \end{align}
for $\delta_\k(\sigma_1,\sigma_2)$ given in \eqref{deltak}.
  This proves Theorem \ref{S=1}.

\subsection{Nonarchimedean integrals: $p\in \Sb$}\label{pinS}

This case is more difficult.
Our goal is to compute (or bound) the local integral
\[I_{A,p}=I_{A,p}(f_p)=\int_{S_n(\Q_p)} f_p(\mat{I_n}SO{I_n} \mat AOO{r\t A^{-1}})
   \theta_{p}(\tr \sigma_1 S) dS.\]
To simplify the computation, we may
  essentially reduce to the case where $A$ is diagonal, as follows.
By the elementary divisors theorem, there exist $U,V\in \GL_n(\Z)$
and a diagonal matrix
\[D=\diag(d_1,\ldots,d_n)\]
with positive integer entries satisfying $d_1|d_2|\cdots|d_n$,
  such that
\begin{equation}\label{UDV}
A=UDV.
\end{equation}

\begin{proposition}
For $A$ as above,
\[I_{A,p}=\int_{S_n(\Q_p)}f_p(\mat{I_n}SO{I_n}\mat{D}OO{rD^{-1}})
  \theta_p(\tr\sigma_1' S) dS,\]
where
\[\sigma_1'=\sigma_U={}^tU\sigma_1U.\]
\end{proposition}

\begin{proof}
By definition,
\[I_{A,p}=
\int_{S_n(\Q_p)} f_p(\mat{I_n}SO{I_n} \mat AOO{r\t A^{-1}})
   \theta_{p}(\tr \sigma_1 S) dS\]
   \[ = \int_{S_n(\Q_p)} f_p (\mat{UDV}{rS\,{}^t U^{-1}  D^{-1} \,{}^t V^{-1}}{O}
  {r\, {}^tU^{-1}
    D^{-1} \,{}^t V^{-1}}) \theta_p(\tr \sigma_1 S) dS.\]
Because $f_p$ is bi-$K_p$-invariant, we are free to multiply its
  argument on the left by $\mat{U^{-1}}{}{}{{}^t U} \in K_p$
  and on the right by $\mat{V^{-1}}{}{}{{}^t V} \in K_p$.
  This gives
   \[I_{A,p} = \int_{S_n(\Q_p)} f_p (\mat{D}{r(U^{-1} S\,{}^t U^{-1})  D^{-1}}{O}{r  D^{-1}})
   \theta_p(\tr \sigma_1 S) dS. \]
   Let $S'=U^{-1} S \,{}^t U^{-1}$. Then $dS'=dS$
   since $S \mapsto S'$ is an isomorphism mapping
   $S(\Zhat)$ to $S(\Zhat)$.
   Hence the above is
   \[ = \int_{S_n(\Q_p)} f_p (\mat{D}{rS'  D^{-1}}{O}{r  D^{-1}} )
  \theta_p(\tr \sigma_1 US'\,{}^t U) dS'.\]
   Now using $\tr \sigma_1 US\,{}^t U = \tr {}^t U \sigma_1 US$, we find
   \[I_{A,p} = \int_{S_n(\Q_p)} f_p (\mat{D}{rS  D^{-1}}{O}{r  D^{-1}} )
  \theta_p(\tr \sigma_1'S) dS. \qedhere\]
\end{proof}

For the purpose of computing the above local integral,
by the $K_p$-invariance of $f_p$,
 we may assume that each
  diagonal entry of $D$ is a power of $p$.
    Thus, we take
\[D=\diag(p^{\alpha_1},\ldots,p^{\alpha_n}),\]
for
\[0\le \alpha_1\le\cdots\le \alpha_n.\]
For $x \in \Q_p$, we define $\ord_p(x)=n$ if $x=p^n z$ with $z$ a unit.
So $r_p=\ord_p(r)$, for example.

\begin{proposition}\label{D}
With the above notation, suppose $I_{A,p}\neq 0$.  Then each $\alpha_j\le r_p$.
Under the additional assumption that $\ord_p(\det\sigma_1)=\ord_p(\det\sigma_2)$,
we further have
\[\alpha_1+\cdots+\alpha_n=\frac{nr_p}2.\]
\end{proposition}

\begin{proof}
As in Theorem \ref{A}, we are assuming that $r\,{}^t\!A^{-1}\in M_n(\Z)$. It follows that likewise
  $rD^{-1}\in M_n(\Z)$, and hence $\alpha_j\le r_p$
  for each $j$.
  Under the additional assumption, taking determinants in the relation
  ${}^t\!A\sigma_1 A=r\sigma_2$ gives $p^{2(\alpha_1+\cdots+\alpha_n)}=p^{nr_p}$, and the
  last assertion follows.
\end{proof}

In principle, one can now compute the integral by applying Proposition \ref{fp}
 and considering various cases to obtain certain exponential sums.
 We will discuss this process in more detail for the special case of $\GSp(4)$
  in Section \ref{n2}.
 It should be evident from this special case that the general case is very complicated.

  We conclude the present section by giving a trivial bound for $I_{A,p}$.

\begin{proposition}\label{trivbound}
With notation as above,
\begin{equation}\label{trivboundeq}
|I_{A,p}|\le \prod_{j=1}^np^{j(r_p-\alpha_j)}=p^{\frac{n(n+1)}2r_p-(\alpha_1
  +2\alpha_2+\cdots+n\alpha_n)}.
\end{equation}
\end{proposition}
\begin{proof}
By Proposition \ref{gamma},
   \[f_p (\mat{D}{rS  D^{-1}}{O}{r  D^{-1}} )\neq 0\iff
   \mat{D}{rS  D^{-1}}{O}{r  D^{-1}} \in C_p.\]
Because $C_p\subset M_{2n}(\Z_p)$, we see that
\[|I_{A,p}|\le \meas\{S\in S_n(\Q_p)|\, p^{r_p} SD^{-1}\in M_n(\Z_p)\}.\]
Writing $S=(s_{ij})$,
\begin{align*}
|I_{A,p}|&\le \meas\{S\in S_n(\Q_p)|\, s_{ij} p^{r_p-\alpha_j} \in \Z_p
   \text{ for all }i, j\}\\
&=\meas\{S\in S_n(\Q_p)|\,
    s_{ij} \in p^{-(r_p-\alpha_j)} \Z_p\text{ for all }1\le i\le j\le n\}
\end{align*}
since $s_{ij}=s_{ji}$.  Hence,
\[   |I_{A,p}| \le \prod_{j=1}^n\prod_{i \le j} \meas(p^{-(r_p-\alpha_j)} \Z_p)
  = \prod_{j=1}^n\prod_{i \le j} p^{r_p-\alpha_j}
 = \prod_{j=1}^{n} p^{j(r_p-\alpha_j)}.
\qedhere\]
\end{proof}

\section{Weighted equidistribution of Satake parameters}

Let $\Ghat=\Spin(2n+1,\C)$ be the complex dual group of $\olG=\PGSp_{2n}$.\footnote[2]{For dual
groups, we always take the ground field to be $\C$ unless specified otherwise.}
Since we are assuming trivial central character, the Satake parameters
  $t_{\pi_p}$ belong to the maximal torus
  $\That$ of $\Ghat$.
For our fixed finite set $\Sb$ of primes, let
\[\X=(\That/W)^{|\Sb|}.\]
    Let $\Pi_\k(N)$ be the set of cuspidal representations of weight $\k$ and level $N$
    defined in \S \ref{siegelmod}.  Each $\pi\in \Pi_{\k}(N)$ determines a point
\begin{equation}\label{tpi}
t_\pi=(t_{\pi_p})_{p\in \Sb}\in \X.
\end{equation}
By Shin's theorem, the points $t_\pi$ become equidistributed relative to the
  Plancherel product measure on $\X$ as $N\to\infty$.
Here we investigate their distribution with certain prescribed harmonic weights.

\begin{proposition} \label{Omega}
There exists a compact subset $\Omega\subset \X$
  such that $t_\pi\in \Omega$ for all $\pi\in \Pi_{\k}(N)$.
\end{proposition}
\begin{proof}
\cite[Theorem XI.3.3]{BW}.
\end{proof}

\subsection{Preliminary result}

  For $\sigma_1,\sigma_2\in \mathcal{R}^+_n$, define the weight
\[w_\pi(\sigma_1,\sigma_2)=
  \sum_{\varphi\in E_\k(\pi,N)}\frac{c_{\sigma_1}(\varphi)
  \ol{c_{\sigma_2}(\varphi)}} {\|\varphi\|^2}.\]
We will show in this section that the Satake parameters $t_{\pi}$,
  weighted by $w_\pi(\sigma,\sigma)$,
  have a uniform distribution relative to a certain Radon measure
  in the limit as $N\to\infty$.

The following is essentially a restatement of Theorem \ref{S=1} (see \S \ref{S1}).

\begin{lemma}\label{weights}
Let $\ds c_{n\k\sigma_1}=
\frac{(\det\sigma_1)^{\k-(n+1)/2}}
{\pi^{n(n-1)/4}(4\pi)^{n(n+1)/2-n\k}\prod_{j=1}^n\Gamma(\k-\frac{n+j}2)}$.
  Then
\[\lim_{N\to\infty} \frac1{\psi(N)}\sum_{\pi\in \Pi_{\k}(N)}w_\pi(\sigma_1,\sigma_2)
  = \delta_\k(\sigma_1,\sigma_2)c_{n\k\sigma_1},\]
  where $\delta_\k(\sigma_1,\sigma_2)$ is defined in \eqref{deltak}.
In particular, if $\sigma_1=\sigma_2 =\sigma\in \mathcal{R}^+_n$ satisfies
  $\delta_\k(\sigma,\sigma)>0$ (e.g. if $\k$ is even),
 then setting $w_\pi=w_\pi(\sigma,\sigma)$ we have
\begin{equation}\label{wpi}
  0<\lim_{N\to\infty}\frac1{\psi(N)}{\sum_{\pi\in \Pi_{\k}(N)}w_\pi}<\infty.
\end{equation}
\end{lemma}

\noindent{\em Remark:}
  If $\k$ is even,
  then $\delta_\k(\sigma,\sigma)>0$ since $A=I_n$ satisfies $\t A\sigma A=\sigma$.
  Hence in this case we always have nonvanishing in \eqref{wpi}.
  This nonvanishing is crucial in what follows.\\

For the compact space $\Omega$ in Proposition \ref{Omega},
  let
\begin{equation}\label{VS}
V_\mathcal{S}\subset C(\Omega)
\end{equation}
 denote the subspace
  consisting of all restrictions $F|_{\Omega}$ of functions
  $F=\prod_{p\in \Sb}\mathcal{S}f_p$ in the image
   $\prod_{p\in \Sb} \C[X^*(\That)]^W$
 of the $\Sb$-product Satake transform. (See \eqref{S} and \eqref{3cong}.)
  For fixed $\k>2n$ and $\sigma\in\mathcal{R}_n^+$ for which $\delta_\k(\sigma,\sigma)>0$,
  let $w_\pi=w_\pi(\sigma,\sigma)$ as above.
  Then we may define a linear functional $\mathcal{L}=\mathcal{L}_{\sigma,\k,\Sb}$ on $V_{\mathcal{S}}$ by
\begin{equation}\label{L}
\mathcal{L}(F) = \limN\frac{\sum_{\pi\in \Pi_{\k}(N)}w_\pi F(t_\pi)}
  {\sum_{\pi\in \Pi_{\k}(N)}w_\pi}.
\end{equation}
By Lemma \ref{weights} and Theorem \ref{A}, the limit exists and is finite.
Endowing $C(\Omega)$ with the $L^\infty$ norm, it contains $V_\mathcal{S}$
  as a dense subalgebra by the Stone-Weierstrass Theorem \cite[p. 122]{Ru}.
  (The latter algebra evidently separates points, and it is closed under complex
  conjugation by Proposition \ref{Sfbar}.)
Because $V_{\mathcal{S}}$ is dense in $C(\Omega)$,
the right-hand side of \eqref{L} exists for $F \in C(\Omega)$
(for details, see e.g. \cite{KL}, pages 358-359).
Moreover it is clear from \eqref{L} that $|\mathcal{L}(F)|
  \le \|F\|_\infty$ for all $F \in C(\Omega)$, so $\mathcal{L}$ is bounded.
  By the Riesz representation
  theorem, there exists a unique Radon measure $\mu=\mu_{\sigma,\k,\Sb}$
  on $\Omega$ such that
\begin{equation}\label{mudef}
\mathcal{L}(F)=\int_{\Omega}F\,d\mu
\end{equation}
for all $F \in C(\Omega)$.
It is clear from \eqref{L} that $\mathcal{L}(1)=1$, so $\mu$ is a probability measure.
This proves the following.

\begin{theorem}\label{equi}
  Let $\k>2n$, and let $\sigma$ be a
 symmetric positive-definite half-integral matrix for which $\delta_\k(\sigma,\sigma)> 0$.
  (This is automatic, for example, if $\k$ is even.)
 Then
  the Satake parameters
  $(t_{\pi})_{\pi\in \Pi_\k(N)}$ of \eqref{tpi}, when weighted by $w_\pi(\sigma,\sigma)$,
    become equidistributed in the space $\Omega$ of Proposition \ref{Omega}
    with respect to the above probability measure $\mu_{\sigma,\k,\Sb}$ in
  the limit as $N\to\infty$ along integers
  coprime to $\Sb$.
\end{theorem}

Of course, one would like to know more about the measure $\mu$, for example
  whether it is supported on the tempered spectrum.
(Recall that an unramified representation $\pi_p$ of $G(\Q_p)$
  is tempered if and only if its Satake parameter $t_{\pi_p}$ lies in a compact
  subgroup of $\That$.)
  We will pursue this question by relating $\mu$ to the Sato-Tate measure, which
  is supported in a compact subtorus of $\That$.  See Theorem \ref{musum} below.

\subsection{Relating two measures}\label{twomeas}

Generally, suppose $\eta$ is a Radon measure on $\Omega$, and
  $\{R_\lambda\}_{\lambda\in\Lambda}$ is a set of continuous functions
  forming an orthonormal basis
  for $L^2(\Omega,\eta)$ which also spans an $L^\infty$-dense subspace of $C(\Omega)$.
  Then the measure $\mu$ in \eqref{mudef} can be expressed as
\begin{equation}\label{mueta}
d\mu(t)=\sum_\lambda \mathcal{L}(R_\lambda)\ol{R_\lambda}(t)d\eta(t),
\end{equation}
  provided the sum is uniformly absolutely convergent on $\Omega$.
  Indeed, for all $\alpha\in \Lambda$,
\[\int_{\Omega} R_\alpha(t) \sum_\lambda \mathcal{L}(R_\lambda)\ol{R_\lambda}(t)d\eta
  =\sum_\lambda\mathcal{L}(R_\lambda)\sg{R_\alpha,R_\lambda}_\eta =\mathcal{L}(R_\alpha)
  =\int_{\Omega}R_\alpha d\mu,\]
and by linearity and density of the span of the $R_\alpha$,
   the above holds as well for all functions in $C(\Omega)$.

\subsection{The Sato-Tate measure}\label{STmeas}

Fix a maximal compact subgroup $\Hhat\subset \Ghat$ with maximal torus
  \[\That_c=\Hhat\cap \That=\That(\C)^1.\]
This is the maximal compact subtorus of $\That$.
  Let $dh$ denote the Haar measure on $\Hhat$ of total volume $1$.
  Because every conjugacy class in $\widehat{H}$ contains exactly one Weyl orbit of $\Tc$,
  the measure $dh$ induces a quotient measure $\mu_{ST}$ on the space $\Tc/W$.
We extend $\mu_{ST}$ to $\That/W$ by taking it to be zero on the complement
  of $\Tc/W$.  This is the Sato-Tate measure.
  In more detail, for $f\in C(\That_c/W)$, we may identify $f$ with a
  class function on $\widehat{H}$, and
\begin{equation}\label{wif}
 \int_{\Tc/W} f(t) d\mu_{ST}(t)= \int_{\widehat{H}} f(h)dh.
\end{equation}
  By the Weyl integration formula, the measure is given explicitly by
\begin{equation}\label{must}
d\mu_{ST}(t) = \Bigl|\det(\operatorname{Ad}(t^{-1}-I)|_{\Lie(\Hhat)/\Lie(\That_c)})\Bigr|dt,
\end{equation}
where $dt$ is the Haar measure giving $\Tc$ volume $1$.
An alternative expression for it is given in \eqref{ST2} below.

Fix a set $\Phi^+$ of positive roots in the root system attached to $\widehat{H}$ and
  $\That_c$.
We shall identify $X^*(\That)$ and $X^*(\That_c)$.
  By the theorem of the highest weight, the irreducible representations
  $\pi_\lambda$ of $\widehat{H}$ are in one-to-one
  correspondence with the elements $\lambda\in \mathcal{C}^+$, where
  $\mathcal{C}^+$ is the positive Weyl chamber of
  $X^*(\That_c)=X^*(\That)\cong X_*(\olT)$ given in \eqref{C+}.
  Let
\begin{equation}\label{Fldef}
F_\lambda=\tr \pi_\lambda
\end{equation}
  denote the trace of $\pi_\lambda$. It is a class function on $\widehat{H}$,
  so we may view it as a function on $\That_c/W$.
  By the Peter-Weyl theorem, the set $\{F_\lambda|\, \lambda\in \mathcal{C}^+\}$
  is an orthonormal basis for the space of $L^2$ class functions
  on $\widehat{H}$ (relative to the measure $dh$).
  In particular, by \eqref{wif} we have
\begin{equation}\label{onb}
\int_{\That(\C)/W} F_\lambda(t) \ol{F_\mu(t)} \,d\mu_{ST}
   =\delta_{\lambda,\mu}
\end{equation}
(for the Kronecker $\delta$).
  Here, the domain of $F_\lambda$ is extended from $\Tc$
  to $\That(\C)$
  by viewing $F_\lambda$ as a sum
\[F_\lambda = \sum_{\mu\in X^*(\That)} m_\lambda(\mu)[\mu]
  \in \C[X^*(\That)]^W.\]
The orthogonality \eqref{onb} can also be proved using \eqref{Weyl} and \eqref{ST2}
  below.

 We shall need the fact that the set $\{F_\lambda|\, \lambda\in \mathcal{C}^+\}$
  spans $\C[X^*(\That)]^W$
  (see \cite[Theorem 23.24]{FH}, using the fact that
  $\Lambda= X^*(\That)$ since $\widehat{\olG}=\Spin(2n+1)$ is simply connected).
  By \eqref{3cong}, this space coincides with the image of the local Satake transform.

Given a tuple $\lu=(\lambda_p)\in \prod_{p\in\Sb}\mathcal{C}^+$, we let
\[F_{\lu}=\prod_{p\in\Sb}F_{\lambda_p}\in \prod_{p\in\Sb}\C[X^*(\That)]^W.\]
 Viewing the $F_{\lu}$ as functions on $\Omega$ (by restriction), they
  span the space $V_{\mathcal{S}}$ of \eqref{VS}.  This follows from the above discussion.

\subsection{Relation between $\mu$ and $\mu_{ST}$}

We continue to assume that $\delta_\k(\sigma,\sigma)>0$, so that
for $F_{\lu}$ as above, we may consider $\mathcal{L}(F_{\lu})$ as in \eqref{L}.

\begin{theorem} \label{musum}
Let $\rho\in X^*(\olT)=X_*(\That)$ be half the sum of the positive roots,
   as in \eqref{rho}.
Suppose that there exits $\e>0$ such that for all tuples $\lu$ as above,
\begin{equation}\label{conj}
\mathcal{L}(\prod_{p\in\Sb}\mathcal{S}(c_{\lambda_p})) \ll_\e \prod_{p\in\Sb}
  p^{(1-\e)\sg{\rho,\lambda_p}},
\end{equation}
where $c_{\lambda_p}$ is the characteristic function of $K_p\lambda_p(p)K_p$.
Then the measure $\mu$ defined in \eqref{mudef} is given by
\[d\mu(t)=\sum_{\lu}
  \mathcal{L}(F_{\lu})\ol{F_{\lu}(t)}\,d\mu_{\Sb}(t),\]
where $\mu_{\Sb}=\prod_{p\in\Sb}\mu_{ST}$ is the product measure on $\X$,
and the above sum converges absolutely and uniformly on $\Omega$.
\end{theorem}

\noindent{\em Remarks:}
(1) Hypothesis \eqref{conj} would follow from (a) adequate
  bounds on the number of matrices $A$ satisfying the conditions of Theorem \ref{A},
  and (b) adequate bounds for the local geometric integrals $I_{A,p}(c_{\lambda_p})$
  for $p\in\Sb$.  (See \eqref{Lp}.)
  In Section \ref{n2}, we will carry this out and prove Hypothesis \eqref{conj}
  in the special case $n=2$, as an application of Theorem \ref{A}.

(2) It is not clear to us whether there is a closed form
  expression for the measure.

\begin{proof}
In Section \ref{STmeas}, we saw that the set $\{F_\lambda|\,\lambda\in\mathcal{C}^+\}$
  is an orthonormal basis for $L^2(\Omega,\mu_{ST})$.  Furthermore, it is dense
  in $V_\mathcal{S}$, which in turn is dense in $C(\Omega)$ as discussed before Theorem \ref{equi}.
Hence, by the discussion in \S \ref{twomeas}, it suffices to prove that
  the given series is uniformly convergent under Hypothesis \eqref{conj}.

To ease the notation, we will first assume that $\Sb=\{p\}$ consists of just one
  prime.
For any weight $\lambda\in X^*(\That)$, define
\[A_\lambda = \sum_{w\in W} (\sgn w)w(\lambda)\in \C[X^*(\That)].\]
For $t \in \That_c$,
\[  |A_\lambda(t)|  \le \left| \sum_{w\in W} (\sgn w) w(\lambda)(t) \right|
   \le \sum_{w\in W} | w(\lambda)(t)| =  |W|.\]
By the Weyl character formula (\cite[Theorem 24.2]{FH}),
\begin{equation}\label{Weyl}
F_\lambda = \frac{A_{\lambda+\rho}}{A_\rho}.
\end{equation}
It is well-known that
\begin{equation}\label{ST2}
d\mu_{ST}(t)= |A_\rho(t)|^2 dt.
\end{equation}
(For example, compare (25.6) of \cite{Bu} ((22.7) in the 2nd edition) with
  Lemma 24.3 of \cite{FH}).

Therefore, for $t\in \That_c$,
we need to prove the convergence of
\[  \sum_{\lambda\in \mathcal{C}^+} |\mathcal{L}(F_\lambda) \ol{F_\lambda(t)}||A_\rho(t)|^2
=
\sum_{\lambda\in \mathcal{C}^+} |\mathcal{L}(F_\lambda) \ol{A_{\lambda+\rho}(t)}
  A_\rho(t)|
  \le |W|^2\sum_{\lambda\in \mathcal{C}^+} |\mathcal{L}(F_\lambda)|.\]

Next, we need to relate $F_\lambda$ to the functions $\mathcal{S}(c_\mu)$
  in order to make use of Hypothesis \eqref{conj}.  This is achieved by
  the following formula of Kato and Lusztig, which holds in any split reductive
  $p$-adic group (\cite[Theorem 7.8.1]{HKP}; see also
  \cite[(3.12) and Proposition 4.4]{Gr}):
\[F_\lambda = p^{-\sg{\lambda,\rho}}
  \sum_{\mu\le\lambda} P_{\mu,\lambda}(p) \mathcal{S}(c_\mu).\]
Here, $\mu$ belongs to $\mathcal{C}^+$, and $P_{\mu,\lambda}$ is the Kazhdan-Lusztig polynomial
\[P_{\mu,\lambda}(p)=p^{\sg{\lambda-\mu,\rho}}\sum_{w\in W}\sgn(w)
  \widehat{P}(w(\lambda+\rho^\vee)-(\mu+\rho^\vee)),\]
where
\[\widehat{P}(\mu)=\sum_{\mu=\sum n(\alpha^\vee)\alpha^\vee}p^{-\sum n(\alpha^\vee)}\ge 0\]
encodes the number of expressions of $\mu$ as a linear combination of positive co-roots
  with coefficients $n(\alpha^\vee)\ge 0$.
We note that $P_{\lambda,\lambda}(p)=1$, \cite[(4.5)]{Gr}.

Therefore, the quantity we need to bound is
\[\sum_{\lambda\in \mathcal{C}^+} |\mathcal{L}(F_\lambda)|
  = \sum_{\lambda \in \mathcal{C}^+} \Bigl| p^{-\sg{\lambda,\rho}}
  \sum_{\mu\le \lambda}P_{\mu,\lambda}(p)
  \mathcal{L}(\mathcal{S}(c_\mu))\Bigr|\]
\begin{equation}\label{mid}
  \le \sum_{\mu \in \mathcal{C}^+} p^{-\sg{\mu,\rho}}
  |\mathcal{L}(\mathcal{S}(c_\mu))| \sum_{\lambda\ge \mu}
\sum_{w\in W}\widehat{P}(w(\lambda+\rho^\vee)-(\mu+\rho^\vee)).
\end{equation}
We claim that for $w\in W$,
\begin{equation}\label{psum}
\sum_{\lambda\ge\mu}\widehat{P}(w(\lambda+\rho^\vee)-(\mu+\rho^\vee))\le 2^{d^+},
\end{equation}
where $d^+$ is the number of positive co-roots. Indeed, the left-hand side
  of \eqref{psum} is
\begin{equation}\label{psum2}
\sum_{\lambda\ge \mu}\hskip .15cm\sum_{\sum n(\alpha^\vee)\alpha^\vee} p^{-\sum n(\alpha^\vee)},
\end{equation}
where the inner sum is extended over all expressions of the form
\[w(\lambda+\rho^\vee)-(\mu+\rho^\vee)
  =\sum_{\alpha^\vee} n(\alpha^\vee)\alpha^\vee\]
 with $n(\alpha^\vee)\ge 0$ and $\alpha^\vee$
  positive co-roots.  The above expression is equivalent to
\begin{equation}\label{lambda}
\lambda = w^{-1}(\sum_{\alpha^\vee} n(\alpha^\vee)\alpha^\vee +(\mu+\rho^\vee))
  -\rho^\vee.
\end{equation}
Thus we may exchange the order of summation in \eqref{psum2}, so the left-hand side of
  \eqref{psum} is equal to
\[ \sideset{}{{}^*}\sum_{\sum n(\alpha^\vee)\alpha^\vee} \, p^{-\sum n(\alpha^\vee)},\]
where the $*$ indicates that we consider only those expressions for which
  the right-hand side of \eqref{lambda} is $\ge \mu$.
The above is of course bounded by the sum over {\em all} nonnegative linear
  combinations of positive co-roots
\[ \sum_{\sum n(\alpha^\vee)\alpha^\vee} \, p^{-\sum n(\alpha^\vee)}
  =\prod_{\alpha^\vee} \sum_{n(\alpha^\vee)=0}^\infty p^{-n(\alpha^\vee)}
  \le \prod_{\alpha^\vee} 2,\]
proving the claim \eqref{psum}.

Combining \eqref{mid} and \eqref{psum}, it follows that
\[  \sum_{\lambda\in \mathcal{C}^+} |\mathcal{L}(F_\lambda)|
  \ll |W|\sum_{\mu\in \mathcal{C}^+} p^{-\sg{\mu,\rho}}|\mathcal{L}(\mathcal{S}(c_\mu))|.\]
Using the given bound \eqref{conj}, the above is
\[ \ll\sum_{\mu \in \mathcal{C}^+}  p^{-\varepsilon \left<\mu,\rho\right>}.\]
There exists a finite set $\{\mu_1, \ldots, \mu_\ell\}\subset \mathcal{C}^+$ such that
$ \mathcal{C}^+ \subset \{\sum_{i=1}^\ell a_i \mu_i \,|\, 0 \le a_i \in \mathbf Z \}$.
Writing $\mu = \sum_{i=1}^\ell a_i \mu_i$,
the above is
\[ \le \prod_{i=1}^\ell \left( \sum_{a_i=0}^\infty p^{-\varepsilon
  \left<\mu_i,\rho\right>a_i} \right) <\infty.\]
This completes the proof when $\Sb=\{p\}$.

The general case is proven in the same way, using
\[\sum_{\lu}|\mathcal{L}(F_{\lu})|=\sum_{\lu}\left|\sum_{{\underline \mu}\le \lu}
  \left(\prod_{p\in\Sb}p^{-\sg{\lambda_p,\rho}}P_{\mu_p,\lambda_p}(p)\right)\mathcal{L}
  (\prod_{p\in\Sb}\mathcal{S}(c_{\mu_p}))\right|\]
\[\le\sum_{\underline\mu}|\mathcal{L}(\prod_{p\in\Sb}\mathcal{S}(c_{\mu_p}))|\prod_{p\in\Sb}
  \sum_{\lambda_p\ge \mu_p}p^{-\sg{\lambda_p,\rho}}|P_{\mu_p,\lambda_p}(p)|\]
\[\ll \sum_{\underline\mu}|\mathcal{L}(\prod_{p\in\Sb}\mathcal{S}(c_{\mu_p}))|\prod_{p\in\Sb}
  p^{-\sg{\mu_p,\rho}}.\]
  Using Hypothesis \ref{conj}, one shows as before that this is finite.
\end{proof}

\begin{corollary}\label{pinf}
 Write $\mu=\mu_p$ for the measure on $\That/W$ defined in \eqref{mudef}
  when $\Sb=\{p\}$.  Then under Hypothesis \ref{conj},
\[
\lim_{p \to \infty} d\mu_p(t)
 =d\mu_{ST}(t).
\]
\end{corollary}

\begin{proof}
Let $\mathbf 0 \in \mathcal{C}^+$ denote the element corresponding to the zero
  vector in $\Z^{n+1}$.
As in the proof of the previous proposition,
\[
\sum_{\lambda \in \mathcal{C}^+ - \{\mathbf 0\}} |\mathcal{L}(F_{\lambda})\ol{F_{\lambda}(t)}|
|A_{\rho} (t)|^2
\ll  \sum_{\mu \in \mathcal{C}^+ - \{\mathbf 0\}} p^{-\e\left<\mu,\rho\right>}.
\]
 Noting that $\sg{\mu,\rho}>0$ when $\mu\neq \mathbf0$,
the right-hand side tends to $0$ as $p$ goes to $\infty$.
Thus
\[\lim_{p\to\infty}d\mu_p(t)= \lim_{p \to \infty} \sum_{\lambda\in \mathcal{C}^+} \mathcal{L}(F_{\lambda})
\ol{F_{\lambda}(t)}\,d\mu_{ST}(t)
 =  \mathcal{L}(F_{\mathbf 0}) F_{\mathbf{0}}(t)d\mu_{ST}(t) = d\mu_{ST}(t). \]
 The last step follows by \eqref{L} and the fact that $F_{\mathbf 0} = 1$ (cf. \eqref{Weyl}).
\end{proof}

\section{Local computation when $n=2$}\label{n2}

Here we refine the discussion from Section \ref{pinS} for $p\in\Sb$,
   with the simplifying assumptions that $n=2$, and
\begin{equation} \label{pnmidsigma}
 p \nmid 4\det \sigma_1.
\end{equation}
(Recall that $\det\sigma_1\in \frac14\Z$.)
  The main goal of this section is to prove the following local bound.

\begin{proposition} \label{conjn2}
  Under the above hypotheses, there exists a constant $\e > 0$ such that
   \begin{equation}\label{n2bound}
 |I_{A,p}(c_\lambda)| \ll p^{(1-\e)\left<\lambda, \rho\right>-\e r_p}
\end{equation}
for all $\lambda\in\mathcal{C}^+$,
   where the implied constant depends only on $p$ and $\e$.
\end{proposition}

Before proving the proposition, let us observe how it implies
  the global Hypothesis \eqref{conj}.

\begin{corollary}\label{n2conj}
Suppose $n=2$, $\sigma_1=\sigma_2=\sigma$, $\delta_\k(\sigma,\sigma)>0$,
  and $p\nmid 4\det\sigma$ for all $p\in\Sb$.  Then
  Hypothesis \eqref{conj} holds.
\end{corollary}

\begin{proof}
In Section \ref{qnotp}, we saw that $I_{A,p} = 1$ for primes $p\notin\Sb$.
From the definition \eqref{L} of $\mathcal{L}$, Theorem \ref{A}, and Lemma \ref{weights},
  it follows that
\begin{equation}\label{Lp}
\mathcal{L}(\prod_{p\in\Sb}\mathcal{S}(c_{\lambda_p}))
  = \frac{\sum_{A} c_{n\k\sigma}\prod_{p\in\Sb}I_{A,p}(c_{\lambda_p})}
  {\delta_\k(\sigma,\sigma)c_{n\k\sigma}}
  =\frac1{\delta_\k(\sigma,\sigma)}
   \sum_{A} \prod_{p\in\Sb}I_{A,p}(c_{\lambda_p}),
\end{equation}
where $A$ runs through the rank-$2$ matrices in $M_2(\Z)/\{\pm 1\}$ satisfying
  $r\,{}^t\!A^{-1}\in M_2(\Z)$ and ${}^t\!A\sigma A=r\sigma$.
In particular, writing
\[ \sigma = \mat{a}{b/2}{b/2}{c}, \qquad A=\mat xyzw \in M_2(\Z), \]
we have $ax^2+bxz+cz^2 = ra$.  Hence
\[4ra^2=(2ax+bz)^2-(b^2-4ac)z^2=(2ax+bz+\sqrt{D}z)(2ax+bz-\sqrt{D}z),\]
where $D=b^2-4ac<0$.  Thus, in the ring of integers $\mathcal O\subset \Q[\sqrt{D}]$, the
  ideal $(2ax+bz+\sqrt D z)$ is a factor of the ideal $(4ra^2)$.
The number of ideal factors of $(4ra^2)$ is $\ll r^{\e/2}$. In view of the fact
  that $|\mathcal O^*|<\infty$, the number of possible choices for $x, z$ is $\ll r^{\e/2}$.
Similarly, the number of choices for $y, w$ is $\ll r^{\e/2}$.
So the number of terms in the sum is $\ll r^\e=\prod_{p\in\Sb}p^{r_p\e}$.
  It follows from \eqref{n2bound} that the above is
\[\ll \prod_{p\in\Sb}p^{(1-\e)\sg{\lambda_p,\rho}},\]
as required.
\end{proof}

  The proof of Proposition \ref{conjn2} is given in Section \ref{Ver}.
In the intervening sections, we describe how to compute the
  local integral $I_{A,p}$ explicitly, with the
  goal of producing the upper bound \eqref{n2bound}.
  In many situations, the trivial bound \eqref{trivboundeq} is adequate, so
  an explicit computation is not necessary.
  In the remaining cases (which, in the notation below, occur when $\beta-1\le t$),
we give a complete treatment of the local integral.

\subsection{Preliminaries}\label{n2a}

 Without loss of generality,
  we consider the case where $f_p=c_\lambda$ is the characteristic function
  of the double coset $Z_pK_p\lambda(p)K_p$, where
\begin{equation}\label{lam}
\lambda(p)=\diag(1,p^t,p^\tau,p^{\tau-t})
\end{equation}
for $0\le t\le \tau/2$ as in \eqref{Ccond}.
  Thus we write $\tau$ in place of $r_p$ (for the purpose of eliminating a subscript).

By Proposition \ref{D}, we need only consider matrices $D$ of the form
\[ D = \diag(p^{\alpha}, p^{\beta}),\quad \,\,0 \le \alpha \le \beta,\quad  \,\,\alpha + \beta=\tau.\]
Write
\[ \sigma_U = \t U \sigma_1 U =\mat a{b/2}{b/2}c. \]
Note that $\sigma_U$ is half-integral, and $\det\sigma_U=\det\sigma_1$.
So by \eqref{pnmidsigma},
 either $p\nmid b$ or $p\nmid ac$.
 We would like to compute the integral
\begin{equation}\label{Ip}
 I_{A,p}= \int_{S(\Q_p)} f_p (\mat{D}{p^\tau SD^{-1}}{O}{p^\tau D^{-1}})
  \theta_{p}(\tr \sigma_U S) dS.
\end{equation}

Writing $S=\mat xyyz$, we let
\[ M=\mat{D}{p^\tau S  D^{-1}}{O}{p^\tau   D^{-1}} =
   \begin{pmatrix}
   p^{\alpha} & 0 & p^{\beta} x & p^{\alpha} y \\
   0 & p^{\beta} & p^{\beta} y & p^{\alpha} z \\
   0 & 0 & p^{\beta} & 0 \\
   0 & 0 & 0 & p^{\alpha} \\
   \end{pmatrix}. \]
By  Proposition \ref{fp}, $M \in \supp f_p$ if and only if the fractional ideal generated
   by all the entries is $(1)=\Z_p$
   and the fractional ideal generated by the determinants of all $2\times 2$ submatrices
  is $(p^{t})=p^t\Z_p$.
The determinants of the $2\times 2$ submatrices of $M$ are shown in the table below:\\
\begin{center}
\begin{tabular}{ c|cccccc }
\backslashbox{rows}{cols} & 1,2 & 1,3 & 1,4 & 2,3 & 2,4 & 3,4 \\
\hline
1,2 & $p^{\alpha+\beta}$ & $p^{\alpha+\beta} y$ & $p^{2\alpha }z$ &  $-p^{2\beta} x$ & $-p^{\alpha+\beta} y$  & $p^{\alpha+\beta}(xz-y^2)$ \\
1,3 & $0$ & $p^{\alpha+\beta}$ & $0$ & $0$ & $0$ & $-p^{\alpha+\beta} y$ \\
1,4 & $0$ & $0$ & $p^{2\alpha}$ & $0$ & $0$ & $p^{\alpha+\beta} x$ \\
2,3 & $0$ & $0$ & $0$ & $p^{2\beta}$ & $0$ & $-p^{\alpha+\beta} z$ \\
2,4 & $0$ & $0$ & $0$ & $0$ & $p^{\alpha+\beta}$  & $p^{\alpha+\beta} y$ \\
3,4 & $0$ & $0$ & $0$ & $0$ & $0$ &  $p^{\alpha+\beta}$\\
\end{tabular}
\end{center}
\vskip .3cm
\noindent Using $\alpha\le \beta$ and $\alpha+\beta=\tau$, we see that
 $M \in \supp f_p$ if and only if
\begin{equation} \label{1x1sub}
   (p^{\alpha},\, p^{\beta} x,\, p^{\alpha} y,\, p^{\alpha} z) = (1)
\end{equation}
and
 \begin{equation} \label{2x2sub}
   (p^{2\alpha},\, p^{\tau} x,\, p^{\tau} y,\, p^{2\alpha} z,\, p^{\tau}(xz-y^2))= (p^{t}).
\end{equation}
Let
\begin{equation} \label{xyz'}
    x' = p^{\beta} x,\quad y'=p^{\alpha}y,\quad z'=p^{\alpha} z.
\end{equation}
Then \eqref{1x1sub} is equivalent to
\begin{equation} \label{1x1sub'}
   (p^{\alpha}, x', y', z') = (1)
\end{equation}
and \eqref{2x2sub} is equivalent to
 \begin{equation} \label{2x2sub'}
   (p^{2\alpha}, p^{\alpha} x', p^{\beta} y', p^{\alpha} z', x' z' - p^{\beta-\alpha}{y'}^2)= (p^{t}).
\end{equation}
If $\alpha \neq 0$, then \eqref{1x1sub} is equivalent to
\begin{equation} \label{xyz1}
   (x', y', z')=(1),
\end{equation}
i.e., $x', y', z' \in \Z_p$ and at least one of them is a unit.

\subsection{Evaluation of the integral $I_{A,p}$} \label{Ieval}
We continue with the notation from above.
Given a Borel subset $S'\subset S(\Q_p)$, we define
\[ I_{S'} = \int_{S'} f_p(\mat{D}{rSD^{-1}}{O}{rD^{-1}})\theta_{p}(\tr \sigma_U S) dS. \]
Define
\[S_0 =\{\mat xyyz\in S(\Q_p)|\, x,y,z\text{ satisfy }
 \eqref{1x1sub} \text{ and } \eqref{2x2sub}\}.\]
Then $I_{S_0}=I_{A,p}$ is the integral \eqref{Ip} we need to compute.

Let $E_{11}=\mat 1000$, $E_{22}=\mat 0001$ and $E'_{12}=\mat 0110$.

\begin{proposition} \label{IS'}
Let $S'$ be a Borel subset of $S_0$.
   Suppose $p\nmid a$ (resp. $p \nmid b$, $p \nmid c$).
   Suppose $\mat xyyz \in S'$ implies that $\mat xyyz \pm \frac 1p E_{11}$
  ($\mat xyyz \pm \frac 1p E'_{12}$, $\mat xyyz \pm \frac 1p E_{22}$ respectively) belongs to $S'$.
   Then $I_{S'}=0$.
\end{proposition}

\begin{proof} Suppose $p\nmid a$. The other cases can be handled similarly.
 By the given property,  $\mat xyyz \in S'$ if and only if $\mat xyyz + \frac 1p E_{11} \in S'$.
  Hence
 \[ I_{S'}
     = \int_{S'} \theta_{p}(\tr \sigma_U S) dS
 = \int_{S'}  \theta_{p}(\tr \sigma_U(S-\frac 1p E_{11})) dS
 = e(\frac{a}{p}) I_{S'}.  \]
The proposition follows.
\end{proof}

\begin{proposition} \label{add1p}
Suppose $\mat xyyz \in S_0$.  Then:
\begin{enumerate}
  \item[(i)] If  $\beta \geq 2$,  $\tau-1 \geq t+1$ and  $p^{t+1} | p^{\beta-1} z'$,
  then $\mat xyyz \pm \frac 1p E_{11} \in S_0$.
  \item[(ii)] If  $\alpha \geq 2$,  $\tau-2 \geq t+1$ and  $p^{t+1} | p^{\beta-1} y'$,
  then $\mat xyyz \pm \frac 1p E'_{12} \in S_0$.
  \item[(iii)] If $\alpha \geq 2$, $2\alpha-1 \geq t+1$, $p^{t+1}|p^{\alpha-1} x'$,
  then $\mat xyyz \pm \frac 1p E_{22} \in S_0$.
\end{enumerate}
\end{proposition}

\noindent {\em Remarks}: 1) If in \eqref{2x2sub'},
$(p^{2\alpha}, p^{\alpha} x', p^{\beta} y', p^{\alpha} z')=(p^t)$,
then the last condition of (i) (resp. (ii), (iii))  can be replaced by
the weaker condition $p^t|p^{\beta-1}z'$ (resp. $p^t|p^{\beta-1}y'$, $p^t|p^{\alpha-1}x'$),
and the second condition of (ii) can be weakened to $\tau-2 \ge t$.

2) There are some other variants; for example,
if $(x',z')=1$ and $\alpha = t$, then in (ii), the conditions can be
  replaced by $\alpha\ge 1$, $\tau-2 \ge t$, and $p^t|p^{\beta-1}y'$.

\begin{proof}
(i) Replace $x$ by $x\pm \frac 1p$ in \eqref{1x1sub} and \eqref{2x2sub}. The left-hand
  side of \eqref{1x1sub} becomes
   \[ (p^{\alpha}, p^{\beta} x \pm  p^{\beta-1}, p^{\alpha} y, p^{\alpha} z). \]
   The left-hand side of \eqref{2x2sub} becomes
   \[ (p^{2\alpha}, p^{\tau} x \pm p^{\tau-1}, p^{\tau} y, p^{2\alpha} z,
  p^{\tau}(xz-y^2) \pm p^{\beta-1} z'). \]
   Under the given hypotheses,  $p^{\beta-1} \in (p)$, $p^{\tau-1}
  \in (p^{t+1})$ and $p^{\beta-1}z' \in (p^{t+1})$.
   Hence $\mat xyyz \pm \frac 1p E_{11}$ satisfies \eqref{1x1sub} and \eqref{2x2sub}.

(ii)  Replace $y$ by $y\pm \frac 1p$ in \eqref{1x1sub} and \eqref{2x2sub}. The
  left-hand side of \eqref{1x1sub} becomes
   \[ (p^{\alpha}, p^{\beta} x, p^{\alpha} y \pm  p^{\alpha-1}, p^{\alpha} z). \]
   The left-hand side of \eqref{2x2sub} becomes
   \[ (p^{2\alpha}, p^{\tau} x , p^{\tau} y\pm p^{\tau-1}, p^{2\alpha} z,
   p^{\tau}(xz-y^2) \mp 2p^{\beta-1} y' - p^{\tau-2} ). \]
   Under the given hypotheses in this case, $p^{\alpha-1} \in (p)$, $p^{\tau-2} \in (p^{t+1})$
  and $p^{\beta-1}y' \in (p^{t+1})$.
   Hence $\mat xyyz\pm\frac 1p E'_{12}$ satisfies  \eqref{1x1sub} and \eqref{2x2sub}.

Assertion (iii) and the remarks can be proven similarly.
\end{proof}

\begin{corollary} \label{pnmidac}
   Suppose $p\nmid a$, and that
   \begin{itemize}
   \item[(i)] $\beta \ge 2$,
   \item[(ii)] $\tau-1 \ge t+1$,
   \item[(iii)] $\beta-1 \ge t+1$.
   \end{itemize}
   Then $I_{A,p}=0$.
\end{corollary}
\begin{proof} Suppose $\mat xyyz \in S_0$. Then by \eqref{1x1sub'}, $z' \in \Z_p$,
 so by the third hypothesis, $p^{t+1}|p^{\beta-1} z'$.
   By Proposition \ref{add1p}, $\mat xyyz \pm \frac 1p E_{11} \in S_0$.
  The assertion now follows by Proposition \ref{IS'}.
\end{proof}

\begin{corollary} \label{pnmidb}
   Suppose $p\nmid b$, and
   \begin{itemize}
   \item[(i)]  $\alpha \ge 2$,
   \item[(ii)] $\tau-2 \ge t+1$,
   \item[(iii)] $\beta-1 \ge t+1$.
   \end{itemize}
   Then $I_{A,p}=0$.
\end{corollary}
\begin{proof} This follows in the same way as the previous corollary,
  using Proposition \ref{add1p} (ii), and Proposition \ref{IS'}.
\end{proof}

\begin{proposition} \label{betat}
Suppose condition (iii) of the above corollaries fails to hold,
i.e., $\beta-1 \le t$. Then exactly one of the following is true:
\begin{enumerate}
\item  $\tau=2\tau'+1$ is odd, $\alpha=\tau'$, $\beta=\tau'+1$, and $t=\tau'$,
   \item  $\tau=2\tau'$ is even, $\alpha=\tau'-1$, $\beta=\tau'+1$, and $t=\tau'$,
   \item  $\tau=2\tau'$ is even, and $\alpha=\beta=t=\tau'$,
\item  $\tau=2\tau'$ is even, $\alpha=\beta=\tau'$, and $t=\tau'-1$.
\end{enumerate}
\end{proposition}
\begin{proof} Suppose $\beta \le t+1$.  Then
    because we always have $t \le \left[\frac{\tau}2\right]$ (where brackets denote
  the floor), it follows that
  $\beta\le t+1\le \left[\frac\tau2\right]+1$.
    On the other hand, $\tau = \alpha+\beta \le 2\beta$,
which gives the lower bound in the following inequality:
\[\left \lceil \frac{\tau}2 \right \rceil\le \beta \le \left[\frac{\tau}2\right]+1.\]
 (Here, $\lceil\cdot\rceil$ denotes the ceiling.)
Using $\beta - 1 \le t \le \left[\frac \tau2 \right]$,
    the result follows easily by considering the possible cases.
\end{proof}

\begin{proposition} \label{betat1}
Suppose $\alpha, \beta, t, \tau$  satisfy Proposition \ref{betat} (1),
  i.e., $\tau=2\tau'+1$, $\alpha=\tau'$, $\beta=\tau'+1$, and $t=\tau'$.
   Then if $\tau'\ge 2$, $I_{A,p}=0$.
\end{proposition}

\begin{proof}
Let $\mat xyyz \in S_0$.  Then \eqref{2x2sub'} is satisfied,
  and since $\alpha=\tau'\ge 2$, \eqref{xyz1} is also satisfied.
  In particular, by \eqref{2x2sub'},
\begin{equation}\label{x'z'}
x' z' \equiv p{y'}^2 \pmod{p^2}.
\end{equation}
It follows that either $x'$ or $z'$ is a unit.  Indeed,
if $p|x'$ and $p|z'$, then $y'$ is a unit by \eqref{xyz1}, leading to an
  obvious contradiction in \eqref{x'z'}.
In fact, by \eqref{x'z'}, $p|x'z'$ and hence exactly one
  of $x'$ or $z'$ is a unit.

First suppose $p\nmid b$.  Note that $(x',z')=1$, $\alpha=t\ge 2$,
  \[\tau-2=2\tau'-1>2\tau'-\tau'=t,\]
 and $p^t|p^{\beta-1}y'$.
By the second remark after Proposition \ref{add1p},
   $\mat xyyz \pm \frac 1p E'_{12} \in S_0$.
  By Proposition \ref{IS'}, $I_{A,p}=I_{S_0}=0$.

Finally, suppose $p|b$. Then as noted earlier, $p\nmid a$.
    Since one of $x'$ or $z'$ is a unit, $(p^{\alpha} x', p^{\alpha} z')=(p^t)$.
  Furthermore, $p^t|p^{\beta-1}z'$, and as above, $\tau-2\ge t$.
    By the first remark after Proposition \ref{add1p},
    $\mat xyyz \pm \frac 1p E_{11} \in S_0$.
    By Proposition \ref{IS'}, $I_{A,p}=I_{S_0}=0$.
\end{proof}

\begin{proposition} \label{betat2}
Suppose $\alpha, \beta, t, \tau$  satisfy Proposition \ref{betat} (2), i.e., $\tau=2\tau'$ is even,
   $\alpha=\tau'-1$, $\beta=\tau'+1$, and $t=\tau'$.  Then if $\tau'\ge 3$, $I_{A,p}=0$.
  In fact, if $p\nmid ac$, then $I_{A,p}=0$ if $\tau'\ge 2$.
\end{proposition}

\begin{proof}
Let $\mat xyyz \in S_0$.
Suppose $\tau' \ge 2$.
   By \eqref{2x2sub'}, $p^{\tau'}|p^{\tau'-1}x'$ and $p^{\tau'}|p^{\tau'-1}z'$,
and hence $p|x'$ and $p|z'$.
Therefore by \eqref{xyz1}, $y'$ is a unit.

Suppose  $p\nmid ac$.  Because $p|z'$, we have $p^{t+1}|p^{\beta-1} z'$.
  Hence by Proposition \ref{add1p}, $\smat xyyz \pm \frac 1p E_{11} \in S_0$.
  By Proposition \ref{IS'}, $I_{S_0}=0$.

Next suppose $p\nmid b$ and $\tau' \ge 3$. Write $x'=px''$ and $z'=pz''$, with
  $x''$ and $z'' \in \Z_p$.  By \eqref{2x2sub'},
  $p^2x''z''\equiv p^2y'^2\pmod {p^{t}}$, so
    $x'' z''\equiv y'^2 \pmod{p^{\tau'-2}}$.
    Because $\tau'\ge 3$, it follows that $x''$ and $z''$ are units.
  Therefore $(p^{\alpha}x')=(p^t)$.
    Obviously $p^{t}|p^{\beta-1}y'$. By the first remark after Proposition \ref{add1p},
    $\smat xyyz \pm \frac 1p E'_{12} \in S_0$.
    By Proposition \ref{IS'}, $I_{A,p}=I_{S_0}=0$.
\end{proof}

\begin{proposition} \label{betat3}
   Suppose $\alpha, \beta, t, \tau$  satisfy Proposition \ref{betat} (3),
i.e. $\tau=2\tau'$ and $\alpha=\beta=t=\tau'$.
Suppose further that $\tau' \ge 2$.  Then
the integral $I_{A,p}$ is given explicitly by \eqref{case3result} below.
\end{proposition}

\begin{proof} Suppose $\mat xyyz \in S_0$.
Then \eqref{xyz1} implies $(p^t x', p^{t} y',p^t z')= (p^t)$.
Hence \eqref{xyz1} and \eqref{2x2sub'} taken together are equivalent to \eqref{xyz1} and
\begin{equation} \label{xzy2}
    x'z' \equiv {y'}^2 \pmod{p^t}.
\end{equation}
If $y'$ is not a unit, then by \eqref{xzy2} and \eqref{xyz1}, exactly one of $x'$ or $z'$ is a unit.
   So there is a partition
\[S_0=S_1\cup S_2\cup S_3,\]
where
\begin{align*}
S_1&=\{\mat xyyz\in S_0:\, \text{$x', y', z'$ are units}\},\\
S_2&=\{\mat xyyz\in S_0:\, \text{$p|y'$, $p|z'$, and $x'$ is a unit}\},\\
S_3&=\{\mat xyyz\in S_0:\, \text{$p|y'$, $p|x'$, and $z'$ is a unit}\}.
\end{align*}

We claim that $I_{S_2}=I_{S_3}=0$.
Let $\smat xyyz\in S_2$.
Then $(p^{\alpha} x')=(p^t)$, $p^{t}|p^{\beta-1}z'$,
   and $\tau-1\ge t+1$.
   By the first remark after Proposition \ref{add1p},
   $\mat xyyz \pm \frac 1p E_{11} \in S_0$.  In fact, this matrix belongs to $S_2$
since $p^{\beta}(x\pm\frac 1p)=x'\pm p^{\beta-1}$ is a unit.
Hence  by Proposition \ref{IS'}, $I_{S_2}=0$ if $p\nmid ac$.
If $p|ac$, then $p\nmid b$.  In this case,
 $(p^{\alpha} x')=(p^t)$, $p^{t}|p^{\beta-1}y'$, and $\tau -1 \ge t+1$.
   By the first remark after Proposition \ref{add1p},
   $\mat xyyz \pm \frac 1p E'_{12} \in S_0$. In fact, the matrix belongs to
  $S_2$ since $p^\alpha (y\pm \frac1p)=y'\pm p^{\alpha-1}\in p\Z_p$.
  By Proposition \ref{IS'}, $I_{S_2}=0$.
The proof that $I_{S_3}=0$ is similar.

For the integral over $S_1$,
  note that the validity of \eqref{xyz1} and \eqref{xzy2} depends only on
 $x',y',z' \pmod{p^{\tau'}}$, which also means that $S_1=S_1+S_2(\Z_p)$
  (where $S_2$ here denotes the symmetric matrices).
Hence, writing $\mathtt{x,y,z}$ for the congruence classes of
  $x',y',z'\mod p^{\tau'}$,
  \[ I_{A,p}=I_{S_1} = \int_{S_1} \theta_{p}(\tr (\sigma_U S)) dS \]
  \[ = \sum_{\mathtt{x,y,z}\in (\Z_p/p^{\tau'}\Z_p)^*,\atop
       \mathtt{x} \mathtt{z} \equiv \mathtt{y}^2 (\operatorname{mod}{p^{\tau'}})}
       \int_{\frac{\mathtt{z}}{p^{\tau'}}+\Z_p}
       \int_{\frac{\mathtt{y}}{p^{\tau'}} + \Z_p}
       \int_{\frac{\mathtt{x}}{p^{\tau'}}+\Z_p}
       \theta_{p}({ax+by+cz}) dxdydz \]
  \[ = \sum_{\mathtt{y} \in (\Z_p/p^{\tau'}\Z_p)^*} e(-\frac{b\mathtt{y}}{p^{\tau'}})
       \left( \sum_{\mathtt{x}, \mathtt{z} \in \Z_p/p^{\tau'}\Z_p \atop
 \mathtt{x} \mathtt{z} \equiv \mathtt{y}^2 (\operatorname{mod}{p^{\tau'}})}
 e(-\frac{a\mathtt{x}+c\mathtt{z}}{p^{\tau'}})\right) \]
  \[= \sum_{\mathtt{y} \in (\Z_p/p^{\tau'}\Z_p)^*} e(-\frac{b\mathtt{y}}{p^{\tau'}})
       \left( \sum_{\mathtt{x}, \mathtt{z} \in \Z_p/p^{\tau'}\Z_p \atop
 \mathtt{x} \mathtt{z} \equiv 1(\operatorname{mod}{p^{\tau'}})}
        e(-\frac{\mathtt{y}(a\mathtt{x}+c\mathtt{z})}{p^{\tau'}})\right) \]
  \begin{equation} \label{case3result}
  =\sum_{\mathtt{x}, \mathtt{z} \in \Z_p/p^{\tau'}\Z_p \atop
 \mathtt{x} \mathtt{z} \equiv 1(\operatorname{mod}{p^{\tau'}})}
  \left(\sum_{\mathtt{y} \in (\Z_p/p^{\tau'}\Z_p)^*}
  e(-\frac{\mathtt{y}(a\mathtt{x}+c\mathtt{z}+b)}{p^{\tau'}})\right).
  \end{equation}
We remark that the computation here is valid for $\tau' \ge 1$.
The sum over $\mathtt{y}$ can be evaluated using \eqref{ysum} below.
\end{proof}

\begin{proposition} \label{betat4}
   Suppose $\alpha, \beta, t, \tau$  satisfy Proposition \ref{betat} (4), i.e.
  $\tau=2\tau'$ is even, $\alpha=\beta=\tau'$, and $t=\tau'-1$. Suppose further that
  $\tau' \ge 2$.  Then the integral $I_{A,p}$ is given by \eqref{case4result} below.
\end{proposition}

\begin{proof}
  In this case, \eqref{xyz1} and \eqref{2x2sub'} are equivalent to \eqref{xyz1} and
\begin{equation} \label{2x2sub4}
   (x'z'-{y'}^2) = (p^{\tau'-1}),
\end{equation}
i.e.,
\begin{equation} \label{xzy2a}
    x'z' \equiv {y'}^2 \pmod{p^{\tau'-1}}
\end{equation}
   but
\begin{equation} \label{xzy2b}
    x'z' \not\equiv {y'}^2 \pmod{p^{\tau'}}
\end{equation}
   As in the previous proof, we integrate over $S_1,S_2$ and $S_3$.
  We first show that $I_{S_2}=I_{S_3}=0$.
  Suppose $\smat xyyz\in S_2$, so $x'\in\Z_p^*$, $p|y'$, and $p|z'$.
Then the conditions of Proposition \ref{add1p}(i) are satisfied,
   so $\mat xyyz \pm \frac 1p E_{11} \in S_0$.
  Since $p^{\beta}(x+\frac 1p)=x'+p^{\beta-1}$ is a unit, this matrix
  in fact belongs to $S_2$.
By Proposition \ref{IS'}, $I_{S_2}=0$, assuming $p\nmid ac$.
On the other hand, if $p|ac$, then $p\nmid b$.
The hypotheses to Proposition \ref{add1p}(ii) are satisfied, so
      $\mat xyyz \pm \frac 1p E'_{12} \in S_2$, and once again
Proposition \ref{IS'} gives $I_{S_2}=0$.
The proof that $I_{S_3}=0$ is similar.

For the integral over $S_1$, just as in the proof of the previous proposition,
  we have
  \[ I_{A,p}=I_{S_1} = \int_{S_1}  \theta_{p}(\tr \sigma_U S) dS \]
  \[ = \sum_{\mathtt{x,y,z} \in (\Z_p/p^{\tau'}\Z_p)^*,\atop
       \mathtt{x} \mathtt{z} \equiv \mathtt{y}^2 (\operatorname{mod}{p^{\tau'-1}}), \,\,
       \mathtt{x} \mathtt{z} \not \equiv \mathtt{y}^2 (\operatorname{mod}{p^{\tau'}})}
       \int_{\frac{\mathtt{z}}{p^{\tau'}}+\Z_p}
       \int_{\frac{\mathtt{y}}{p^{\tau'}} + \Z_p}
       \int_{\frac{\mathtt{x}}{p^{\tau'}}+\Z_p}
       \theta_{p}(a{x}+b{y}+c{z}) dxdydz \]
  \[ = \sum_{\mathtt{y} \in (\Z_p/p^{\tau'}\Z_p)^*} e(-\frac{b\mathtt{y}}{p^{\tau'}})
       \left( \sum_{\mathtt{x}, \mathtt{z} \in \Z_p/p^{\tau'}\Z_p \atop
 \mathtt{x} \mathtt{z} \equiv \mathtt{y}^2 (\operatorname{mod}{p^{\tau'-1}}),
 \mathtt{x} \mathtt{z} \not \equiv \mathtt{y}^2 (\operatorname{mod}{p^{\tau'}})}
 e(-\frac{a\mathtt{x}+c\mathtt{z}}{p^{\tau'}})\right) \]
  \[ = \sum_{\mathtt{y} \in (\Z_p/p^{\tau'}\Z_p)^*} e(-\frac{b\mathtt{y}}{p^{\tau'}})
       \left( \sum_{\mathtt{x}, \mathtt{z} \in \Z_p/p^{\tau'}\Z_p \atop
  \mathtt{x} \mathtt{z} \equiv 1(\operatorname{mod}{p^{\tau'-1}}),
 \mathtt{x} \mathtt{z} \not \equiv 1(\operatorname{mod}{p^{\tau'}})}
        e(-\frac{\mathtt{y}(a\mathtt{x}+c\mathtt{z})}{p^{\tau'}})\right) \]
  \begin{equation} \label{case4result}
  =\sum_{\mathtt{x}, \mathtt{z} \in \Z_p/p^{\tau'}\Z_p \atop
  \mathtt{x} \mathtt{z} \equiv 1(\operatorname{mod}{p^{\tau'-1}}),
 \mathtt{x} \mathtt{z} \not \equiv 1(\operatorname{mod}{p^{\tau'}})}
  \left(\sum_{\mathtt{y} \in (\Z_p/p^{\tau'}\Z_p)^*}
  e(-\frac{\mathtt{y}(a\mathtt{x}+c\mathtt{z}+b)}{p^{\tau'}})\right)
  \end{equation}
 \[
  = \sum_{h=1}^{p-1}
 \sum_{\mathtt{x}, \mathtt{z} \in \Z_p/p^{\tau'}\Z_p \atop
  \mathtt{x} \mathtt{z} \equiv 1+hp^{\tau'-1}(\operatorname{mod}{p^{\tau'}})}
  \left(\sum_{\mathtt{y} \in (\Z_p/p^{\tau'}\Z_p)^*}
  e(-\frac{\mathtt{y}(a\mathtt{x}+c\mathtt{z}+b)}{p^{\tau'}})\right).
  \]
Once again, this computation is valid for $\tau' \ge 1$.
We remark that \eqref{case4result} can be rewritten as
\begin{equation} \label{case4result'}
 \sum_{\mathtt{x}, \mathtt{z} \in \Z_p/p^{\tau'}\Z_p \atop
  \mathtt{x} \mathtt{z} \equiv 1(\operatorname{mod}{p^{\tau'-1}})}
  \left(\sum_{\mathtt{y} \in (\Z_p/p^{\tau'}\Z_p)^*}
  e(-\frac{\mathtt{y}(a\mathtt{x}+c\mathtt{z}+b)}{p^{\tau'}})\right)
  - \text{\eqref{case3result}}.\qedhere
\end{equation}
\end{proof}

\begin{corollary}\label{betat3cor}
Suppose $\beta-1\le t$ and $\tau\ge 5$.  Then
  $I_{A,p} \ll_p p^{3\tau/4}$.
\end{corollary}

\begin{proof}
By Propositions \ref{betat}-\ref{betat4}, we may assume that $\tau=2\tau'$ is even,
  and $I_{A,p}$ is given by \eqref{case3result} or \eqref{case4result}.
Recall the formula for the Ramanujan sum
  \begin{equation}\label{ysum}
 \sum_{\mathtt{y} \in (\Z_p/p^{\tau'}\Z_p)^*} e(\frac{\mathtt{y}\ell}{p^{\tau'}})
    = \begin{cases}
    p^{\tau'}-p^{\tau'-1} & \text{ if }  p^{\tau'}|\ell, \\
    -p^{\tau'-1} & \text{ if } p^{\tau'-1} \| \ell, \\
    0 & \text{ if }p^{\tau'-1}\nmid \ell
    \end{cases}
  \end{equation}
(\cite[Theorem 4.3]{Hua}).
Thus the summation over $\mathtt{y}$ in \eqref{case3result} is nonzero only if
\begin{equation} \label{case3sum}
 a\mathtt{x} + c\mathtt{z} + b \equiv 0 \pmod{p^{\tau'-1}}.
 \end{equation}
Since $\mathtt{xz}\equiv 1\mod p^{\tau'}$, this is equivalent to
\begin{equation} \label{case3sum'}
 a\mathtt{x}^2 + b \mathtt{x} + c \equiv 0 \mod{p^{\tau'-1}},
 \end{equation}
and also to
\begin{equation} \label{case3sum''}
 c\mathtt{z}^2 + b\mathtt{z} + a \equiv 0 \mod{p^{\tau'-1}}.
 \end{equation}
Suppose $p|a$ and $p|c$, so that $p\nmid b$. Then \eqref{case3sum} has no solution,
so the summation is zero.
If $p\nmid a$ (resp. $p \nmid c$), then the
number of $\mathtt x\mod p^{\tau'}$ satisfying \eqref{case3sum'} (resp. $\mathtt z
  \mod p^{\tau'}$ satisfying \eqref{case3sum''}) is
  $\ll p p^{(\tau'-1)/2} \ll_p p^{\tau'/2}$ (cf. \cite[Lemma 9.6]{ftf}).
Now applying \eqref{ysum}, we see that \eqref{case3result} is $O(p^{3\tau'/2})$.

Similarly, \eqref{case4result} is
 \[ \ll \sum_{h=1}^{p-1}
 \sum_{\mathtt{x}, \mathtt{z} \in \Z_p/p^{\tau'}\Z_p \atop{
  \mathtt{x} \mathtt{z} \equiv 1+hp^{\tau'-1}\pmod{p^{\tau'}},
   \atop a\mathtt{x}+c\mathtt{z}+b \equiv 0 \pmod{p^{\tau'-1}} }}
  p^{\tau'} \ll p^2 p^{\frac{\tau'-1}2} p^{\tau'} \ll p^{3\tau'/2}. \qedhere\]
\end{proof}

\subsection
{Proof of Proposition \ref{conjn2}}\label{Ver}

  For $p\in\Sb$, we require a bound for the local integral $I_{A,p}$ given in
  \eqref{Ip}.
We continue to use the notation of Section \ref{n2a}.  Thus,
$\tau=r_p$, $\lambda(p)=\diag(1,p^t,p^{\tau},p^{\tau-t})$ for  $0\le t\le \tau/2$,
$D = \diag(p^{\alpha}, p^{\beta})$ for  $0 \le \alpha \le \beta$ with
   $\alpha + \beta=\tau$, and $\sigma_U = \t U \sigma_1 U =\mat a{b/2}{b/2}c$.

By \eqref{rho}
  \[ \rho = \frac 32 e_0-2e_1 - e_2, \]
so
  \[ \left<\lambda, \rho\right> =  \frac 32 \ell_0 - 2\ell_1 - \ell_2 = \frac 32 \tau - t. \]

To prove Proposition \ref{conjn2}, we must show that for some $\e>0$,
   \begin{equation}\label{n2bound'}
 |I_{A,p}(c_\lambda)| \ll
p^{(1-\e)\left<\lambda, \rho\right>-\e \tau}
 =p^{(1-\e)(\frac 32 \tau  -t) - \e \tau}.
\end{equation}

By the trivial bound \eqref{trivboundeq},
  \[|I_{A,p}|\le p^{2\alpha+\beta}=p^{\alpha+\tau}. \]
  Therefore \eqref{n2bound'} is certainly satisfied when
  \begin{equation} \label{conje}
  \alpha+\tau\leq (1-\e)(\frac 32 \tau - t) - \e \tau,
  \end{equation}
or equivalently,
\[\alpha+(1-\e)t\le \frac{1-5\e}2\tau.\]
Note that $(1-\e)(\frac 32 \tau  -t) - \e \tau$ is a decreasing function of $\e$.
Therefore if \eqref{n2bound'} holds for some particular $\e=\e_0$,
  then it holds for all smaller positive $\e$.  For concreteness, we
  will verify it for $\e_0=0.01$, in which case the above inequality
  takes the form
  \begin{equation} \label{conje2}
  \alpha+0.99t \le 0.475 \tau.
  \end{equation}

For any given value of $\tau$, there are only finitely many permissible
  values for $t,\alpha$ and $\beta$, so the associated integral is bounded
  by a constant depending only on $\tau$.  Therefore we may assume that
\begin{equation}\label{tau7}
\tau\ge 7.
\end{equation}

If $p\nmid b$ (resp. $p\nmid ac$) and the conditions of Corollary \ref{pnmidac}
(resp. Corollary \ref{pnmidb}) hold, the integral vanishes and
  the desired bound is trivially satisfied.

Suppose condition (iii) of either Corollary \ref{pnmidac} or Corollary \ref{pnmidb}
fails.
 Then by Corollary \ref{betat3cor},
\[I_{A,p}\ll p^{3\tau/4} < p^{(1-2\e)\tau}\]
  when $\e < 1/8$.
  If $t=\tau/2$ (as is the case in Proposition \ref{betat3}), then
 \[p^{(1-2\e)\tau}= p^{(1-\e)(\frac 32 \tau-t)- \e\tau}. \]
If $t=\tau/2-1$ (as is the case in Proposition \ref{betat4}), then
\[p^{(1-2\e)\tau}\ll p^{(1-2\e)\tau +(1-\e)}=p^{(1-\e)(\frac32\tau-t)-\e\tau}.\]
Either way, we obtain the desired bound for $I_{A,p}$ when $\e<1/8$.

Suppose condition (i) of either Corollary \ref{pnmidac} or Corollary \ref{pnmidb}
  fails.  Then (using $\alpha\le \beta$ in the first case)  $\alpha \le 1$.
   By \eqref{2x2sub}, $t\le 2$. Hence by \eqref{tau7},
    \[\alpha + 0.99 t < 3 <3.325=0.475 \times 7 \leq 0.475\tau,\]
 so
\eqref{conje2} is satisfied in this case, and the desired bound holds.

Suppose condition (ii) of either Corollary \ref{pnmidac} or Corollary \ref{pnmidb}
  fails.
Then $\tau - 2 \le t \le \frac{\tau}2$, which means that $\tau \le 4$, contradicting
\eqref{tau7}.

This proves \eqref{n2bound'} and hence Proposition \ref{conjn2}.

\appendix
\section{Discrete series matrix coefficients for $\GSp(2n)$}
Here we explicitly compute certain discrete series matrix coefficients
  for $\GSp_{2n}(\R)$ using ideas of Harish Chandra.
  Our main references for the background material
  are \cite{AS} and \cite{Kn}.

\subsection{Root System}

The Lie algebra of $\Sp_{2n}(\R)$ is
\begin{align*}
\lie&=\{X|\, JX+{}^tXJ=0\}\\
&=
  \left\{\mat{A}{B}{C}{D}\in M_{2n}(\R)|\, A=-\t D,\, B=\t B,\, C=\t C\right\}.
\end{align*}
We have
\[\lie=\klie\oplus \plie, \]
where
\[\klie=\left\{\mat{A}{B}{-B}{A}\in M_{2n}(\R)|\, A=-\t A, B=\t B\right\}\]
is the Lie algebra of $K$,\footnote[2]{In this appendix, $K$ denotes the
  compact subgroup \eqref{K} which was denoted $K_\infty$ elsewhere in the paper.}
\[\plie=\left\{\mat{A}{B}{B}{-A}\in M_{2n}(\R)|\, A=\t A, B=\t B\right\}.\]
Let $\h\subset \klie_\C$ denote the real subspace consisting of all matrices of the form
\begin{equation}\label{h}
 \mat{} {\begin{smallmatrix}a_1\\&a_2\\&&\ddots\\&&&a_n\end{smallmatrix}}
{\begin{smallmatrix}-a_1\\&-a_2\\&&\ddots\\&&&-a_n\end{smallmatrix}}{},\end{equation}
for $a_j\in i\R$ ($i^2=-1$).
  Then $\h$ is a compact Cartan subalgebra of $\lie_\C$.
For each $j=1,\ldots,n$, define a linear form $e_j$ on $\h_\C$ by
 taking the above matrix to $ia_j$.
Let $\Delta=\Delta(\h_\C,\lie_\C)\subset \h_\C^*$ be the set of roots.  One finds that
\[\Delta=\{\pm 2e_j, \pm(e_j+e_k), e_j-e_k|\,j\neq k\}.\]
The compact roots (i.e. those whose root spaces belong to $\klie_\C$) are
\[\Delta_K=\{e_j-e_k|\,j\neq k\}.\]
Let $\Delta_{nc}=\Delta-\Delta_K$ denote the set of noncompact roots.
We fix the inner product in $\Delta$ determined by
\[\sg{e_i,e_j}=\delta_{ij}.\]

We set $\e_j=-e_j$, and fix the following ordered basis for $\h_\C^*$:
\[\{\e_1+\e_2+\cdots+\e_n,\,\, \e_n-\e_{n-1},\,\,
 \e_{n-1}-\e_{n-2},\,\, \ldots,\,\, \e_2-\e_1\}.\]
This determines a good ordering of $\Delta$ (i.e. $\Delta_{nc}^+>\Delta_K^+$)
  in which $\Delta^+=\Delta_{nc}^+\cup\Delta_{K}^+$
  is given by
\[
\left.\begin{array}{cl}2\e_j,&1\le j\le n\\
  \e_j+\e_k,& 1\le j<k\le n\end{array}\right\} \Delta_{nc}^+\]
\[\left.\begin{array}{cl} \e_k-\e_j,&1\le j<k\le n\end{array}\right\}\Delta_K^+.\]
Then
\begin{equation}\label{deltaG}
\delta_G=\frac12\sum_{\alpha\in\Delta^+}\alpha=\e_1+2\e_2+\ldots+n\e_n.
\end{equation}

\noindent{\em Remark:} The positive system (here denoted ${\Delta'}^{+}$) considered in
  \cite{AS} is the one used in Chapter IX of \cite{Kn}. The relationship
  to the present system is
  ${\Delta_{K}'}\!^{+}=\Delta_K^+$ and ${\Delta_{nc}'}\!^{+}=-\Delta^+_{nc}$.
In other words, $\Delta^+=-w_{K}({\Delta'}^+)$, where $w_{K}$ is the element of
  the Weyl group $W_{K}$ taking ${\Delta}_K^+$ to $-{\Delta}_K^+$.
  (See Remark (1) after Theorem 9.20 in \cite{Kn}.)
\\

For each root $\alpha\in\Delta$, let $\lie_\alpha\subset \lie_\C$ be the associated
  root space. Define the following subalgebras of $\lie_\C$:
\[\plie^+=\bigoplus_{\alpha\in \Delta_{nc}^+} \lie_\alpha
=\left\{\mat{A}{-iA}{-iA}{-A}|\, A=\t A\in M_n(\C)\right\},\]
and
\[\plie^-=\bigoplus_{\alpha\in \Delta_{nc}^+} \lie_{-\alpha}
=\left\{\mat{A}{iA}{iA}{-A}|\, A=\t A\in M_n(\C)\right\}.\]
The corresponding analytic subgroups of $\Sp_{2n}(\C)$ are
\[{P}^+=\exp(\plie^+)=\left\{\mat{I_n+A}{-iA}{-iA}{I_n-A}|\, A=\t A\in M_n(\C)\right\}\]
and
\[{P}^-=\exp(\plie^-)=\left\{\mat{I_n+A}{iA}{iA}{I_n-A}|\, A=\t A\in M_n(\C)\right\}.\]
We also have
\[K_\C=\exp(\klie_\C)=
\left\{\mat{A}{B}{-B}{A}\in \GL_{2n}(\C)|\, (A+iB)\,\t(A-iB)=I_n\right\}.\]
Since $A$ and $B$ are complex, the condition on $(A+iB)$ is
  equivalent to $A+iB\in \GL_n(\C)$, reflecting the fact that $\U(n,\C)\cong
  \GL_n(\C)$.

\subsection{Realization in $\SU(n,n)$}
  Recall that
\[\SU(n,n)=\left\{g\in \SL_{2n}(\C)| \t\, \ol{g}\mat{I_n}{}{}{-I_n}g
  =\mat{I_n}{}{}{-I_n}\right\}.\]
Define
\[G' = \{g \in \SU(n,n)| \t gJg = J\}.\]
One can show that
\[
G'=\left\{\mat{\alpha}{\beta}{\ol{\beta}}{\ol{\alpha}}\in \SL_{2n}(\C)\left|
\begin{array}{c} \t\,\ol{\alpha}\alpha - \t \beta \ol{\beta}=I_n\\
  \t\,\ol{\beta}\alpha=\t\alpha\ol{\beta}\end{array}\right.\right\}.
\]
Let $\tau=\mat{I_n}{iI_n}{iI_n}{I_n}.$  Then the map
\[g\mapsto g'=\tau^{-1}g\tau\]
is an isomorphism from $\Sp_{2n}(\R)$ into $G'$.
For any object $O$ associated to $\Sp_{2n}(\R)$, we let $O'$ denote the corresponding
  object for $G'$.  Writing $g=\mat{A}{B}{C}{D}$,
  we have
\begin{equation}\label{g'}
g'=\frac{1}{2}\mat{(A+D)+i(B-C)}{(B+C)+i(A-D)}{(B+C)-i(A-D)}{(A+D)-i(B-C)}.
\end{equation}
Taking $g=\mat{A}{B}{-B}{A}\in K$, we see that $g'=\mat{A+Bi}{}{}{A-Bi}$,
  where $A+Bi$ is unitary.  Thus
\[K'=\tau^{-1}K\tau
=\left\{\mat{\alpha}{}{}{\t{\alpha}^{-1}}|\,\alpha\in \U(n)\right\}.\]
Likewise,
\[K'_\C=\tau^{-1}K_\C\tau=\left\{\mat{\alpha}{}{}{\t\alpha^{-1}}|\,\alpha\in \GL_n(\C)\right\}
\]
\[P'^+=\tau^{-1}P^+\tau=\left\{\mat{I_n}{O}{-2iA}{I_n}|\,A=\t A\in M_n(\C)\right\}
\]
and
\[P'^-=\tau^{-1}P^-\tau=\left\{\mat{I_n}{2iA}{O}{I_n}|\,A=\t A\in M_n(\C) \right\}.
\]

\subsection{Holomorphic discrete series}

We recall without proof some properties of the holomorphic discrete series for the group
  $G=\Sp_{2n}(\R)$.  This material is due to Harish-Chandra \cite{HC}.
  We follow the exposition in Chapter VI of \cite{Kn}.
  Suppose $\lambda\in \h_\C^*$ is analytically
  integral (i.e. $\lambda(H)\in 2\pi i\Z$ whenever $\exp(H)=1$) and dominant
  with respect to $K$ (i.e. $\sg{\lambda,\alpha}>0$ for all $\alpha\in\Delta_K^+$).
  Let $(\Phi_\lambda,V)$ denote the irreducible unitary representation of
  $K$ with highest weight $\lambda$.
  Let $v_\lambda\in V$ be a highest weight unit vector (unique up to unitary scaling).

  Extend $\Phi_\lambda$ to a holomorphic representation of $K_\C$.
  For $g\in G$, let $\mu(g)\in K_\C$ denote the middle component in the
  Harish-Chandra decomposition
\[G\subset P^+K_\C P^-\]
  (\cite[Theorem 6.3]{Kn}).  Define
\[\psi_\lambda(g)=\sg{\Phi_\lambda(\mu(g))^{-1}v_\lambda,v_\lambda}_V.\]
 Then under the condition
\begin{equation}\label{cond}
  \sg{\lambda+\delta_G,\alpha}<0\quad\text{ for all }\quad\alpha\in \Delta_{nc}^+,
\end{equation}
  $\psi_\lambda$ is a nontrivial square-integrable function on $G$ (\cite[Lemma 6.9]{Kn}),
   and its translates under the left regular representation
  generate an irreducible square-integrable representation $(\pi_\lambda,V_\lambda)$ of $G$
  (\cite[Theorem 6.6]{Kn}).
  Furthermore, by \cite[p. 160 (6)]{Kn},
\[
\sg{\pi_\lambda(g)\psi_\lambda,\psi_\lambda}_{L^2}=\psi_\lambda(g^{-1})\|\psi_\lambda\|^2.
\]

  To each $v\in V$ we associate the function
 $\sg{\Phi_\lambda(\mu(g))^{-1}v,v_\lambda}\in V_\lambda$.
 This defines a $K$-equivariant embedding $V\rightarrow V_\lambda$, and hence $\Phi_\lambda$
  occurs as a $K$-type in $\pi_\lambda$, with highest weight vector $\psi_\lambda$.
  In fact, this $K$-type occurs with multiplicity one and
  $w_\lambda=\frac{\psi_\lambda}{\|\psi_\lambda\|}\in V_\lambda$ is a highest
  weight unit vector (\cite[p. 160 (5)]{Kn}). Our aim is to compute the matrix coefficient
\begin{equation}\label{mc}
\sg{\pi_\lambda(g) w_\lambda, w_\lambda}=\psi_\lambda(g^{-1})
=\sg{\Phi_\lambda(\mu(g^{-1})^{-1})v_\lambda,v_\lambda}_V.
\end{equation}

We will use the realization of $G$ in $\SU(n,n)$ since it facilitates the computation
  of $\mu(g)$.
Let $g'=\mat{\alpha}{\beta}{\ol{\beta}}{\ol{\alpha}}\in G'.$  Then
  the Harish-Chandra decomposition is given explicitly by
\begin{equation}\label{HCD}
\mat{\alpha}{\beta}{\ol{\beta}}{\ol{\alpha}}=\mat{I_n}{O}{\ol{\beta}\alpha^{-1}}{I_n}
\mat{\alpha}{O}{O}{\t {\alpha}^{-1}}
  \mat{I_n}{{\alpha}^{-1}{\beta}}{O}{I_n}.
\end{equation}
To verify this decomposition, note that the lower right corner on the right-hand side
  is $\t \alpha^{-1} +\ol{\beta}\alpha^{-1}\beta$.
By the fact that $\t\alpha \ol{\alpha}=I_n+\ol{\t\beta} {\beta}$,
\[\ol{\alpha}=\t\alpha^{-1} +\t \alpha^{-1}\ol{\t\beta}\beta.\]
Also $\ol{\t\beta}\alpha=\t\alpha\ol{\beta}\implies
  \t\alpha^{-1}=\ol{\beta}\alpha^{-1} \ol{\t\beta}^{-1}$.  Substituting this into
  the second term above, we see that the lower right-hand corner is equal to $\ol{\alpha}$
  as needed.

Let $g=\mat{A}{B}{C}{D}\in \Sp_{2n}(\R)$,
  and let $g'=\mat{\alpha}{\beta}{\ol{\beta}}{\ol{\alpha}}\in G'$ as in \eqref{g'}.
By \eqref{HCD}, we see that
$\mu(g') =\mat{\alpha}{O}{O}{\t{\alpha}^{-1}}$.
  Now $g'^{-1}=\mat{\t\,\ol{\alpha}}{-\t \beta}{-\ol{\t\beta}}{\t \alpha}$, so
\begin{equation}\label{mu}
\mu(g'^{-1})=\mat{\t\, \ol{\alpha}}{O}{O}{\ol{\alpha}^{-1}}.
\end{equation}

\subsection{Explicit formula for the matrix coefficient}\label{Eval}

Although it is possible to compute the matrix coefficient of $\pi_\lambda$ explicitly
  for any $\lambda$ (at least when $n=2$),
  for simplicity, we have chosen here to treat just the case where
  $\dim\Phi_\lambda=1$.  This discussion is valid for any $n\ge 2$.

For $\k\in\Z$, let $\Phi_\k$ be the character of $K=\U(n)$ defined by
\begin{equation}\label{Phi}
\Phi_\k(\mat{A}{B}{-B}{A})=\det(A+Bi)^{\k}.
\end{equation}
Its holomorphic extension to $\U_n(\C)=\GL_n(\C)$ is given by the same
  formula.  The (unique) weight of this character is
\begin{equation}\label{lambdak}
\lambda_\k=-\k(\e_1+\cdots+\e_n).
\end{equation}
This weight satisfies condition \eqref{cond} exactly when $\k>n$.

\begin{proposition}\label{mcSp} The character $\Phi_\k=\det^\k$ arises
  as the minimal $K$-type of a holomorphic discrete series representation
  $\pi_\k^+$ of $\Sp_{2n}(\R)$ if and only if $\k>n$.  If this holds and $w_\k$ is
  a unit vector of weight $\lambda_\k$,
  then for any $g=\smat{A}{B}{C}{D}\in \Sp_{2n}(\R)$ we have
\[ \sg{\pi^+_\k(g)w_\k,w_\k} =
    \frac{2^{n\k}}{\det(A+D+i(-B+C))^{\k}}.
\]
\end{proposition}

\begin{proof}
By \eqref{mc} and \eqref{mu},
\[ \sg{\pi^+_\k(g)w_\k,w_\k} =\Phi_\k(\mu(g^{-1}))^{-1}
=\det(\t\,\ol{\alpha})^{-\k}.\]
By \eqref{g'}, if $g=\smat ABCD$,
\begin{equation}\label{alpha}
\alpha=\frac{1}{2}((A+D)+i(B-C)).
\end{equation}
The given formula now follows.
\end{proof}

\begin{corollary}
For $g=\smat ABCD\in \Sp_{2n}(\R)$, define $f_\k^+(g)=\ol{\left< \pi_\k(g) w_\k, w_\k \right>}$.
Then
\begin{equation}\label{mcmain}
   f_\k^+(g) = \frac{2^{n\k}}{\det(A+D+i(B-C))^{\k}}.
\end{equation}
\end{corollary}

\begin{corollary}\label{f+abs}
    For $\k>n$ and $g=\smat ABCD\in \Sp_{2n}(\R)$,
   $|f_\k^+(g)|$ equals
    \[\frac{2^{n\k}}{\det(2I_n + A\t A+B\t B+C\t\, C+D\t D
  +i(A\t\, C-C\t A+B\t D-D\t B))^{\k/2}}.
\]
\end{corollary}
\begin{proof}
  For $\alpha=A+D+i(B-C)$,
    $|\det\alpha|^2 = \det \alpha \det \t\,\ol{\alpha} = \det \alpha \t\,\ol{\alpha}$.
    Expand this using
the relations given in \eqref{ac} and \eqref{ac2}.
  The corollary then follows immediately.
\end{proof}

With the formula for the matrix coefficient in hand, we can compute its $L^p$-norms.
The formulas above and the calculations below closely parallel
  those for $\GL_2(\R)$ given in \cite[\S14]{KL}.

\begin{proposition}\label{fdcalc} For any real number $\ell>0$,
  the function $|f_\k^+|^\ell$ is integrable over $G=\Sp_{2n}(\R)$ if and only if
  $\ell \k>2n$. If this condition holds, then with Haar measure normalized
  as in the proof below,
\[\int_G|f_\k^+(g)|^{\ell}dg = \frac {2^{n(n+1)}\prod_{j=1}^{n-1}j!}
  {\prod_{1\le i\le j\le n}(\ell\k - (i+j))}.\]
\end{proposition}

\begin{proof}
 The matrix coefficient $f_\k^+$ is bi-$K$-invariant, so it is convenient
 to use the Cartan decomposition $G=KA^+K$, where
\[A^+=\{a=\diag(a_1,\ldots,a_n,
  a_1^{-1},\ldots,a_n^{-1})|\, 1<a_1<a_2<\cdots<a_n\}.\]
  We may view $\Delta^+$ as the set of positive roots relative to the action of
  the diagonal subgroup on $\lie$, and $A^+=\exp(\mathfrak{a}^+)$, where
  $\mathfrak{a}^+$ is the positive Weyl chamber.
  By a standard integration formula
   (\cite[Lemma 4.2]{vdB}), when $dg$ is suitably normalized we have
\[\int_{G}|f_\k^+(g)|^\ell dg
  =\int_{K\times A^+\times K}|f_\k^+(k_1 a k_2)|^\ell
  \prod_{\alpha\in \Delta^+}(a^\alpha-a^{-\alpha})\, dk_1da\,dk_2,\]
where the Haar measure of the compact group $K$ is taken to be $1$.
Note that $(a^\alpha-a^{-\alpha})>0$ by the definition of $A^+$.
We now change notation and write $a=\diag(a_1,\ldots,a_n)$.  Using
 Corollary \ref{f+abs}, the above is
\[=2^{n\ell\k}\int_{a\in \GL_n(\R)\text{ diagonal},\atop{1<a_1<\cdots<a_n}}
 \det(2I_n+a^2+a^{-2})^{-\ell\k/2}\prod_{\alpha\in \Delta^+}(a^\alpha-a^{-\alpha}) da\]
\begin{align*}
    =2^{n\ell\k}\int_1^\infty\int_{a_1}^\infty\cdots&\int_{a_{n-1}}^\infty
  \prod_{d=1}^n(2+a_d^2+a_d^{-2})^{-\ell\k/2}(a_d^2-a_d^{-2})
 \\
&\times \prod_{1\le i< j\le n}(a_ia_j-a_i^{-1}a_j^{-1})
(a_ja_i^{-1}-a_j^{-1}a_i)
  \frac{da_n}{a_n}\cdots \frac{da_1}{a_1}.
\end{align*}
To ease notation below, set $\kappa=\ell\k$.  Then letting $u_j=2+a_j^2+a_j^{-2}$,
  the above is
\begin{equation}\label{int1}
=2^{n\kappa-n}\int_4^\infty\int_{u_1}^\infty\cdots\int_{u_{n-1}}^\infty
  \prod_{d=1}^n u_d^{-\kappa/2}
  \prod_{1\le i<j\le n}(u_j-u_i)\, du_n\cdots du_1.
\end{equation}

This can be evaluated using the Selberg integral, but we have chosen to give
  a self-contained treatment since it does not take much more space to do so.
Observe that
\[\prod_{i<j}(u_j-u_i)
  =\sum_{\sigma\in S_n}\sgn(\sigma)u_1^{\sigma(1)-1}\cdots u_n^{\sigma(n)-1}.\]
Hence \eqref{int1} becomes
\[\frac{2^{n\kappa}}{2^n}\sum_{\sigma\in S_n}\sgn(\sigma)
  \int_4^\infty
 \int_{u_1}^\infty\int_{u_2}^\infty\cdots\int_{u_{n-1}}^\infty
  u_1^{\sigma(1)-1-\kappa/2}\cdots u_n^{\sigma(n)-1-\kappa/2}du_n\cdots du_1\]
\[=\frac{2^{n\kappa}}{2^n}\sum_{\sigma\in S_n}
\frac{\sgn(\sigma)(-1)^n4^{\sigma(1)+\sigma(2)+\cdots+\sigma(n)-n\kappa/2}}
{\prod_{j=1}^{n}\bigl(\sigma(n)+\sigma(n-1)+\cdots+\sigma(n-j+1)-j\kappa/2\bigr)}
\]
\[=2^{n^2}\sum_{\sigma\in S_n}
\frac{\sgn(\sigma)}
{\bigl(\kappa/2-\sigma(n)\bigr)\bigl(\kappa-\sigma(n)-\sigma(n-1)\bigr)\cdots
  \bigl(n\kappa/2-\sigma(n)-\cdots-\sigma(1)\bigr)},
\]
provided $\kappa>2n$ (otherwise the integral diverges).
Replace $\sigma$ by $\sigma\tau$, where $\tau(i)=n-i+1$ for all $1\le i\le n$.
    The above is then
\[=2^{n^2}\sum_{\sigma\in S_n}
\frac{\sgn(\sigma\tau)}
{\bigl(\kappa/2-\sigma(1)\bigr)\bigl(\kappa-\sigma(1)-\sigma(2)\bigr)\cdots
  \bigl(n\kappa/2-\sigma(1)-\cdots-\sigma(n)\bigr)}.
\]
Applying Lemma \ref{Sn} below with $b_i=\kappa/2-i$, the above is
\[=2^{n^2+n}\frac{\sgn(\tau)\prod_{i<j}(i-j)}{\prod_{i\le j}(\kappa-(i+j))}
=2^{n(n+1)}\frac{\prod_{i<j}(j-i)}
  {\prod_{i\le j} (\kappa-(i+j))},\]
by the fact that $\sgn(\tau)=(-1)^{n\choose 2}$
  (both are $1$ iff $n\equiv 0,1\mod 4$).
The proposition now follows.
\end{proof}

\begin{lemma}\label{Sn}
Let $F$ be a field of characteristic $0$, and let $b_1,\ldots,b_n$ be
  indeterminate variables.
Then in the field $F(b_1,\ldots,b_n)$ of rational functions,
\begin{equation}\label{sigmasum}
\sum_{\sigma\in S_n}\frac{\sgn(\sigma)}{b_{\sigma(1)}(b_{\sigma(1)}+b_{\sigma(2)})
\cdots(b_{\sigma(1)}+\cdots+b_{\sigma(n)})}=
2^n\frac{\prod_{i<j} (b_j-b_i)}{\prod_{i\le j}(b_i+b_j)}.
\end{equation}
\end{lemma}

\begin{proof}
Let $A(b_1, \ldots, b_n)$ denote the left-hand side of \eqref{sigmasum}.
Then
\begin{equation}\label{Aclaim}
 A(b_1, \ldots, b_n) = \sum_{\ell=1}^n  (-1)^{n-\ell}
\frac{A(b_1, \ldots, b_{\ell-1}, b_{\ell+1}, \ldots, b_n)}{b_1 + \cdots + b_n}.
\end{equation}
    To see this, write $A(b_1,\ldots,b_n)$ as
\[\frac1{b_1+\cdots+b_n}\sum_{\ell=1}^n \sum_{\sigma\in S_n,\atop{\sigma(n)=\ell}}
  \frac{\sgn(\sigma)}{b_{\sigma(1)}(b_{\sigma(1)}+b_{\sigma(2)})
\cdots(b_{\sigma(1)}+\cdots+b_{\sigma(n-1)})}.
\]
Given $\sigma\in S_n$ with $\sigma(n)=\ell$, define $\sigma'\in S_{n-1}\subset S_n$ by
\[\sigma'(i)=\begin{cases} \sigma(i)&\text{if }1\le \sigma(i)<\ell\\
\sigma(i)-1&\text{if }\ell+1\le \sigma(i)\le n.\end{cases}\]
Then $\sigma$ is the composition of $\sigma'$ with $n-\ell$ transpositions, so
$\sgn(\sigma)=(-1)^{n-\ell}\sgn(\sigma')$.
Since $\sigma\mapsto\sigma'$ defines a bijection between the set of
  such $\sigma$ and $S_{n-1}$,
\eqref{Aclaim} follows.

We may now prove \eqref{sigmasum} by induction on $n$. The base case $n=2$
  is easy to check by hand.
Applying the inductive hypothesis to \eqref{Aclaim},
\[
 A(b_1, \ldots, b_n)  =
\frac{2^{n-1}}{b_1 + \cdots + b_n} \sum_{\ell=1}^n (-1)^{n-\ell}
\frac{\prod_{i<j;  i, j \neq \ell} (b_j-b_i)}{\prod_{i\le j; i,j \neq \ell}(b_i+b_j)}.
\]
Let
\begin{equation} \label{B1}
B(b_1, \ldots, b_n) = \frac{b_1+\cdots+b_n}{2^{n-1}} \Biggl( \prod_{1\le i\le j\le n} (b_i+b_j)  \Biggr)
 A(b_1, \ldots, b_n)
\end{equation}
\begin{equation} \label{B2}
=  \sum_{\ell=1}^n (-1)^{n-\ell} \Biggl( \prod_{i < j; i, j \neq \ell} (b_j-b_i) \Biggr) \left( \prod_{i=1}^n (b_i+b_\ell) \right).
\end{equation}
This is a homogeneous polynomial of degree $\frac{(n-1)(n-2)}2 + n = \frac{n(n-1)}2+1$.
Because
\[ A(b_1, \ldots, b_n) = \sgn(\sigma) A(b_{\sigma(1)}, \ldots, b_{\sigma(n)}) \]
for all permutations $\sigma\in S_n$, $B$ inherits this property from \eqref{B1}.
 In particular,
\begin{equation} \label{B3}
 B(b_1, \ldots, b_n) = -B(b_{\sigma(1)}, \ldots, b_{\sigma(n)})
\end{equation}
if $\sigma=(i\,\,j)$ is any $2$-cycle.  It follows that
$B(b_1, \ldots, b_n)=0$ if $b_i=b_j$ for any $i\neq j$.
Hence $\prod_{i < j} (b_i-b_j)$ divides $B(b_1,\ldots,b_n)$, and
\[\frac{B(b_1, \ldots, b_n)}{\prod_{i < j} (b_j-b_i)}\]
is a homogeneous symmetric  polynomial of degree $\frac{n(n-1)}2+1-\frac{n(n-1)}2=1$.
Hence it has the form $c(b_1+\cdots + b_n)$ for some constant $c$.
The monomial $b_n^n \prod_{i=2}^{n-1} b_i^{i-1}$ appears in \eqref{B2} with
  coefficient $2$ (take $\ell=n$), and in $(b_1+\cdots+b_n) \prod_{i<j}(b_j-b_i)$ with coefficient $1$.
  Therefore $c=2$, and \eqref{sigmasum} follows.
\end{proof}

\subsection{Integrability}

An irreducible unitary representation $\pi$ of $G=\Sp_{2n}(\R)$ is said to be
  {\em integrable} if it has a nonzero matrix coefficient belonging to $L^1(G)$,
or equivalently, if {\em all} of its $K$-finite matrix coefficients
  belong to $L^1(G)$.
  Applying Proposition \ref{fdcalc} with
  $\ell=1$, we immediately obtain the following.

\begin{proposition}\label{integrable} The representation $\pi_\k^+$ is integrable
  if and only if $\k>2n$.
\end{proposition}

More generally,
let $\pi_\lambda$ be the discrete series representation of $G$ (holomorphic or not)
  with Harish-Chandra parameter $\lambda$. Then by a
  theorem due to Trombi, Varadarajan, Hecht and Schmid,
  $\pi_\lambda$ is integrable if and only if
\begin{equation}\label{piint}
|\sg{\lambda,\beta}|>\tfrac12\sum_{\alpha\in\Delta^+}|\sg{\alpha,\beta}|\quad
\text{for all}\quad \beta\in\Delta_{nc}
\end{equation}
(\cite{TV}, \cite{HS}; see also Mili\v ci\'c \cite{Mi}).
With notation as in \eqref{deltaG} and \eqref{lambdak}, the Harish-Chandra parameter of
  $\pi_\k^+$ is
\begin{equation}\label{hcp}
\lambda=\lambda_\k +\delta_G=(1-\k)\e_1+(2-\k)\e_2+\cdots+(n-\k)\e_n
\end{equation}
 (see Remark (1) after Theorem 9.20 of \cite{Kn}).
Note that for any $j,\ell$,
\begin{equation}\label{lambdaip}
|\sg{\lambda_\k+\delta_G,\e_j+\e_\ell}|=|(j+\ell)-2\k|=2\k-(j+\ell).
\end{equation}
Using this, one may easily verify that \eqref{piint} holds for $\lambda$
  exactly when $\k>2n$, for an alternative proof of Proposition \ref{integrable}.

\subsection{Formal Degree}\label{fd}

Recall that the formal degree of $\pi=\pi_\lambda$
is the constant $d_\pi>0$
  (depending only on the choice of Haar measure on $G=\Sp_{2n}(\R)$) satisfying
\[\int_{G}|\sg{\pi_\lambda(g)v,w}|^2dg=\frac{\|v\|^2\|w\|^2}{d_\lambda}\]
for all $v,w\in V_\pi$.
Applying Proposition \ref{fdcalc} with $\ell=2$, we immediately find
  the following.

\begin{proposition} \label{dkprop}
The formal degree of $\pi_\k^+$ is the following polynomial in $\k$
  of degree $n+{n\choose 2}=\frac{n(n+1)}2$:
\begin{equation}\label{dk}
d_\k = a\prod_{1\le i\le j\le n}(2\k-(i+j)),
\end{equation}
where $a$ is a nonzero constant depending on $dg$.
\end{proposition}
\noindent{\em Remarks:}
(1) Harish-Chandra proved that there exists a choice of
  Haar measure for which
\[d_\lambda = \prod_{\beta\in \Delta^+} \left|\frac{\sg{\lambda+\delta_G,\beta}}
  {\sg{\delta_G,\beta}}\right|\]
for all $\lambda$ (\cite{HC} \S10).
If $\lambda$ is given by \eqref{hcp}, then by evaluating the above
  expression explicitly as in \eqref{lambdaip}, we obtain an alternative proof
  of \eqref{dk}.\\
 (2) With measure normalized
  as in the proof of Proposition \ref{fdcalc},
  $a=({2^{n(n+1)}\prod_{j=1}^{n-1}j!})^{-1}$.\\
(3) If we adopt the classical normalization of measure, so that
\[d_\k^{-1}=\int_{G}|f_\k^+(g)|^2dg=\int_{\mathcal{H}_n}
  \left|f_{\k}^+(\smat{I_n}X{}{I_n}\smat{Y^{1/2}}
{}{}{Y^{-1/2}})\right|^2 \frac{dXdY}{(\det Y)^{n+1}},\]
then $a=2^{-n(n+2)}\pi^{-n(n+1)/2}$ (\cite[\S A.1]{PSS}).

\subsection{Extension of $\pi^+_\k$ to $\GSp_{2n}$}\label{pik}

We can extend $\pi^+_\k$ to a representation of $\GSp_{2n}(\R)$ in the following way.
 First induce $\pi_\k^+$ to the group of symplectic similitudes with multiplier
  $\pm 1$, namely
 \[ \Sp^{\pm}_{2n}=\sg{\mat{I_n}{}{}{-I_n}}\Sp_{2n}.\]
  Let $V^+$ be the space of $\pi_\k^+$.  Then the new space is
\[V=\{f:\Sp^{\pm}_{2n}(\R)\rightarrow V^+|
\, f(gx)=\pi_\k^+(g)f(x)\text{ for all }g\in \Sp_{2n}(\R)\},\]
and $\Sp^{\pm}_{2n}(\R)$ acts by right translation.
  Note that any $f\in V$ is determined by $f(\smat{I_n}{}{}{I_n})$ and
  $f(\smat{I_n}{}{}{-I_n})$.  We identify $V^+$ with the subspace of $V^\pm$
  consisting of functions which vanish on $\smat{I_n}{}{}{-I_n}$.  Letting
  $V^-$ denote the space of functions vanishing on the identity element, we have
\begin{equation}\label{ds}
V=V^+\oplus V^-.
\end{equation}
  We make $V$ into a Hilbert space by defining
\begin{equation}\label{ip}
\sg{f,h}=
\sg{f(1), h(1)}_{V^+}+ \sg{f(\sigma), h(\sigma)}_{V^+},\end{equation}
where $\sigma=\mat{I_n}{}{}{-I_n}$.
Then \eqref{ds} is an orthogonal direct sum.

  Denote this representation on $V$ by $\pi_\k$.  One easily sees that each subspace
  in \eqref{ds} is stable
  under $\Sp_{2n}(\R)$, and $\pi_\k|_{\Sp_{2n}(\R)}=\pi_\k^+\oplus \pi_\k^-$, where $\pi_\k^-$
  is also irreducible and square integrable.
  Let
\[Z^+=\{\mat{zI_n}{}{}{zI_n}|\, z>0\}.\]
  Then
\[\GSp_{2n}=Z^+\times \Sp^{\pm}_{2n}.\]
  Extend $\pi_\k$ to a representation of $\GSp_{2n}(\R)$ by requiring $Z^+$ to act trivially.
  This is an irreducible square integrable representation, also denoted $\pi_\k$.
  For any $z\in Z(\R)$ and $f\in V$, we have
\[\pi_\k(z)f(g)=f(\sgn(z)g)=\pi_\k^+(\sgn(z))f(g)=\sgn(z)^{n\k}f(g)\]
  by \eqref{Phi}. This shows that the central character of $\pi_\k$ is
\begin{equation}\label{pikcc}
\chi_{\pi_\k}(z)=\sgn(z)^{n\k}.
\end{equation}

As before, let $w_\k\in V^+$ denote a unit vector of weight $\lambda_\k$.
Define $\phi_0\in V$ by
\[\phi_0(\mat{I_n}{}{}{I_n})=w_\k\quad\text{and}\quad
\phi_0(\mat{I_n}{}{}{-I_n})=0.\]
This is a lowest weight vector, spanning the minimal $K$-type of $\pi_\k$, which is
  the two-dimensional representation $\Ind_{K}^{K^{\pm}}(\Phi_\k)$, where
  $K^\pm =K\cup K\smat{I_n}{}{}{-I_n}$.

\begin{proposition} The representation $\pi_\k$ is irreducible, unitary,
  and square-integrable when $\k>n$.  It is integrable exactly when $\k>2n$,
  and in this case, the formal degree of $\pi_\k$ coincides with that
  of $\pi_\k^+$ given in \eqref{dk}.
\end{proposition}

\begin{proof}
Everything follows more or less immediately from the corresponding properties
  of $\pi_\k^+$. Indeed, define the matrix coefficient
\begin{equation}\label{fkdef}
f_\k(g)=\ol{\sg{\pi_\k(g)\phi_0,\phi_0}}.
\end{equation}
Unraveling the definitions, we see that $f_\k(g)=f_\k^+(g)$ for $g\in \Sp_{2n}(\R)$,
  and $f_\k(g)=0$ if $r(g)<0$.
Using the fact that $\pi_\k$ is $Z^+$-invariant, we have
\[\int_{\GSp_{2n}(\R)/Z}|f_\k(g)|^\ell dg=\int_{\Sp_{2n}^\pm}|f_\k(g)|^\ell dg
  =\int_{\Sp_{2n}}|f_\k(g)|^\ell dg=\int_{\Sp_{2n}}|f_\k^+(g)|^\ell dg.\]
The assertions now follow from Proposition \ref{fdcalc}.
\end{proof}

\begin{theorem}\label{mcthm}
  For $g=\mat{A}{B}{C}{D}\in \GSp_{2n}(\R)$,
\[ \sg{\pi_\k(g)\phi_0,\phi_0} =\begin{cases}
    \displaystyle\frac{r(g)^{\frac{n\k}{2}}2^{n\k}}
  {\det(A+D+i(-B+C))^{\k}} & \text{if }r(g)>0\\\\
  0&\text{if }r(g)<0.\end{cases}
\]
\end{theorem}
\begin{proof}
Let $r=r(g)$.  Suppose $r<0$.  Then $\pi_\k(g)\phi_0(x)\neq 0\iff r(x)<0$,
  i.e. $\pi_\k(g)\phi_0\in V^-$, so it is orthogonal to  $\phi_0\in V^+$.
  Thus $\sg{\pi_\k(g)\phi_0,\phi_0} =0$ in this case.

Now suppose $r>0$.  Let
\[h=r^{-1/2}g=\mat{r^{-1/2}A}{r^{-1/2}B}{r^{-1/2}C}{r^{-1/2}D}.\]
  It is easy to see that $r(h)=1$, i.e. that $h\in \Sp_{2n}(\R)$.
  By definition of $\pi_\k$, $\pi_\k(g)\phi_0=
  \pi_\k(h)\phi_0$.  Therefore by \eqref{ip},
\[ \sg{\pi_\k(g)\phi_0,\phi_0} =\sg{\pi_\k(h)\phi_0,\phi_0}
=\sg{\pi_\k^+(h)w_\k,w_\k}_{V^+}\]
\[=\frac{2^{n\k}}{r^{-n\k/2}\det(A+D+i(-B+C))^\k}\]
by Proposition \ref{mcSp}.
\end{proof}

\noindent The following corollaries are easily proven as in \S \ref{Eval}.

\begin{corollary} \label{matcoeffmain}
For $f_\k(g)=\ol{\sg{\pi_\k(g)\phi_0,\phi_0}}$, if $r(g)>0$, then
\[
   f_\k(g) = \frac{r(g)^{\frac{n\k}{2}}2^{n\k}}{\det(A+D+i(B-C))^{\k}}.
\]
If $r(g)<0$, then $f_\k(g)=0$.
\end{corollary}

\begin{corollary} \label{mcnorm}
    If $g=\smat ABCD\in \GSp_{2n}(\R)$ with $r(g)>0$, then $|f_\k(g)|$ equals
\[
    \frac{r(g)^{\frac{n\k}{2}}2^{n\k}}
  {{\det(2r(g)I_n + A\t A+B\t B+C\t\, C+D\t D+i(A\t\, C-C\t A+B\t D-D\t B))^{\k/2}}}.
\]
\end{corollary}

In the special case $n=2$, we can make the above more explicit and provide
  a convenient upper bound for the matrix coefficient.

\begin{proposition}\label{finf2}
 Suppose $n=2$, and $r(g)>0$.  Then
\[
|f_{\k}(g)|
= \frac{ r(g)^{\k}\,4^{\k}}
  {({4r(g)^2+2r(g) \sum_{i,j} g_{ij}^2 + \sum_{i=3}^8 X_i^2)^{\k/2}}},
\]
where $g_{ij}$ are the entries of $g$, and the $X_i$ are the bilinear forms
  in these entries defined in the proof below.  Consequently,
\begin{equation}\label{fkbound} |f_\k(g)| \leq \frac{ (8r(g))^{\k/2}}
  {({2r(g)+ \sum_{i,j} g_{ij}^2)^{\k/2}}}.
\end{equation}
\end{proposition}

\begin{proof}
Write $g=\smat ABCD$.  Let $A_{ij}$ denote the $(i,j)$-th entry of $A$, and likewise
  for the entries of $B,C,D$. Then
\begin{align}
\notag \det&\bigl(2r(g)I_2 + A\t A+B\t B+C\t\,C+D\t D+i(A\t\, C-C\t A+B \t D-D\t B)\bigr) \\
\label{deteq}
 =&  (2r(g) + A_{11}^2+A_{12}^2+B_{11}^2+B_{12}^2+B_{11}^2+C_{12}^2+D_{11}^2+D_{12}^2 )\\
\notag &\times (2r(g) + A_{21}^2+A_{22}^2+B_{21}^2+B_{22}^2+C_{21}^2+C_{22}^2+D_{21}^2+D_{22}^2) - X_1^2 - X_2^2,
\end{align}
where
\[ X_1 = A_{11} A_{21} + A_{12} A_{22} + B_{11} B_{21} + B_{12} B_{22} + C_{11} C_{21} + C_{12} C_{22} + D_{11} D_{21} + D_{12} D_{22},\]
\[ X_2 = A_{11} C_{21} + A_{12} C_{22} + B_{11} D_{21} + B_{12} D_{22} - C_{11} A_{21} - C_{12} A_{22} - D_{11} B_{21} - D_{12} B_{22}. \]
By Degen's eight-square identity,
    \begin{align*}
        (A_{11}^2+&A_{12}^2+B_{11}^2+B_{12}^2+C_{11}^2+C_{12}^2+D_{11}^2+D_{12}^2 )\\
        &\times
    (A_{21}^2+A_{22}^2+B_{21}^2+B_{22}^2+C_{21}^2+C_{22}^2+D_{21}^2+D_{22}^2)
    = \sum_{i=1}^8 X_i^2, \end{align*}
for $X_1,X_2$ as above, and
\[ X_3 = A_{11} A_{22} - A_{12} A_{21} + B_{11} D_{22} - B_{12} D_{21} - C_{11} C_{22} + C_{12} C_{21} - D_{12} B_{21} + D_{11} B_{22}, \]
\[ X_4 = A_{11} C_{22} - A_{12} C_{21} + B_{11} B_{22} - B_{12} B_{21} - C_{12} A_{21} + C_{11} A_{22} - D_{11} D_{22} + D_{12} D_{21}, \]
\[ X_5 = A_{11} B_{21} - A_{12} D_{22} - B_{11} A_{21} + B_{12} C_{22} - C_{11} D_{21} - C_{12} B_{22} + D_{12} A_{22} + D_{11} C_{21}, \]
\[ X_6 = A_{11} D_{21} - A_{12} B_{22} - B_{11} C_{21} + B_{12} A_{22} + C_{11} B_{21} + C_{12} D_{22} - D_{11} A_{21} - D_{12} C_{22}, \]
\[ X_7 = A_{11} D_{22} + A_{12} B_{21} - B_{11} A_{22} - B_{12} C_{21} + C_{11} B_{22} - C_{12} D_{21} + D_{11} C_{22} - D_{12} A_{21}, \]
\[ X_8 = A_{11} B_{22} + A_{12} D_{21} - B_{11} C_{22} - B_{12} A_{21} - C_{11} D_{22} + C_{12} B_{21} - D_{11} A_{22} + C_{21} D_{12}. \]
Therefore \eqref{deteq} equals
\[4r(g)^2+2r(g)\sum_{i,j}g_{ij}^2+\sum_{i=3}^8 X_i^2,\]
and the proposition follows from Corollary \ref{mcnorm}.
\end{proof}

\section{Off-diagonal terms}\label{AB}

In this appendix, we give a very rough estimate for the off-diagonal terms in the
  relative trace formula to obtain the following quantitative version of Theorem \ref{A}
  for a fixed level $N$, valid when $n=2$.

\begin{theorem}\label{Ebound} Suppose $n=2$, $\k \ge 17$, $\Sb$ is a finite
  set of primes, $r=\prod_{p\in\Sb}p^{r_p}$ is the similitude attached to
  the local test functions $f_p$ as in \eqref{r},
  and $N$ is a fixed level prime to $\Sb$.
Then with notation as in Theorem \ref{A},
\[  \frac1{\psi(N)}\sum_{\pi\in \Pi_\k(N)}\sum_{\varphi\in E_\k(\pi,N)}\frac{c_{\sigma_1}(\varphi)
  \ol{c_{\sigma_2}(\varphi)}} {\|\varphi\|^2}
\prod_{p\in\Sb}(\mathcal{S}f_p)(t_{\pi_p})\]
\[=
 \sum_{A} \int_{S_2(\A)} f_1(\mat{I_2}SO{I_2} \mat AOO{r\t A^{-1}})
  \theta(\tr \sigma_1 S) dS+
O\left(\frac{\k^{21/2} (8r)^{\k/2}}{N^{\k-12}}\right),
  \]
for an absolute implied constant.
\end{theorem}
\noindent{\em Remarks:} (1) The sum over $A$ is dependent on $\k$, but not on $N$.

(2) With extra work, one can
increase the power of $N$ in the error term, and
  decrease the lower bound on $\k$.

  (3) As noted, this is a very rough estimate, which is meaningful only when $\k$ is fixed.
  Indeed in the $\k$-aspect, the ``error term"
  is actually larger than the ``main term" when $\k$ is large, as the latter contains the
  factor $c_{n\k\sigma_1}$
  of Theorem \ref{S=1} with very rapid decay in $\k$.
  \\

Before specializing to the case $n=2$, we give the
  general form of the Fourier trace formula
  for a fixed level $N$ and any $n$.
With notation as in \S \ref{test}-\ref{asymp}, define
\[I_N =
\frac1{\psi(N)}\iint\limits_{\scriptstyle (N(\Q)\bs N(\A))^2}
  {K_{f_N}(n_1,n_2)}
\ol{\theta_1(n_1)}\theta_2(n_2) dn_1 dn_2.
\]
By using the spectral and geometric forms of the kernel, we have
\begin{equation}\label{ptf}
  \frac1{\psi(N)}\sum_{\pi\in \Pi_\k(N)}\sum_{\varphi\in E_\k(\pi,N)}\frac{c_{\sigma_1}(\varphi)
  \ol{c_{\sigma_2}(\varphi)}} {\|\varphi\|^2}
\prod_{p\in\Sb}(\mathcal{S}f_p)(t_{\pi_p})=M(f)+E(f),
\end{equation}
where
\[M(f)
 =\sum_{A} \int_{S_n(\A)} f_1(\mat{I_n}SO{I_n} \mat AOO{r\t A^{-1}})
  \theta(\tr \sigma_1 S) dS\]
is the ``diagonal" term appearing in Theorem \ref{A}, and
\[E(f)=\frac1{\psi(N)}\iint_{(N(\Q)\bs N(\A))^2} \sum_{\g=\smat ABCD\in \olG(\Q),{C\neq O}}
  \hskip -.4cmf(n_1^{-1}\g n_2)
\ol{\theta_1(n_1)}\theta_2(n_2) dn_1 dn_2\]
is the off-diagonal contribution.  (One checks readily that the integrand is
  indeed invariant under $N(\Q)\times N(\Q)$.)

  It is possible to express $E(f)$ as a sum of
  explicit orbital integrals (cf. \cite[\S2]{pethil}), with
  the orbits determined via the Bruhat decomposition of $\olG(\Q)$.
  For our more modest goal of proving Theorem \ref{Ebound}, it suffices to
  show that
\begin{equation}\label{Ebound2}
E(f)\ll
\frac{\k^{21/2} (8r)^{\k/2}}{N^{\k-12}}
\end{equation}
when $n=2$.

\begin{lemma} \label{sumint}
For $\k \ge 2$, $\Delta>0$, and $a\in\R$,
\[
\sum_{n=-\infty}^\infty \frac{1}{((n+a)^2 + \Delta^2)^{\k/2}}
\le \frac{\k+\Delta}\Delta \int_\R \frac{dx}{(x^2+\Delta^2)^{\k/2}}.
\]
\end{lemma}

\begin{proof}
Let
\[ f(x) = \frac{1}{((x+a)^2+\Delta^2)^{\k/2}}. \]
Then $\ds f'(x) = - \frac{\k(x+a)}{((x+a)^2+\Delta^2)^{\k/2+1}}$, so that
\[ |f'(x)| \le \k \frac{((x+a)^2 + \Delta^2)^{1/2}}{((x+a)^2+\Delta^2)^{\k/2+1}}
=\frac{\k}{((x+a)^2+\Delta^2)^{1/2}}f(x)
   \le \frac{\k}{\Delta} f(x).\]
By Euler's summation formula (\cite[p. 495]{MV}),
\[ \sum_{c < n \le d} f(n) = \int_c^d f(x) dx - f(d)\{d\}+f(c)\{c\} + \int_c^d \{x\} f'(x) dx, \]
where $\{x\}=x-[x]$ is the fractional part of $x$.
Hence
\[ \sum_{n=-\infty}^{\infty} f(n) \le \int_\R f(x) dx
+ \int_\R |f'(x)| dx \le (\frac{\k}{\Delta}+1) \int_{\R} f(x) dx \]
\[ = \frac{\k+\Delta}{\Delta} \int_\R \frac{dx}{((x+a)^2+\Delta^2)^{\k/2}}
= \frac{\k+\Delta}{\Delta} \int_\R  \frac{dx}{(x^2+\Delta^2)^{\k/2}}.\qedhere\]
\end{proof}

\begin{proof}[Proof of Theorem \ref{Ebound}]
As before, for any set $R$, we let $S_2(R)$ denote the $2\times 2$
  symmetric matrices over $R$.
For such a matrix $S$, let $n_S=\smat{I_2}S{O}{I_2}$.
Because $[0,1) \times \Zhat$ is a fundamental domain for $\Q \bs \A$,
\[E(f)= \frac1{\psi(N)}\iint_{S_2([0,1) \times \Zhat)^2}
\hskip -.2cm \sum_{\gamma =  \smat{A}{B}{C}{D} \in \mathbb{G}(\Q), C \neq O}
  \hskip -.7cm f(n_{S}^{-1} \gamma  n_{S'}) \theta(\tr \sigma_2 S' - \tr \sigma_1 S)\, dS dS'.\]
Assuming $n_{S_{\fin}}^{-1} \gamma n_{S'_{\fin}} \in \supp f_{\fin}$,
  by Proposition \ref{gamma} we may take
\[r(n_{S}^{-1} \gamma n_{S'})=r(\g) =\pm r,\quad
 (n_{S}^{-1} \gamma n_{S'})_{\fin} \in M_{4}(\Zhat),\quad
 C\in M_2(N\Z). \]
In fact, because $f_\infty$ is supported on matrices with positive similitude,
  we can take $r(\g)=r$.
By the fact that $S_{\fin}, S'_{\fin} \in {S}_2(\Zhat)$, it also follows that
$A, B, D \in M_2(\Z)$.

For a square matrix $g$ with real entries, define
\[ \Qf(g) = \sum_{i, j} g_{ij}^2\]
where $g_{ij}$ is the $(i,j)$-th entry of $g$. Then
since $|f(n_{1}^{-1}\g n_{2})|\le \psi(N)|f_\infty(n_{1,\infty}^{-1}\g n_{2,\infty})|$,
 it follows by \eqref{fkbound} that
\begin{equation}\label{Efbound}
|E(f)|\le \frac{1}{2}  d_{\k} (8r)^{\k/2} \iint\limits_{{S}_2([0,1))^2}
\sum_{A, B, D \in \mathbf{M}_2(\Z),\atop{O\neq C\in M_2(N\Z)}}
(2r + \Qf(n_{S}^{-1} \mat ABCD n_{S'}))^{-\k/2} dS dS'.
\end{equation}
(The factor $1/2$ accounts for the fact that we quotient
  by the center of $G(\Q)$.)
We first consider the sum over $B$:
\[
\sum_{B \in {M}_2(\Z)} \left(2r + \Qf(n_{S}^{-1} \mat ABCD n_{S'})\right)^{-\k/2}
\]
\[
=\sum_{B \in {M}_2(\Z)}
\left(2r+\Qf(\mat{A-SC}{B-SD+AS'-SCS'}{C}{CS'+D})\right)^{-\k/2}.
\]
By four applications of Lemma \ref{sumint}, the above is
\[
\ll \k^4 \int_{{M}_2(\R)}
\left(2r + \Qf(\mat{A-SC}{Y}{C}{CS'+D})\right)^{-\k/2} dY.
\]
Summing this over $A,D$, we find in the same way that
\[
    \sum_{A,B, D \in {M}_2(\Z)} \left(2r + \Qf(n_{S}^{-1} \mat ABCD n_{S'})\right)^{-\k/2}\]
\[\ll \k^{12} \int_{{M}_2(\R)^3}
\left(2r + \Qf(\mat{X}{Y}{C}{Z})\right)^{-\k/2} dXdYdZ.
\]
We now have an expression that is independent of $S,S'$, so using \eqref{dk},
  \eqref{Efbound} gives
\[
|E(f)|\ll \k^{15} (8r)^{\k/2}
\sum_{O\neq C\in M_2(N\Z)} \int_{{M}_2(\R)^3}
\left(2r + \Qf(\mat{X}{Y}{C}{Z})\right)^{-\k/2} dXdYdZ.\]
Because $C \neq O$, the above integral is bounded above by
\[
 \int_{{M}_2(\R)^3}
\Qf(\mat{X}{Y}{C}{Z})^{-\k/2} dXdYdZ\]
\[=\int_{M_2(\R)^3} (\Qf(X)+\Qf(Y)+\Qf(Z)+\Qf(C))^{-\k/2}dXdYdZ.
\]
Replacing $X, Y, Z$ with $\Qf(C)^{1/2}X$, $\Qf(C)^{1/2}Y$, $\Qf(C)^{1/2}Z$
  respectively, the above is
\[
=
\int_{\R^{12}} \frac{
\Qf(C)^{-\k/2+6}\,
    dX_{11} \cdots dX_{22} dY_{11} \cdots dY_{22} dZ_{11} \cdots dZ_{22}}
{(1+X_{11}^2 + \cdots + X_{22}^2 + Y_{11}^2 + \cdots + Y_{22}^2 + Z_{11}^2 + \cdots + Z_{22}^{2})^{\k/2}},
\]
which converges if $\k \ge 13$.  Indeed, for such $\k$, we find
  using spherical coordinates that the above is
\[
\ll \Qf(C)^{-\k/2+6}
  \int_0^\infty\frac{\rho^{11}}{(1+\rho^2)^{\k/2}}d\rho
\ll \k^{-6}(C_{11}^2 + C_{12}^2 + C_{21}^2+ C_{22}^2)^{-\k/2+6}
\]
since the integral over $\rho$ is equal to
 $\frac12B(6,\frac\k2-6)$ for the Beta function $B(x,y)$,  and by Stirling's formula,
  $B(6,\frac\k2-6)\sim \Gamma(6)(\frac{\k}2-6)^{-6}$ as $\k\to\infty$.

Finally, we need to sum over all $O\neq C\in M_2(N\Z)$.  Write $C_{ij} = NC_{ij}'$.
We may assume $C_{11} \neq 0$.
  (The other cases can be handled by the exactly the same method.)  Thus
\[
|E(f)|
\ll \frac{\k^9 (8r)^{\k/2}}{N^{\k-12}}
 \sum_{C'_{ij}\in \Z,\, C'_{11} \neq 0}
((C'_{11})^2 + (C'_{12})^2 + (C'_{21})^2+ (C'_{22})^2)^{-\k/2+6}.
\]
\[
\ll \frac{\k^{12} (8r)^{\k/2}}{N^{\k-12}}
 \sum_{C'_{11} \neq 0}
  \int_{\R^3} ((C'_{11})^2 + x^2 + y^2 + z^2)^{-\k/2+6} dx\, dy\, dz
\quad(\text{by Lemma \ref{sumint}})
\]
\[
=
 \frac{\k^{12} (8r)^{\k/2}}{N^{\k-12}}
\sum_{c \neq 0} c^{-\k+15} \int_{\R^3} (1+x^2 + y^2 + z^2)^{-\k/2+6} dx\, dy\, dz
\]
\[\ll \frac{\k^{12} (8r)^{\k/2}}{N^{\k-12}}
\sum_{c\neq 0}\frac{1}{c^{\k-15}} \int_0^{\infty}
\frac{\rho^2}{(1+\rho^2)^{\k/2-6}}d\rho.\]
This converges as long as $\k\ge 17$. The integral is equal to
  $\tfrac12B(\tfrac32,\tfrac{\k-15}2)\sim \tfrac12\Gamma(\tfrac 32)(\tfrac{\k-15}2)^{-3/2}\ll
  \k^{-3/2}$.  Hence for $\k\ge 17$, the above is
$\ll \frac{\k^{21/2} (8r)^{\k/2}}{N^{\k-12}}$
for an absolute implied constant.
This proves \eqref{Ebound2}, and Theorem \ref{Ebound} follows.
\end{proof}

\vskip .5cm
\Addresses
\end{document}